\documentclass[a4paper, 12 pt]{paper}
\usepackage{amsmath,amssymb,amscd,mathrsfs}
\usepackage{a4wide,latexsym,amsfonts,amssymb,exscale,enumerate,amscd}
\usepackage{amsthm}
\usepackage{color}
\usepackage{tikz-cd}
\usepackage{pgfplots}
\usepackage{verbatim}
\usepackage{graphicx}
\usepackage{pstricks}
\usepackage[tiling]{pst-fill}
\usepackage{hyperref}
\usetikzlibrary{arrows}
\usetikzlibrary{decorations.pathmorphing}
\usepackage{setspace}
\hypersetup{linkbordercolor=blue, hidelinks}
\usepackage[all]{xy}
\SelectTips{cm}{}
\usetikzlibrary{decorations.markings}
\usetikzlibrary{decorations.pathreplacing}
\usetikzlibrary{arrows,shapes,positioning}
\tikzstyle directed=[postaction={decorate,decoration={markings,
    mark=at position #1 with {\arrow{>}}}}]
\tikzstyle rdirected=[postaction={decorate,decoration={markings,
    mark=at position #1 with {\arrow{<}}}}]
\usepackage{graphicx}
\makeatletter

\@addtoreset{equation}{section}

\newtheorem{Proposition}{Proposition}[subsection] 
\newtheorem{Lemma}[Proposition]{Lemma}
\newtheorem{Theorem}[Proposition]{Theorem}
\newtheorem{Corollary}[Proposition]{Corollary}
\newtheorem{Definition}[Proposition]{Definition}
\newtheorem{Remark}[Proposition]{Remark}

\newtheorem{Example}[Proposition]{Example}

\newbox\squ  
\setbox\squ=\hbox{\vrule width.6pt
   \vbox{\hrule height.6pt width.4em\kern1ex
         \hrule height.6pt}%
   \vrule width.6pt}

\def\Tr{\operatorname{Tr}}
\def\tr{\operatorname{tr}}

\def\Id{\operatorname{Id}}
\def\C{{\mathbb C}}
\def\Q{{\mathbb Q}}
\def\Z{{\mathbb Z}}
\def\N{{\mathbb N}}

\def\0{{\bar 0}}
\def\1{{\bar 1}}

\def\End{{\operatorname{End}}}

\def\im{{\operatorname{im}}}

\def\phi{{\varphi}}
\def\emptyset{{\varnothing}}
\def\underbar{\mathpalette\@underbar}
\def\@underbar#1#2{\settowidth{\@tempdimb}{$#1#2$}\@tempdimb=0.8\@tempdimb
                   \ooalign{$#1#2$\crcr%
                         \hfil\rule[-.5mm]{\@tempdimb}{.4pt}\hfil}}

\newcommand{\maps}{\colon}

\newcommand\sE{{\mathcal{E}}}
\newcommand\sF{{\mathcal{F}}}
\newcommand{\onel}{{\mathbf 1}_{\bar{\nu}}}

\newcommand{\nub}{\overbar{\nu}}
\newcommand{\mub}{\overbar{\mu}}
\newcommand{\lb}{\overbar{\l}}

\newcommand{\onenn}[1]{{\mathbf 1}_{#1}}

\newcommand{\bfit}[1]{\textit{#1}}

\newcommand{\ra}{\rangle}
\newcommand{\la}{\langle}

\newcommand{\overbar}[1]{\mkern 1.5mu\overline{\mkern-1.5mu#1\mkern-1.5mu}\mkern 1.5mu}

\newcommand{\refequal}[1]{\xy {\ar@{=}^{#1}
(-1,0)*{};(1,0)*{}};
\endxy}

\newcommand{\U}{\dot{{\bf U}}}
\newcommand{\UA}{{_{\cal{A}}\dot{{\bf U}}}}

\newcommand{\Ucat}{\cal{U}}

\newcommand{\xsum}[2]{
  \vcenter{\xy
  (0,.4)*{\sum};
  (0,3.8)*{\scs #2};
  (0,-3.2)*{\scs #1};
  \endxy}
}

\usepackage{fancyheadings}
\newcommand{\cat}[1]{\ensuremath{\mbox{\bfseries {\upshape {#1}}}}}

\newcommand{\BOX}{\hbox {$\sqcap$ \kern -1em $\sqcup$}}

\renewcommand{\to}{\rightarrow}

\newcommand{\scs}{\scriptstyle}

\newcommand{\lineu}[1]{\xybox{%
  (-2,0)*{};
  (2,0)*{};
  (0,0)*{}; (0,-18)*{} **\dir{-}; ?(.5)*\dir{<}+(1.7,-7)*{\scs #1};
}}
\newcommand{\lined}[1]{\xybox{%
  (-2,0)*{};
  (2,0)*{};
  (0,0)*{}; (0,-18)*{} **\dir{-}; ?(.5)*\dir{>}+(1.7,-7)*{\scs #1};
}}

\newcommand{\lowrru}[1]{\xybox{%
  (-8,0)*{};
  (8,0)*{};
  (-6,-18)*{};(6,-9)*{} **\crv{(-6,-13) & (6,-15)} ?(1)*\dir{>};
  (6,-9)*{};(6,0)*{}  **\dir{-} ?(.3)*\dir{ }+(2,0)*{\scs {\bf j}};
}}

\newcommand{\lowllu}[1]{\xybox{%
  (-8,0)*{};
  (8,0)*{};
  (6,-18)*{};(-6,-9)*{} **\crv{(6,-13) & (-6,-15)} ?(1)*\dir{>};
  (-6,-9)*{};(-6,0)*{}  **\dir{-} ?(.3)*\dir{ }+(-2,0)*{\scs {\bf j}};
}}

\newcommand{\bbe}[1]{\xybox{%
  (-2,0)*{};
  (2,0)*{};
  (0,0);(0,-18) **\dir{-}; ?(.5)*\dir{<}+(2.3,0)*{\scriptstyle{#1}};
}}

\newcommand{\bbf}[1]{\xybox{%
  (-2,0)*{};
  (2,0)*{};
  (0,0);(0,-18) **\dir{-}; ?(.5)*\dir{>}+(2.3,0)*{\scriptstyle{#1}};
}}

\newcommand{\bbpef}[1]{\xybox{%
  (-6,0)*{};
  (6,0)*{};
  (-4,0)*{}="t1";
  (4,0)*{}="t2";
  "t1";"t2" **\crv{(-4,-6) & (4,-6)}; ?(.15)*\dir{>} ?(.9)*\dir{>}
   ?(.5)*\dir{}+(0,-2)*{\scriptstyle{#1}};
}}
\newcommand{\bbpfe}[1]{\xybox{%
  (-6,0)*{};
  (6,0)*{};
  (-4,0)*{}="t1";
  (4,0)*{}="t2";
  "t2";"t1" **\crv{(4,-6) & (-4,-6)}; ?(.15)*\dir{>} ?(.9)*\dir{>}
  ?(.5)*\dir{}+(0,-2)*{\scriptstyle{#1}};
}}

\newcommand{\bbcfe}[1]{\xybox{%
  (-6,0)*{};
  (6,0)*{};
  (-4,0)*{}="t1";
  (4,0)*{}="t2";
  "t1";"t2" **\crv{(-4,6) & (4,6)}; ?(.15)*\dir{>} ?(.9)*\dir{>}
  ?(.5)*\dir{}+(0,2)*{\scriptstyle{#1}};
}}
\newcommand{\bbcef}[1]{\xybox{%
  (-6,0)*{};
  (6,0)*{};
  (-4,0)*{}="t1";
  (4,0)*{}="t2";
  "t2";"t1" **\crv{(4,6) & (-4,6)}; ?(.15)*\dir{>}
  ?(.9)*\dir{>} ?(.5)*\dir{}+(0,2)*{\scriptstyle{#1}};
}}

\newcommand{\ccbub}[2]{
\xybox{%
 (-6,0)*{};
  (6,0)*{};
  (-4,0)*{}="t1";
  (4,0)*{}="t2";
  "t2";"t1" **\crv{(4,6) & (-4,6)}; ?(.7)*\dir{}+(-2,0)*{\scs #2}
  ?(.05)*\dir{>} ?(1)*\dir{>};
  "t2";"t1" **\crv{(4,-6) & (-4,-6)};
   ?(.3)*\dir{}+(0,0)*{\bullet}+(0,-3)*{\scs {#1}};
}}
\newcommand{\cbub}[2]{
\xybox{%
 (-6,0)*{};
  (6,0)*{};
  (-4,0)*{}="t1";
  (4,0)*{}="t2";
  "t2";"t1" **\crv{(4,6) & (-4,6)};?(.7)*\dir{}+(-2,0)*{\scs #2};
   ?(0)*\dir{<} ?(.95)*\dir{<};
  "t2";"t1" **\crv{(4,-6) & (-4,-6)};
   ?(.3)*\dir{}+(0,0)*{\bullet}+(0,-3)*{\scs {#1}};
}}

\newcommand{\bbdl}[1]{\xybox{%
  (2,0);(0,-8) **\crv{(2,-2)&(0,-6)}; ?(.5)*\dir{>}
}}
\newcommand{\bbdlu}[1]{\xybox{%
  (2,0);(0,-8) **\crv{(2,-2)&(0,-6)}; ?(.5)*\dir{<}
}}
\newcommand{\bbdr}[1]{\xybox{%
  (-2,0);(0,-8) **\crv{(-2,-2)&(0,-6)}; ?(.5)*\dir{>}
}}
\newcommand{\bbdru}[1]{\xybox{%
  (-2,0);(0,-8) **\crv{(-2,-2)&(0,-6)}; ?(.5)*\dir{<}
}}

\newcommand{\lccbub}[1]{
\xybox{%
 (-12,0)*{};
  (12,0)*{};
  (-12,0)*{}="t1";
  (12,0)*{}="t2";
  "t2";"t1" **\crv{(12,14) & (-12,14)}; ?(0)*\dir{>} ?(1)*\dir{>};
  "t2";"t1" **\crv{(12,-14) & (-12,-14)}; ?(.3)*\dir{}+(2,-1)*{\scs #1};
}}
\newcommand{\lcbub}[1]{
\xybox{%
 (-12,0)*{};
  (12,0)*{};
  (-12,0)*{}="t1";
  (12,0)*{}="t2";
  "t2";"t1" **\crv{(12,14) & (-12,14)}; ?(0)*\dir{<} ?(1)*\dir{<};
  "t2";"t1" **\crv{(12,-14) & (-12,-14)}; ?(.3)*\dir{}+(2,-1)*{\scs #1};
}}

\newcommand{\lcap}{
\xybox{%
 (-12,2)*{};(12,2)*{}; **\crv{(12,14) & (-12,14)};
 }}
\newcommand{\xlcap}{
\xybox{%
 (-20,0)*{};(20,0)*{}; **\crv{(20,22) & (-20,22)};
 }}

\newcommand{\lcupef}{\xybox{%
 (12,0)*{};(-12,0)*{} **\crv{(12,-12) & (-12,-12)} ?(0)*\dir{>} ?(1)*\dir{>};
 }}
\newcommand{\xlcupef}{\xybox{%
 (-20,0)*{};(20,0)*{}; **\crv{(20,-22) & (-20,-22)} ?(0)*\dir{>} ?(1)*\dir{>};
 }}
\newcommand{\ecross}{\xybox{%
(-6,0)*{};
  (6,0)*{};
 (-4,-4)*{};(4,4)*{} **\crv{(-4,-1) & (4,1)}?(1)*\dir{>};
 (4,-4)*{};(-4,4)*{} **\crv{(4,-1) & (-4,1)}?(1)*\dir{>};
 }}

\newcommand{\sfline}{\xybox{%
 (0,-4)*{};(0,4)*{} **\dir{-} ?(1)*\dir{<};
}}

\usepackage{tikz}
\usetikzlibrary{decorations.markings}
\usetikzlibrary{decorations.pathreplacing}
\newcommand{\hackcenter}[1]{
 \xy (0,0)*{#1}; \endxy}

\tikzset{->-/.style={decoration={
  markings,
  mark=at position #1 with {\arrow{>}}},postaction={decorate}}}

\tikzset{middlearrow/.style={
        decoration={markings,
            mark= at position 0.5 with {\arrow{#1}} ,
        },
        postaction={decorate}
    }
}

\usepackage{graphicx}
\usepackage{color}

\theoremstyle{plain}
\newtheorem{theorem}{Theorem}

\theoremstyle{definition}

\theoremstyle{remark}

\renewcommand{\to}{\rightarrow}

\def\Id{\mathrm{Id}}

\def\mf{\mathfrak}
\def\shuffle{\,\raise 1pt\hbox{$\scriptscriptstyle\cup{\mskip
               -4mu}\cup$}\,}

\numberwithin{equation}{section}

\def\emph#1{{\sl #1\/}}

\let\hat=\widehat
\let\tilde=\widetilde

\let\theta=\vartheta
\let\epsilon=\varepsilon

\usepackage{bbm}
\def\C{{\mathbbm C}}
\def\N{{\mathbbm N}}

\def\Z{{\mathbbm Z}}
\def\Q{{\mathbbm Q}}

\def\cal#1{\mathcal{#1}}%
\def\1{\mathbbm{1}}%
\def\tr{\mathrm{tr}}%
\def\nn{\notag}

\def\la{\langle}
\def\ra{\rangle}

\DeclareGraphicsExtensions{.pdf,.eps}
\usepackage{psfrag}
\usepackage{bbm}

\def\cal#1{\mathcal{#1}}

\def \k {\mathbbm{k}}
\def \Z {\mathbbm{Z}}

\def \N {\mathbbm{N}}
\def \Q {\mathbbm{Q}}
\def \E {\mathcal{E}}

\def \U {\mathcal{U}}

\def \C {\mathcal{C}}
\def \Tr{\operatorname{Tr}}

\def \Span{\operatorname{Span}}
\def \Ob{\operatorname{Ob}}

\def \HH{\operatorname{HH}}
\def \min {{\rm min}}
\def \Id {{\rm Id}}

\def\k{\mathbbm{k}}

\newcommand{\xto}[1]{{\overset{#1}{\longrightarrow}}}

\newcommand\nc{\newcommand}
\nc\rnc{\renewcommand}
\nc\Kar{\operatorname{Kar}}
 \nc\modQ {{\mathbb Q}}
\nc\modZ {{\mathbb Z}}
\nc\simeqto{\overset{\simeq}{\longrightarrow }}
\nc\modC {{\mathcal C}}
\nc\modD {{\mathcal D}}
\nc\K{\mathcal {K}}
\nc\CC{\mathbf{C}}

\nc\calU{\mathcal{U}}
\nc\cU{\calU}
\nc\Ab{\mathbf{Ab}}
\nc\Ko{K_0}
\nc\TrhorCC{\Tr^{\mathrm{hor}}(\CC)}
\nc\AdCat{\mathbf{AdCat}}
\nc\TrCC{\Tr(\CC)}
\nc\Udot{\dot{\mathcal{U}}}
\nc\diag{\mathrm{d}}
\nc\modU {\mathcal{U}}
\nc\bfU{\mathbf{U}}
\nc\dU{\dot{\mathbf U}}
\nc\dUZ{{_\modZ\dot{\mathbf U}}}
\nc\UZ{{_\modZ \mathbf U} }
\nc\fsl{\mathfrak{sl}}
\nc\Uaa{{\bf U} (\mathfrak{sl}_2\otimes \Q[t,t^{-1}])}
\nc\UZslt{{_\modZ\mathbf{U}} (\mathfrak{sl}_2\otimes \Q[t])}
\nc\UdZslt{{_\modZ\dot{\mathbf{U}}} (\mathfrak{sl}_2\otimes \Q[t])}
\nc\LL{L^+\fsl_2}
\nc\UL{\mathbf U(\LL)}
\nc\UZL{\UZ(\LL)}
\nc\dUZL{\dUZ(\LL)}
\nc\dUL{\dU(\LL)}
\nc\Kom{\rm Kom}
\nc\GL{\rm{GL}}
\nc\g{\mathfrak{g}}

\nc\tG{\tilde{G}} \nc\tE{\tilde{E}}
\nc\Vect{\rm Vect}
\nc{\Gras}{{\rm {Gr}}}
\nc\FMod{\rm FMod}
\nc\yto[1]{\underset{#1}{\to}}
\nc\Ear{\yto{E}}

\def\l{\lambda}
\newcommand{\iccbub}[2]{
\xybox{%
 (-6,0)*{};
  (6,0)*{};
  (-4,0)*{}="t1";
  (4,0)*{}="t2";
  "t2";"t1" **\crv{(4,6) & (-4,6)}; ?(.7)*\dir{}+(-2,0)*{\scs #2}
  ?(.05)*\dir{>} ?(1)*\dir{>};
  "t2";"t1" **\crv{(4,-6) & (-4,-6)};
   ?(.3)*\dir{}+(0,0)*{\bullet}+(0,-3)*{\scs {#1}};
}}
\newcommand{\icbub}[2]{
\xybox{%
 (-6,0)*{};
  (6,0)*{};
  (-4,0)*{}="t1";
  (4,0)*{}="t2";
  "t2";"t1" **\crv{(4,6) & (-4,6)};?(.7)*\dir{}+(-2,0)*{\scs #2};
   ?(0)*\dir{<} ?(.95)*\dir{<};
  "t2";"t1" **\crv{(4,-6) & (-4,-6)};
   ?(.3)*\dir{}+(0,0)*{\bullet}+(0,-3)*{\scs {#1}};
}}

\allowdisplaybreaks

\newdimen\hoogte    \hoogte=8pt    
\newdimen\breedte   \breedte=8pt   
\newdimen\dikte     \dikte=0.5pt    

\newenvironment{young}{\begingroup
       \def\vr{\vrule height0.8\hoogte width\dikte depth 0.2\hoogte}
       \def\fbox##1{\vbox{\offinterlineskip
                    \hrule height\dikte
                    \hbox to \breedte{\vr\hfill##1\hfill\vr}
                    \hrule height\dikte}}
       \vbox\bgroup \offinterlineskip \tabskip=-\dikte \lineskip=-\dikte
            \halign\bgroup &\fbox{##\unskip}\unskip  \crcr }
       {\egroup\egroup\endgroup}
\def\diagram#1{\relax\ifmmode\vcenter{\,\begin{young}#1\end{young}\,}\else%
              $\vcenter{\,\begin{young}#1\end{young}\,}$\fi}

\newdimen\Hoogte    \Hoogte=12pt    
\newdimen\Breedte   \Breedte=12pt   
\newdimen\Dikte     \Dikte=0.5pt    

\newenvironment{Young}{\begingroup
       \def\vr{\vrule height0.8\Hoogte width\Dikte depth 0.2\Hoogte}
       \def\fbox##1{\vbox{\offinterlineskip
                    \hrule height\Dikte
                    \hbox to \Breedte{\vr\hfill##1\hfill\vr}
                    \hrule height\Dikte}}
       \vbox\bgroup \offinterlineskip \tabskip=-\Dikte \lineskip=-\Dikte
            \halign\bgroup &\fbox{##\unskip}\unskip  \crcr }
       {\egroup\egroup\endgroup}
\def\Diagram#1{\relax\ifmmode\vcenter{\,\begin{Young}#1\end{Young}\,}\else%
              $\vcenter{\,\begin{Young}#1\end{Young}\,}$\fi}

\begin{document}
\title{Trace  and categorical \(\mathfrak{sl}_n\) representations }
\author {Zaur Guliyev} 

\maketitle

\begin{abstract}
Khovanov-Lauda \cite{lauda2} define a 2-category \(\U\) such that the split Grothendieck group \(K_0(\U)\) is isomorphic to an integral version of the quantized universal enveloping algebra \(\mathbf{U}(\mf{sl}_n)\), \(n \geq 2\). Beliakova-Habiro-Lauda-Webster \cite{lcurr} prove that the trace decategorification \(\Tr(\U)\) of the Khovanov-Lauda 2-category is isomorphic to the 
the current algebra \(\mathbf{U}(\mathfrak{sl}_n [t])\) -- the universal enveloping algebra of the Lie algebra \(\mathfrak{sl}_n \otimes \mathbb{C} [t]\). A 2-representation of \(\,\U\) is a 2-functor from \(\, \U\) to a linear, additive 2-category.  In this note we are interested in the 2-representation, defined by Khovanov-Lauda using bimodules over cohomology rings of flag varieties. This 2-representation induces an action of the current algebra \(\mathbf{U}(\mathfrak{sl}_n [t])\) on the cohomology rings. We explicitly compute the  action of \(\mathbf{U}(\mathfrak{sl}_n [t])\) generators using the trace functor. It turns out that the obtained current algebra module is related to another family of \(\mathbf{U}(\mathfrak{sl}_n [t])\)-modules, called local Weyl modules. Using known results about the cohomology rings, we are able to provide a new proof of the character formula for the local Weyl modules.  

\end{abstract}
 
\tableofcontents
\section*{Introduction}
\addcontentsline{toc}{section}{Introduction}

{\em Categorification} is a process of finding category-theoretic analogs of set-like mathematical structures by replacing sets with categories, functions with functors and equations between functions by natural isomorphisms between functors.  {\em Decategorification} is the inverse process which 'forgets' the morphisms in the category and identifies the isomorphic objects. Decategorification is usually understood as a functor that returns the  split Grothendieck group  $K_0(\C)$ of an additive category \(\C\). The split Grothendieck group $K_0(\C)$ is the abelian group generated by isomorphisms classes $\{[x]_{\cong}\}_{ x\in \Ob(\C)}$, modulo the relation $$[x \oplus y]_{\cong} = [x]_{\cong}  + [y]_{\cong}.$$  A 2-category is a category with an additional structure -- the hom-sets are also categories. The split Grothendieck group of a 2-category can be thought of as a procedure for turning the 2-categorical structure of an additive category into a 1-categorical structure of an abelian group. We get a 1-category if we forget 2-morphisms in a 2-category and identify all isomorphic 1-morphisms.

We say that a category is  \(\k\)-linear if the set of morphisms form a \(\k\)-linear vector space, and  compositions are bilinear over \(\k\). The trace of a \(\k\)-linear category \(\C\) is defined by
 \begin{equation*} \label{trdef}
  \Tr \C =\left( \bigoplus_{x \in \Ob(\C )} \End_{\C} (x) \right) /\Span_{\k}\{f\circ g-g\circ f\},
\end{equation*}
where $f$ and $g$ run through all pairs of morphisms $f \colon x\to y$ and $g\colon y \to x$ with $x,y\in \Ob(\C)$. The trace \(\Tr \C\) of a 2-category \(\C\) is a category who has the same object set as \(\C\), and morphisms of \(\Tr \C\) are the trace classes of 2-endomorphisms of  \(\, \C\).

A milestone in categorification was achieved independently by M. Khovanov and A. Lauda  \cite{lauda2} and by R. Rouquier \cite{rouq}. They define a 2-category \(\U_Q(\mf{sl}_n)\) whose split Grothendieck group gives the integral version of the idempotent completion \( \dot{\mathbf{U}}_q (\mathfrak{sl}_n)\) of the associated quantum group \( \mathbf{U}_q (\mathfrak{sl}_n)\). In this note we will mainly work with the diagrammatic 2-category defined by Khovanov-Lauda \cite{lauda2}. They categorify Beilinson-Lusztig-MacPherson \cite{beil} idempotent version \( \dot{\mathbf{U}}_q (\mathfrak{sl}_n)\), in which the unit in \( \mathbf{U}_q (\mathfrak{sl}_n)\) is replaced by a collection of idempotents \(\{1_{\nub}\}\), indexed by elements \(\nub \) of the integral \(\mathfrak{sl}_n\) weight lattice \(X\). The  \(\mathbb{Q}(q)\)-module  \( \dot{\mathbf{U}}_q (\mathfrak{sl}_n)\) has a natural category structure with the object set \(X\), and hom-sets from \(\nub \in X\) to \(\mub \in X\) are given by  \( 1_{\mub} \dot{\mathbf{U}}_q (\mathfrak{sl}_n) 1_{\nub}\). The algebra \( \dot{\mathbf{U}}_q (\mathfrak{sl}_n)\) ̇ is generated over \(\mathbb{Q}(q)\) by elements \(1_{\nub}\),  \(1_{\nub +\alpha_i} E_i 1_{\nub}\) and \(1_{\nub -\alpha_i} F_i 1_{\nub}\),  modulo certain relations, where \( \nub \in X\), \(i \in I =\{1, \dots , n-1\}\) and \(\alpha_i$ is the \(i\)-th simple root.

Khovanov and Lauda define $\U_Q(\mf{sl}_n)$ to be the 2-category with the object set \(X\) and 1-morphisms generated by the identity 1-morphism $\onel $ and the morphisms  $\cal{E}_i \onel \colon \nub \to \nub+\alpha_i$, $\cal{F}_i\onel \colon \nub \to \nub-\alpha_i$ for all \( i \in I\). For each 1-morphism $f$ there is also a degree shifted 1-morphism $f \la t\ra$ in $\U_Q(\mf{sl}_n)$ for each \(t \in \Z\). The 2-morphisms are \(\k\)-vector spaces spanned by certain oriented planar diagrams, modulo isotopies and local relations. For example, 
\[ 
\xy 0;/r.17pc/:
  (0,0)*{\xybox{
    (-4,-4)*{};(4,4)*{} **\crv{(-4,-1) & (4,1)}?(1)*\dir{>} ;
    (4,-4)*{};(-4,4)*{} **\crv{(4,-1) & (-4,1)}?(1)*\dir{>};
    (-6,-7)*{\scs {\cal E}_i};
     (6,-7)*{\scs {\cal E}_{j}};
     (-6,7)*{\scs {\cal E}_{j }};
     (6,7)*{\scs {\cal E}_{i} \la 1\ra};
     (11,1)*{\scs  \nub}; (-16,1)*{\scs  \nub +\alpha_i+\alpha_j};
     (-10,0)*{};(10,0)*{}; (0,8)*{};(0,-8)*{};
     }};
  \endxy
  \]
is a 2-morphism \( {\cal E}_{i} {\cal E}_{j} \onel \to {\cal E}_{j} {\cal E}_{i} \onel \la 1\ra \).  Let $\dot{\U}_Q(\mf{sl}_n)$ be the smallest 2-category which contains $\U_Q(\mf{sl}_n)$ and has splitting idempotents. The split Grothendieck group $K_0(\dot{\U}_Q(\mf{sl}_n))$ can be regarded as a \(\Z[q, q^{-1}]\)-module if we forget the 2-morphisms, identify isomorphic 1-morphisms and let the degree shift \(\la 1\ra \) of a 1-morphism correspond to a multiplication by \(q\). Khovanov and Lauda prove that there is a \(\Z[q, q^{-1}]\)-module  isomorphism between an integral  form of  \( \dot{\mathbf{U}}_q (\mathfrak{sl}_n)\)  and the split Grothendieck group $K_0(\dot{\U}_Q(\mf{sl}_n))$ which maps $q^t E_i 1_{\nub}$ to $[\cal{E}_i \onel \la t\ra]_{\cong}$, $q^tF_i 1_{\nub}$ to  $[\cal{F}_i\onel  \la t\ra ]_{\cong}$ for all \( i \in I\) and $q^t1_{\nub}$ to $[ \onel \la t\ra]_{\cong} $.

 Let \(\U\) be the modification of $\U_Q(\mf{sl}_n)$,  where degree shifts of each 1-morphism are identified, and 2-morphisms are not required  to preserve the degree. The split Grothendieck group  \(K_0(\U)\) is a \(\Z\)-module rather than a \(\Z[q,q^{-1}]\)-module since identifying morphisms that differ by a degree shifts will correspond to setting \(q=1\) on the \(K_0\) level. 

The trace \( \Tr \U\) turns out to be a larger algebra with a \(\Z\)-grading.  Let \(\mf{sl}_n[t]=\mf{sl}_n \otimes \k[t]\) be the extension of \(\mf{sl}_n \) by the polynomials in parameter \(t\). The Lie bracket \( [a \otimes t^r, b \otimes t^s] = [a,b] \otimes t^{r+s}\) defines a graded Lie algebra structure on \(\mf{sl}_n \otimes \k[t]\) with  \(\deg (a \otimes t^r) = 2r\). We call   the universal enveloping algebra \(\mathbf{U}(\mf{sl}_n [t])\) the {\em current algebra}. Current algebra is generated over $\mathbb{\k} $ by \(E_{i,r}=E_{i}\otimes t^r, F_{i,r}=F_i\otimes t^r, H_{i,r}=H_i\otimes t^r\) for all $r \geq 0$ and \(i \in I\), where \( E_i, F_i, H_i\) are the Chevalley basis of \(\mf{sl}_n \). Beliakova-Habiro-Lauda-Webster \cite{lcurr} prove that \( \Tr \U\)  is algebraically isomorphic to \(\mathbf{U}(\mf{sl}_n [t])\).

Representations of \( \mathbf{U}_q (\mathfrak{sl}_n)\) are used to obtain invariants of links and 3-manifolds. A motivation for categorifying the representations is to achieve even stronger invariants. 
A {\em 2-representation} or {\em categorical representation} of  \(\mathbf{U}_q(\mathfrak{sl}_n)\) is a graded, additive 2-functor \( \U_Q(\mathfrak{sl}_n) \to {\cal K}\) for some graded, additive 2-category \( {\cal K}\). Categorical representations of  \(\mathbf{U}_q(\mathfrak{sl}_n)\) have been studied even before the categorification of the quantum group  \(\mathbf{U}_q(\mathfrak{sl}_n)\) itself. A categorification of the tensor products of the fundamental \(\mathfrak{sl}_2\) representation  was first algebraically formulated by Bernstein-Frenkel-Khovanov \cite{bfk}, later extended to the quantum group  \( \mathbf{U}_q(\mathfrak{sl}_2)\) by Frenkel-Khovanov-Stroppel \cite{fren}, Stroppel \cite{ strop}, Chuang-Rouquier \cite{cr}, and to \( \mathbf{U}_q(\mathfrak{sl}_n)\) by Mazorchuk-Stroppel \cite{maz}, Webster \cite{ben}. A big role here has always played the graded, parabolic BGG category \({\cal O}\) of certain finitely generated \(\mathfrak{gl}_m\) weight modules. 

A modification of  \({\cal K}\) by allowing 2-morphisms of any degree lets us to define a 2-representation \(\U \to {\cal K}\). We define the center \(Z(x)\) of an object \(x \in \Ob({\cal K})\) to be the endomorphism ring \(\End(1_x)\) and {\em the center of objects} \(Z(\K)\) to be \(Z(\K)= \bigoplus_{x \in \Ob(\K)} Z(x) \). In \cite{trac} it has been proven that any 2-representation \(\, \U \to \mathcal{K}\) induces  an  action of \(\,\mathbf{U}(\mf{sl}_n[t]) \) on \(\Tr \mathcal{K}\) and on  \(Z (\mathcal{K})\).

Our first result is a construction of a current algebra module using  a 2-representation defined by Khovanov-Lauda \cite{lauda2}. We first present their 2-representation in terms of symmetric polynomials and then compute the current algebra action on the center of objects.

A {\em \(n\)-composition of \(N\)} is an ordered \(n\)-tuple of non-negative integers which sum up to a fixed integer \(N\). Let  \(\nu = (\nu_1, \nu_2, \dots, \nu_n) \) be  a {\em \(n\)-composition of \(N\)} and \(P_\nu\) be the subalgebra of the polynomial ring in variables \( X_1, X_2, \dots, X_N\), which is symmetric in the first \(\nu_1\)-tuple of variables, in the second \(\nu_2\)-tuple of variables and so forth. This ring is generated by complete symmetric polynomials \(h_r(\nu; i)\) in \(i\)-th tuple of \(\nu_i\) variables over all \(r \geq 0\), \(1\leq i \leq n\), as well as by elementary symmetric polynomials \(e_r(\nu; i)\) or by power sum symmetric polynomials \(p_r(\nu; i)\).  

We will construct a 2-category \(\textbf{Pol}_N\) together with a 2-representation
\(\Theta_N \colon \mathcal{U} \to \textbf{Pol}_N\). Objects in \(\textbf{Pol}_N\) are the graded rings \(P_\nu\) for each \(n\)-composition of \(N\), 1-morphisms are bimodules over \(P_\nu\), and 2-morphisms are bimodule maps.  \(\textbf{Pol}_N\) is equivalent to the 2-category \(\textbf{EqFlag}_N\), which is defined by Khovanov-Lauda to show the non-degeneracy of \(\, \U\). \(\textbf{EqFlag}_N\) is constructed using bimodules over  equivariant cohomology rings of partial flag varieties, and mapping Chern classes in \(\textbf{EqFlag}_N\)  to elementary symmetric polynomials in \(\textbf{Pol}_N\) defines an equivalence of 2-categories. 

 The split Grothendieck group of \(\textbf{Pol}_N\) gives the finite-dimensional, simple highest weight \(\mathfrak{sl}_n\) module \(V(\lb_0)\) of the highest weight \( \lb_0 = N\omega_1\) for the first fundamental weight \(\omega_1\).  This proven by Frenkel-Khovanov-Stroppel \cite{fren} for the case \(n=2\) and by Khovanov-Lauda \cite{lauda2} for general \(n\).
 
Let \(\check{R}^{\lb}\) be the deformed KLR algebra assigned to a dominant weight \(\lb\) (see \cite{ben}, section 3.3), and let \(\check{R}^{\lb} \text{-}  \mathrm{ pMod}\) be the category of finitely generated, projective, graded \(\check{R}^{\lb} \)-modules. The 2-functor \(\Theta_N \) factors through \(\check{R}^{\lb} \text{-}  \mathrm{ pMod}\):
\[ \begin{large}
 \begin{tikzcd}
  \scs \mathcal{U} \arrow[rd] \arrow[r, "\Theta_N"] & \scs \textbf{Pol}_N \\
  &\scs \check{R}^{\lb_0} \text{-}  \mathrm{ pMod}\arrow[u,  "\Theta'_N", swap] 
  \end{tikzcd},
  \end{large}
  \]
where \(\Theta'_N\) is a strongly equivariant 2-functor in the sense of Webster \cite{ben}.

The 2-representation  \(\Theta_N\) induces a current algebra action on the center of objects 
\(Z(\textbf{Pol}_N) =  \bigoplus_{\nu} P_{\nu} \). We identify the current algebra module structure in \(\bigoplus_{\nu} P_{\nu} \) explicitly  in our next result.

\begin{theorem} \label{thf}
Let \(\nu\) be an \(n\)-composition of \(N\), \(i \in I\), \(\alpha_i = (0, \dots, 0, 1, -1,0, \dots 0)\), \(k_i = \sum_{l=1}^i \nu_l\) and \(m, j \geq 0\).
The action of the current algebra \(\mathbf{U}(\mathfrak{sl}_n[t])\) on \(\bigoplus_{\nu} P_{\nu} \) is defined as follows.
\begin{enumerate}
\item The map \(F_{i,j} \colon P_\nu \to  P_{\nu-\alpha_i} \)  is the \(P_{\nu-\alpha_i}\)-module homomorphism such that 
\begin{equation*} \label{ss1}
F_{i,j} (X^{m}_{k_i})=
\sum_{l=0}^{\nu_i-1}(-1)^{l} e_l(\nu-\alpha_i;i) h_{m+j+\nu_i- \nu_{i+1}  -1-l }(\nu-\alpha_i; i+1).
\end{equation*}
\item The map \(E_{i,j} \colon P_\nu \to  P_{\nu+\alpha_i} \)  is the \(P_{\nu+\alpha_i}\)-module homomorphism such that 
\begin{equation*} \label{ss2}
E_{i,j} (X^{m}_{k_i+1})= 
\sum_{l=0}^{\nu_{i+1}-1}(-1)^{l} e_l(\nu+\alpha_i;i+1)   h_{m+j+\nu_{i+1}- \nu_i-1-l}(\nu+\alpha_i; i) .
\end{equation*}
\item  The map \(H_{i,j} \colon P_\nu \to P_\nu\) is the multiplication by
 \( (-1)^j(p_j(\nu; i+1) - p_j(\nu; i))\)  if \(j>0\), and by \(\nu_i - \nu_{i+1}\) if  \(j=0\). 
 \end{enumerate}
\end{theorem}
The proof of Theorem \ref{thf} follows from a direct computation using diagrammatic description of 2-morphisms in  \(\U\) and some relations in the polynomial algebra. The action of \(E_{i,j}\) and  \(F_{i,j}\) for  \(j=0\)  are already given by Brundan \cite{brun1} using Schur-Weyl duality, Theorem \ref{thf} generalizes it for a current algebra parameter \(j\) in the context of 2-representation theory.

Brundan  \cite{brun1} defines a finite-dimensional quotient \( C_\nu^\l\)  of the algebra \(P_\nu\) for a given \(n\)-composition \(\nu\) and {\em \(n\)-partition} \(\l\) of \(N\). We will show that the current algebra action defined in Theorem \ref{thf} descends to the quotient ring \(\bigoplus_{\nu} C_\nu^\l\), following a similar proof for \(\mf{sl}_n\)  action  given by Brundan  \cite{brun1}.

The rings \( C_\nu^\l \)  appear naturally in geometry and 2-representation theory. Brundan and Ostrik \cite{brun2} prove that \( C_\nu^\l \) is isomorphic to the cohomology ring of a Spaltenstein variety -- a quiver variety of type \(A\). Mackaay \cite{mack} shows that \(\bigoplus_{\nu} C_\nu^\l \)  is isomorphic to the center of the \(\mathfrak{sl}_m\) web algebra and to the center of the foam category. Brundan \cite{brun1} proves an algebra isomorphism  between  \(\bigoplus_{ \nu} C_\nu^\l \) and the center of parabolic category \(\mathcal{O}^\l\). Varagnolo-Vasserot-Shan \cite{vas} and Webster \cite{bens} prove a graded algebra isomorphism between the center of the cyclotomic KLR algebra and \(\bigoplus_{\nu} C_\nu^\l\).

It turns out that the current algebra action on \(\bigoplus_{\nu} C_\nu^\l\) is related to a certain family of current algebra modules, called {\em local Weyl modules}.  A local Weyl module \(W(\lb)\) of the current algebra  \(\mathbf{U}(\mathfrak{sl}_n[t])\) assigned to a dominant weight \(\lb \in X\) is defined to be the graded module generated by \(v_{\lb}\) over  \(\mathbf{U}(\mathfrak{sl}_n[t])\), subject to the following relations for all \(i \in I\) and \(j \geq 0\): 
\[ E_{i,j} v_{\lb}=F^{\lb_i+1}_{i,0} v_{\lb}=0 , \qquad H_{i,j} v_{\lb}=\delta_{0,j} \, \lb_i v_{\lb},\]
where \(\delta\) is the Kronecker symbol.  \(W(\lb)\)   is  finite-dimensional, inherits degree from the current algebra, and each homogeneous  piece  \(W(\lb) \{2r\}\), \(r \geq 0\) is an \(\mathfrak{sl}_n\)-module.

Local Weyl modules play an important role in the study of finite-dimensional representations of affine Lie algebras. Chari-Pressley \cite{chp}  were first to introduce them, where they prove  certain universality properties. Since then local Weyl modules and their infinite-dimensional analogues, called {\em global Weyl modules}, have attracted a vast amount of interest. Chari-Loktev \cite{lamp} show that local Weyl modules are isomorphic to another families of current algebra modules, called Demazure modules and fusion modules. Fusion modules are certain graded tensor products of fundamental \(\mathfrak{sl}_n\) representations, and by setting \(t=1\) we recover the usual tensor products. Kodera-Naoi \cite{kod} prove that radical and socle series of local Weyl modules coincide with their grading filtration. They also relate local Weyl modules to the homology groups of Nakajima's quiver varieties of type \(A\).

The graded character and dimension formulas for  \(\bigoplus_{\nu} C_\nu^\l \) are given by Brundan \cite{brun1}. This result allows us to compute  graded characters and  dimensions of the local Weyl modules in terms of the {\em Kostka-Foulkes polynomials} (see \cite{mcd}). The Kostka-Foulkes polynomials \(K_{\l, \nu}(t)\), indexed by a partition \(\l\) and a composition \(\nu\)  play major role in the representation theory of affine Lie algebras and symmetric groups. 

\begin{theorem} \label{lqqq}
Let \(\l\) be a \(n\)-partition of a positive integer \(N\). Let \(\lb = \sum_{i=1}^{n-1} (\l_i - \l_{i+1}) \omega_i\), where \(\omega_i\), \(i \in I\) are fundamental \(\mf{sl}_n\) weights. 
 The graded character of the local Weyl  module \(W(\lb)\) is given by the formula
 \begin{equation*} 
\text{ch}_t \,  W(\lb) = \sum_{\tau} K_{ \tau^T, \l^T}  (t^{}) \, \text{ch}_{} \, V(\bar{\tau}) ,
\end{equation*}
and the graded dimension of the weight space \(\nub\) is given by
\begin{equation*} 
\sum_{r \geq 0}
 \dim_\mathbb{C} \, W_{\nub}(\lb) \{ 2r \} \, t^r = \sum_{\tau} K_{\tau, \nu}(1) \, K_{ \tau^T, \l^T}  (t^{}),
 \end{equation*}
 where the summations on the right hand sides are over all \(n\)-compositions of \(N\), \(V(\bar{\tau}) \) is the finite-dimensional, simple \(\mf{sl}_n\) module with the highest weight \(\bar{\tau} = \sum_{i=1}^{n-1} (\tau_i - \tau_{i+1}) \omega_i\), and \( W_{\nub}(\lb) \) is the weight space 
 \( W_{\nub}(\lb) = \{w \in W_{}(\lb)\, | \, H_{i,0} w = (\nu_i-\nu_{i+1})w \}\). 
\end{theorem} 

The proof of Theorem \ref{lqqq} follows easily from the work of Shan-Varagnolo-Vasserot \cite{vas} (also see Webster \cite{bens}). They show that \(\bigoplus_{\nu} C_\nu^\l \) is isomorphic to the dual Weyl module \(W^*(\lb)\). We only need to adjust the degree, and this will gives a character formula. This formula can be understood as a {\em character decategorification} of \(W(\lb)\). The graded dimension is obtained from the character formula.

Theorem \ref{lqqq} is probably already known to experts. We, however, give a new proof with a different approach and a simpler formula. Sanderson \cite{sand} computes characters of Demazure modules, and this can be combined with results of  Chari-Loktev \cite{lamp} to derive the characters for local Weyl modules. Chari-Ion (\cite{bgg}, Theorem 4.1) also give a character formula, however they use an entirely different approach, combinatorics and proof of which is beyond our understanding.

\textit{Organization.} The structure of this note is organized as follows. We start by giving the relevant definitions concerning traces and 2-categories. In the second section we define the Khovanov-Lauda 2-category using the diagrammatic presentation. We also state main results in trace decategorification of categorified quantum groups.  In the third and fourth sections, we concern ourselves with the local Weyl modules of current algebra and 2-representations of categorified quantum \(\mathfrak{sl}_n\). The last section is devoted to the polynomial construction of the 2-category \(\textbf{Pol}_N\). We then explicitly describe the current algebra action on the center of objects of \(\textbf{Pol}_N\) and use it to derive the character formula of local Weyl modules.

\textit{Acknowledgements.} The author thanks \text{Anna Beliakova} for suggesting this project and teaching a great deal about the categorified quantum groups; \text{Jonathan Brundan} and \text{Aaron Lauda} for valuable comments on earlier versions of this note. The author was supported by the Swiss National Science Foundation.

 \section{Trace functor}
\subsection{Trace of a ring}
Let \(\k\) be a field of zero characteristics. The trace of a square matrix with elements in \(\k\) is the sum of its diagonal entries. Given two square matrices $A,B$, the trace satisfies the property
\begin{equation} \label{trace-property}\tr(AB)=\tr(BA)\, .
\end{equation}
More generally, for any finite-dimensional $\k$-vector space $V$, a trace is a linear map
 $$ \tr\colon \End_{\k}(V) \to \k$$
such that the {\it trace relation} holds:
$$
\tr (fg)=\tr (gf)\quad\text{for any}\quad f,g \in \End_{\k}(V)\, .$$
As a simple invariant of a linear endomorphism, the trace has many
important applications. For example, given a representation $\phi\colon G\to {\rm Aut}_{\k}(V)$ of a group $G$,
the character $\chi_{\phi} \maps G \to \k$ of the representation $\phi$ is the function sending each group element $g\in G$ to the trace of $\phi(g)$.  The character of a representation carries much of the essential information about a representation. 

We now generalize the notion of the trace from the ring of endomorphisms of \(V\) to any \(\k\)-algebra.  For a \(\k\)-algebra $A$, the {\em Hochschild homology} $\HH_*(A)$ of $A$ can be defined as the homology of the Hochschild chain complex
\begin{equation} \label{hoc}
C_{\bullet}=C_{\bullet}(A)\colon \quad \quad \dots\longrightarrow C_n(A) \xto{d_n} C_{n-1}(A) \xto{d_{n-1}} \dots \xto{d_2} C_1(A) \xto{d_1} A \longrightarrow 0,
\end{equation}

where $C_n(A)=A^{\otimes n+1}$ and
$$
d_n(a_0\otimes\dots\otimes
a_n):=\sum_{i=0}^{n-1}(-1)^ia_0\otimes\dots\otimes
a_ia_{i+1}\otimes\dots\otimes a_n+
(-1)^na_na_0\otimes a_1\otimes\dots\otimes a_{n-1}
$$
for $a_0,\ldots ,a_{n-1}\in A$.
We define  {\em the trace} or  {\em the cocenter} of \(A\) to be the zeroth Hochschild homology \(\HH_0(A)\) of $A$. 
\begin{Definition}
The trace of a \(\k\)-algebra \(A\) defined as 
\[ \Tr A = \HH_0(A)= A/[A,A] = A/\, \text{Span}_{\, \k}\{ab-ba \, | \, a,b \in A\} .\] 
\end{Definition}
Many interesting algebras can be presented by diagrams in the plane modulo some local relations.  Multiplication in $A$ is represented by stacking diagrams on top of each other:
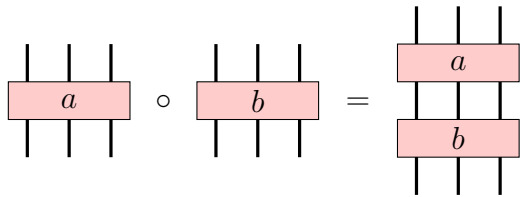
\[
\hackcenter{\begin{tikzpicture}
    \draw[very thick] (-.55,0) -- (-.55,1.5);
    \draw[very thick] (0,0) -- (0,1.5);
    \draw[very thick] (.55,0) -- (.55,1.5);
    \draw[fill=red!20,] (-.8,.5) rectangle (.8,1);
    \node () at (0,.75) {$a$};
\end{tikzpicture}}
\;\;
\circ \;\;
\hackcenter{\begin{tikzpicture}
    \draw[very thick] (-.55,0) -- (-.55,1.5);
    \draw[very thick] (0,0) -- (0,1.5);
    \draw[very thick] (.55,0) -- (.55,1.5);
    \draw[fill=red!20,] (-.8,.5) rectangle (.8,1);
    \node () at (0,.75) {$b$};
\end{tikzpicture}}
\;\; = \;\;
\hackcenter{\begin{tikzpicture}
    \draw[very thick] (-.55,0) -- (-.55,2.5);
    \draw[very thick] (0,0) -- (0,2.5);
    \draw[very thick] (.55,0) -- (.55,2.5);
    \draw[fill=red!20,] (-.8,.5) rectangle (.8,1);
    \draw[fill=red!20,] (-.8,1.5) rectangle (.8,2);
    \node () at (0,.75) {$b$};
    \node () at (0,1.75) {$a$};
\end{tikzpicture}}
\]
for $a,b \in A$.  Then the trace $\Tr A$ can be described by diagrams on the annulus, modulo the same local relations
as those for $A$ and the map
\[
 A \longrightarrow A/[A,A]
\]
has a diagrammatic description
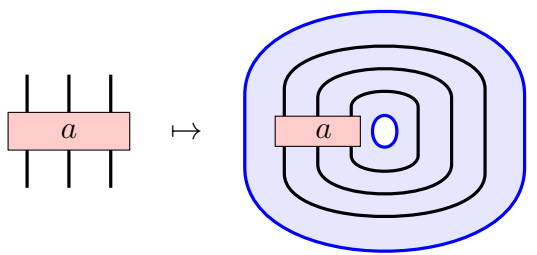
\[
\hackcenter{\begin{tikzpicture}
    \draw[very thick] (-.55,0) -- (-.55,1.5);
    \draw[very thick] (0,0) -- (0,1.5);
    \draw[very thick] (.55,0) -- (.55,1.5);
    \draw[fill=red!20,] (-.8,.5) rectangle (.8,1);
    \node () at (0,.75) {$a$};
\end{tikzpicture}}
\quad \mapsto
\quad
\hackcenter{\begin{tikzpicture}[scale=0.80]
      \path[draw,blue, very thick, fill=blue!10]
        (-2.3,-.6) to (-2.3,.6) .. controls ++(0,1.85) and ++(0,1.85) .. (2.3,.6)
         to (2.3,-.6)  .. controls ++(0,-1.85) and ++(0,-1.85) .. (-2.3,-.6);
        \path[draw, blue, very thick, fill=white]
            (-0.2,0) .. controls ++(0,.35) and ++(0,.35) .. (0.2,0)
            .. controls ++(0,-.35) and ++(0,-.35) .. (-0.2,0);
    \draw[very thick] (-1.65,-.7) -- (-1.65, .7).. controls ++(0,.95) and ++(0,.95) .. (1.65,.7)
        to (1.65,-.7) .. controls ++(0,-.95) and ++(0,-.95) .. (-1.65,-.7);
    \draw[very thick] (-1.1,-.55) -- (-1.1,.55) .. controls ++(0,.65) and ++(0,.65) .. (1.1,.55)
        to (1.1,-.55) .. controls ++(0,-.65) and ++(0,-.65) .. (-1.1, -.55);
    \draw[very thick] (-.55,-.4) -- (-.55,.4) .. controls ++(0,.35) and ++(0,.35) .. (.55,.4)
        to (.55, -.4) .. controls ++(0,-.35) and ++(0,-.35) .. (-.55,-.4);
    \draw[fill=red!20,] (-1.8,-.25) rectangle (-.4,.25);
    \node () at (-1,0) {$a$};
\end{tikzpicture}} \; .
 \]

Notice that by sliding elements of $A$ around the annulus we see that $ab=ba$ in the quotient $\HH_0(A)$. 

The Hochschild cochain complex is defined as
\begin{equation*} \label{hoc}
C^{\bullet}=C^{\bullet}(A)\colon  \quad  0  \xrightarrow{} A   \xrightarrow{\delta_0} C^0(A) \xrightarrow{\delta_1}  C^{1}(A) \xrightarrow{\delta_2}  \dots \xrightarrow{\delta_{n-1}}  C^{n-1}(A) \xrightarrow{\delta_n}  C^n(A) \longrightarrow \dots,
\end{equation*}

where $C^n(A)=\text{Hom}_{\mathbf{k}} (A^{\otimes n+1}, A)$ and
$$
\delta_n (f)(a_0\otimes\dots\otimes
a_n)=\sum_{i=0}^{n-1}(-1)^i f(a_0\otimes\dots\otimes
a_ia_{i+1}\otimes\dots\otimes a_n)+
(-1)^n f(a_na_0\otimes a_1\otimes\dots\otimes a_{n-1}),
$$
$$
\delta_0 (a_0)( a_1)= a_0a_1-a_1a_0
$$
for $a_0,\ldots ,a_{n-1}\in A$, $f \in \text{Hom}_{\mathbf{k}} (A^{\otimes n}, A)$.
The  {\em  Hochschild cohomology} is 
$$ HH^n(A) = \ker \delta_{n}/ \im\, \delta_{n-1}$$
for \(n \geq 0\). In particular,
$$ HH^0(A) = \ker \delta_{0} = \{ a \in A \, | \, \delta_0(a) =0 \}= \{ a \in A \, | \, a_0a -aa_0 =0 \quad \text{for all} \quad a_0 \in A\}.$$

The  zeroth Hochschild cohomology $\HH^0(A)$ of a ring $A$ can be identified with the center of $A$:
\[
 \HH^0(A) \cong Z(A) .
\]
The center $Z(A)$ acts naturally on the trace $\Tr A$.  This action can be graphically understood by cutting the diagram of an element $b \in \Tr A$  and inserting the diagram of $a\in Z(A)$:
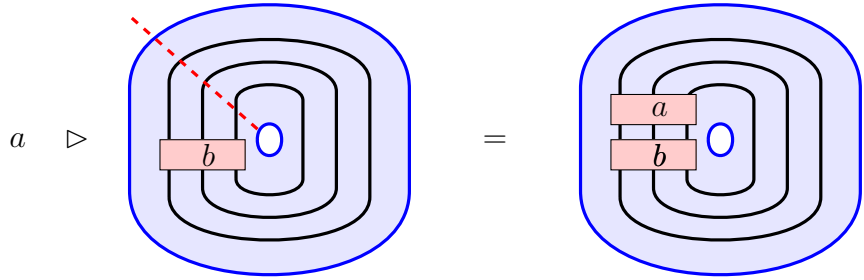
\[ a\quad\rhd\quad
\hackcenter{\begin{tikzpicture}[scale=0.80]
      \path[draw,blue, very thick, fill=blue!10]
        (-2.3,-.6) to (-2.3,1.1) .. controls ++(0,1.85) and ++(0,1.85) .. (2.3,1.1)
         to (2.3,-.6)  .. controls ++(0,-1.85) and ++(0,-1.85) .. (-2.3,-.6);
    \draw[very thick] (-1.65,-.7) -- (-1.65, 1.2).. controls ++(0,.95) and ++(0,.95) .. (1.65,1.2)
        to (1.65,-.7) .. controls ++(0,-.95) and ++(0,-.95) .. (-1.65,-.7);
    \draw[very thick] (-1.1,-.55) -- (-1.1,1.05) .. controls ++(0,.65) and ++(0,.65) .. (1.1,1.05)
        to (1.1,-.55) .. controls ++(0,-.65) and ++(0,-.65) .. (-1.1, -.55);
    \draw[very thick] (-.55,-.4) -- (-.55,.9) .. controls ++(0,.35) and ++(0,.35) .. (.55,.9)
        to (.55, -.4) .. controls ++(0,-.35) and ++(0,-.35) .. (-.55,-.4);
    \draw[fill=red!20,] (-1.8,-.25) rectangle (-.4,.25);
    \node () at (-1,0) {$b$};
   \draw[very thick, red, dashed] (0,.25) -- (-2.3,2.3);
            \path[draw, blue, very thick, fill=white]
            (-0.2,.25) .. controls ++(0,.35) and ++(0,.35) .. (0.2,.25)
            .. controls ++(0,-.35) and ++(0,-.35) .. (-0.2,.25);
\end{tikzpicture}}
\qquad  =\qquad
\hackcenter{\begin{tikzpicture} [scale=0.80]
      \path[draw,blue, very thick, fill=blue!10]
        (-2.3,-.6) to (-2.3,1.1) .. controls ++(0,1.85) and ++(0,1.85) .. (2.3,1.1)
         to (2.3,-.6)  .. controls ++(0,-1.85) and ++(0,-1.85) .. (-2.3,-.6);
    \draw[very thick] (-1.65,-.7) -- (-1.65, 1.2).. controls ++(0,.95) and ++(0,.95) .. (1.65,1.2)
        to (1.65,-.7) .. controls ++(0,-.95) and ++(0,-.95) .. (-1.65,-.7);
    \draw[very thick] (-1.1,-.55) -- (-1.1,1.05) .. controls ++(0,.65) and ++(0,.65) .. (1.1,1.05)
        to (1.1,-.55) .. controls ++(0,-.65) and ++(0,-.65) .. (-1.1, -.55);
    \draw[very thick] (-.55,-.4) -- (-.55,.9) .. controls ++(0,.35) and ++(0,.35) .. (.55,.9)
        to (.55, -.4) .. controls ++(0,-.35) and ++(0,-.35) .. (-.55,-.4);
    \draw[fill=red!20,] (-1.8,-.25) rectangle (-.4,.25);
    \node () at (-1,0) {$b$};
      \draw[fill=red!20,] (-1.8,.5) rectangle (-.4,1);
    \node () at (-1,0) {$b$};\node () at (-1,.75) {$a$};
            \path[draw, blue, very thick, fill=white]
            (-0.2,.25) .. controls ++(0,.35) and ++(0,.35) .. (0.2,.25)
            .. controls ++(0,-.35) and ++(0,-.35) .. (-0.2,.25);
\end{tikzpicture}}
\]

\subsection{Trace of a linear category}
A {\em small category } \(\C\) consists of a set of objects \(\Ob( \C)\) and a set of morphisms  \(\C(x, y)\) for each pair $x,y\in \Ob(\C)$. We assume that categories are small for the rest of this section. A {\em \(\k\)-linear category} is a category in which hom-sets are \(\k\)-linear vector spaces, and compositions of morphisms are bilinear over \(\k\). A \(\k\)-algebra can be understood as a \(\k\)-linear category with one object. We generalize the notion of trace to \(\k\)-linear categories. 
\begin{Definition}
The trace $\Tr \modC $ of the category $\modC$ is defined as
\begin{equation}
 \Tr \C =\left( \bigoplus_{x \in \Ob(\C )} \End_{\C} (x) \right) /\Span_{\k}\{f\circ g-g\circ f\},
\end{equation}
where \(f\) and \(g\) run through all pairs of morphisms \(f \colon x   \to y\), \(g \colon y \to x\) and \(x, y \in \text{Ob}(\mathcal{C})\). 
\end{Definition}
We denote the trace class of a morphism \(f\)  by \([f]\).  A {\em \(\k\)-linear functor} between two \(\k\)-linear categories $\C$ and $\cal{D}$ is a functor  $F\colon \C\to \cal{D}$ such that  the map \(F \colon \C(x,y)\to \cal{D}(F(x), F(y))\) is an abelian group homomorphism for $x,y\in \Ob(\C)$. The trace \(\Tr\) is a \(\k\)-linear functor.

An {\em additive} category is a \(\k\)-linear  category equipped with a zero object and biproducts, also called direct sums. 

\begin{Lemma}
  \label{r14}
  If $\modC $ is an additive category, then for $f \colon x\rightarrow x$ and $g \colon y\rightarrow y$ we have
  \begin{gather*}
    [f\oplus g]=[f]+[g]
  \end{gather*}
in $\Tr \C $.
\end{Lemma}

\begin{proof}
  Since $f\oplus g=(f\oplus 0)+(0\oplus g)\colon x\oplus y\rightarrow x\oplus y$, we have
  \begin{gather*}
    [f\oplus g]=[f\oplus 0]+[0\oplus g].
  \end{gather*}
  Now we have $[f\oplus0]=[ifp]=[pif]=[f]$, where $p\colon x\oplus y\rightarrow x$
  and $i\colon x\rightarrow x\oplus y$ are the projection and the
  inclusion maps. Similarly, $[0\oplus g]=[g]$ holds. 
\end{proof}

\begin{Example}
Let $\C=\cat{Vect}_\k$ be the  category of finite-dimensional $\k$-vector spaces.
Since any finite-dimensional vector space is isomorphic to a finite direct sum of copies of $\k$,
by Lemma~\ref{r14} we have $\Tr \cat{Vect}_\k =\k$.
\end{Example}

\begin{Example}
Let $\C=\cat{grVect}_\k$ be the  category of finite-dimensional $\Z$-graded vector spaces, i.e.
any $V\in\Ob(\C)$ decomposes as $V=\bigoplus_{n\in\Z} V_{n}$ with $\deg(x)=n$ for any $x\in V_n$, and morphisms in $\cat{grVect}_\k$ are degree preserving. Then $\Tr \cat{grVect}_\k=\bigoplus_{n\in\Z}\k=\k [q,q^{-1}]$. The
multiplication by $q$ is interpreted as a shift of the degree by one.
\end{Example}

The previous examples can be further generalized. For a linear category $\modC$ there is a universal additive category generated by $\modC $, called the {\em additive closure} $\modC ^{\oplus}$, in which the objects are formal finite direct sums of objects in $\modC $, and the morphisms are
matrices of morphisms in $\modC$.  There is a canonical fully faithful functor $ \modC \rightarrow \modC ^\oplus$.  Every linear functor $F \colon\modC \rightarrow \modD $  to an additive category $\modD $ factors through $ \modC \rightarrow \modC ^\oplus$ uniquely up to natural isomorphism. We consider the trace of the additive closure $\modC ^{\oplus}$ of a linear category $\modC $.
The homomorphism
\begin{gather*}
  \Tr i \colon\Tr \modC  \rightarrow \Tr \modC ^\oplus
\end{gather*}
induced by the canonical functor $i\colon \modC \rightarrow \modC ^{\oplus}$ is an
isomorphism. The inverse
$\tr\colon \Tr \modC ^\oplus \rightarrow \Tr \modC $ is defined by
\begin{gather*}
  \tr [(f_{k,l})_{k,l}]=\sum_k [f_{k,k}]
\end{gather*}
for an endomorphism in $\modC ^\oplus$
\begin{gather*}
  (f_{k,l})_{k,l\in \{1,\ldots ,n\}} \colon x_1\oplus\dots \oplus x_n\rightarrow x_1\oplus\dots \oplus x_n
\end{gather*}
with $f_{k,l}\colon x_l\rightarrow x_k$ in $\modC $.

%
\subsection{Idempotent completions and split Grothendieck group}

A {\em projection} or {\em idempotent} in $\cat{Vect}_{\k}$ is an endomorphism $p \maps V \to V$ satisfying the relation $p^2=p$.  In this case $(\Id_V-p )\maps V \to V$ is also an idempotent, and together these two idempotents decompose the space $V$ into a direct sum of the image of the projection $p$ and the image of the projection $\Id_V-p$.

More generally, an idempotent in a \(\k\)-linear category $\modC$ is an endomorphism $e \maps x \to x$ satisfying $e^2=e$. An idempotent $e\colon x\rightarrow x$ in $\modC $ is said to {\em split} if there is an object $y$ and morphisms $g\colon x\rightarrow y$, $h\colon y\rightarrow x$ such that $hg=e$ and $gh=1_y$.
The Karoubi envelope $\Kar(\modC)$ is the category whose objects are pairs $(x,e)$ of objects $x\in \Ob(\C)$ and
an idempotent endomorphism $e \colon x\rightarrow x$ in $\mathcal
C$. The morphisms
$$f \colon (x,e)\rightarrow(y,e')$$
are morphisms $f\colon x\rightarrow y$ in $\modC $ such that $f=e'fe$.
Composition is induced by the composition in $\mathcal C$, and the identity morphism is
$e \colon (x,e)\rightarrow(x,e)$. $\Kar(\modC )$ is equipped with a \(\k\)-linear category structure. The Karoubi envelope can be thought of as a minimal enlargement of the category $\modC$ in which all idempotents split.

We can identify $(x,1_x)$ with the object $x$ of $\mathcal{C}$. This identification gives rise to a  natural embedding functor $\iota\colon\mathcal C\rightarrow
\Kar(\mathcal C)$ such that $\iota(x)=(x,1_x)$ for $x\in \Ob(\modC )$ and
$\iota(f\colon x\rightarrow y)=f$.  The Karoubi envelope $\Kar(\modC )$ has
the universality property that if $F\colon \modC \rightarrow \modD $ is
a \(\k\)-linear functor to a \(\k\)-linear category $\modD $ with split idempotents,
then $F$ extends to a functor from $\Kar(\modC)$ to $\modD$ uniquely
up to natural isomorphism. The following proposition illustrates one of the key advantages of the trace, namely its invariance under passing to the Karoubi envelope.

\begin{Proposition}
  \label{r1}
  The map $\Tr \iota \colon\Tr \modC  \longrightarrow \Tr \Kar(\modC )$ induced by $\iota$ is bijective.
\end{Proposition}

\begin{proof}
Recall that an endomorphism $f\maps (x,e) \to (x,e)$ in $\Kar(\modC)$ is just a morphism $f\maps x \to x$ in $\modC$ satisfying the condition that $f=efe$.  Define a map $u \colon\Tr \Kar(\modC  ) \longrightarrow \Tr \modC  $ sending $[f]\in \Tr \Kar(\modC )$ to $[f]\in \Tr \modC $.
Then one can check that $u$ is an inverse to $\Tr \iota $.
\end{proof}

The split Grothendieck group \(K_0(\cal{C})\) of an additive category \(\cal{C}\) is an abelian group generated by isomorphism classes \([x]_{\simeq}\) of objects  \(x \in \text{Ob}(\cal{C})\) modulo the relation \([x\oplus y]_{\simeq} =[x]_{\simeq}+[ y]_{\simeq}\). The split
Grothendieck group $K_0$ defines a functor from additive categories to abelian groups. 

\begin{Example}
$K_0(\cat{Vect}_\k)=\Z$ is generated by $[\k]_{\cong}$.
\end{Example}

Define a homomorphism
\begin{gather*}
  h_\modC \colon K_0(\modC )\longrightarrow \Tr \modC 
\end{gather*}
by
\begin{gather*}
  h_\modC ([x]_{\cong})=[1_x]
\end{gather*}
for $x\in \Ob(\modC )$.  Indeed, one can easily check that
\begin{gather*}
  h_\modC ([x\oplus y]_{\cong})=[1_x]+[1_y] \, ,
\end{gather*}
since $1_{f \oplus g} = 1_f \oplus 1_g$.
The map $h_\modC $ defines a natural transformation
\begin{gather*}
  h\colon K_0 \Rightarrow  \Tr \colon \AdCat\rightarrow \Ab,
\end{gather*}
where $\AdCat$ denotes the category of additive small categories. The map  $h_\C$ is called {\em Chern character map}. It is neither injective nor surjective in general.

A {\em graded \(\k\)-linear category} is a \(\k\)-linear category equipped with an auto-equivalence \(\la 1\ra \colon\C \to \C \). In a graded \(\k\)-linear category, for  objects \(x, y \in \text{Ob}(\C)\) and a morphism \(f \colon x \to y\), there exist objects \(x \la t\ra, y \la t\ra  \in \text{Ob}(\C)\) and a morphism \(f \la t \ra \colon x\la t \ra \to y \la t \ra\) for each \(t \in \Z\), where \(\la t\ra\) is \(\la 1\ra\) applied \(t \in \Z\) times.
The trace and the split Grothendieck groups of a graded linear category \(\C\) are \(\Z[q, q^{-1}]\)-modules such that \([x\la t \ra]_\simeq = q^t [x]_\simeq \), and \([f\la t \ra] = q^t [f] \) for \(x \in \text{Ob}(\C)\) and \(f \in \text{End}(x)\).

A {\em translation} in \(\C\) is a family of natural isomorphisms \(x \simeq x \la t \ra\) for all \(x \in \text{Ob}(\C)\) and \(t \in \Z\).
Given a graded \(\k\)-linear category \(\C\), we can form a category \( \C^{\star} \) with the same objects as \(\C\) and morphisms
\begin{equation} \label{grr}
 \C^{\star} (x, y) = \bigoplus_{t \in \Z}\C^{}(x, y\la t \ra) .
\end{equation}

The category \(\C^{\star} \) admits a translation, since for any \(x \in \text{Ob}(\C)\) and  \(t \in \Z\) the natural isomorphism \(x \simeq x \la t \ra\)  is given by 
\(1_x \in \C(x, x) = \C(x, x\langle t \rangle \langle -t \rangle) \subset \C^{\star} (x, x\langle t \rangle) \) together with the inverse map \(1_x\langle t \rangle \in   \C(x\langle t \rangle, x\langle t \rangle) = \C(x\langle t \rangle, x\langle 0 \rangle \langle t \rangle)  \subset  \C^{\star} (x\langle t \rangle, x)\).

The isomorphism \(x \simeq x \la t \ra\) forgets the \(q\)-grading by setting \(q=1\) on the split Grothendieck group and makes \(K_0(\C^{\star} )\) a \(\Z\)-module. The same is also true for  \(\Tr \C^{\star}  \) since \([f\langle t \rangle] =[f]\) for any endomorphism \( f \colon x \to x\). 

However, hom-spaces of \(\C^{\star} \) contain degree \(t\) morphisms  \(\C^{\star} (x,y\la t \ra)\) for all \(t \in\Z\). Hence \(\Tr \C^{\star} \) decomposes into equivalence classes of morphisms of degree \(t\):
\begin{equation}
\Tr \C^{\star} (x, y) = \bigoplus_{t \in \Z} \Tr \C (x , y \la t \ra).
\end{equation}
Thus, \( \Tr \C^{\star} (x, y)  \) is a \(\Z\)-graded abelian group.

\subsection{2-categories and trace}
\begin{Definition}
A 2-category \(\C\) consists of the following data.
\begin{itemize}
\item A collection of objects.
\item For each pair of objects \(x,y\), a category  \(\C(x,y)\). The objects of  \(\C(x,y)\) are called 1-morphisms from \(x\) to \(y\). For any \(u,u' \in \text{Ob}(\C(x,y))\), morphisms \(f\colon u \to u'\) are called 2-morphisms.
\item There is a composition of 1-morphisms: if \(u \in \text{Ob}(\C(x,y))\) and \(v \in \text{Ob}(\C(y,z))\),  then \(vu \in \text{Ob}(\C(x,z))\). This composition is associative: for any object \(\zeta\) and \(w \in \text{Ob}(\C(z,\zeta))\)
we have 
\[ (wv)u = w(vu).\]
\item There is a horizontal composition of 2-morphisms: if \(u, u' \in \text{Ob}(\C(x,y))\), \(f\colon u \to u'\)  and \(v, v' \in \text{Ob}(\C(y,z))\), \(g\colon v \to v'\),  then there exists a 2-morphism \(g*f \colon vu \to  v'u'\). The horizontal composition is associative.
\item 2-morphisms can also be composed vertically:  if \(u, u' , u'' \in \text{Ob}(\C(x,y))\), \(f\colon u \to u'\)  and, \(f'\colon u' \to u''\), then there is a 2-morphism \(f' f \colon u \to u''\). This composition is associative. Moreover, let \(v, v' , v'' \in \text{Ob}(\C(y,z))\), \(g\colon v \to v'\)  and, \(g'\colon v' \to v''\).  The following rule holds between horizontal and vertical compositions of 2-morphisms:
\[ (g'g) *(f'f) = (g'*f')(g*f).\]
\item For every object \(x\) there is an identity 1-morphism \(1_x \in \text{Ob}(\C(x,x))\) such that for any \(u \in \text{Ob}(\C(x,y))\), we have \(1_y u = u1_x=u\). 
\item For every 1-morphism \(u\in  \text{Ob}(\C(x,y))\) there exists an identity 2-morphism \(1_u\) such that for all \(u' \in \text{Ob}(\C(x,y))\) and \(f\colon u \to u'\), we have \(1_{u'} f = f1_u=f\) and 
\( 1_{1_{y}} * f = f* 1_{1_x} = f\). Additionally, the identity 2-morphisms must be compatible with composition of 1-morphisms: \(1_v * 1_u = 1_{vu}\) for any \(v \in \text{Ob}(\C(y,z))\).
\end{itemize}
\end{Definition}  

\begin{Definition}
A 2-functor \(F\colon \C \to \mathcal{D}\) from a 2-category \(\C\) to a 2-category \(\mathcal{D}\) consists of 
\begin{itemize}
\item a function \(F\colon \text{Ob} (\C) \to \text{Ob} (\mathcal{D}) \), and 
\item for each pair of objects \(x,y\) in \(\C\), a functor \(F \colon \C(x,y) \to \mathcal{D}(Fx, Fy)\).
\end{itemize}
\end{Definition}

\subsection{$\Kar$, $K_0$ and $\Tr$ for 2-categories} \label{sec_traceK0-2cats}
%
 A 2-category \(\C\) is \(\k\)-linear if the categories \(\C(x,y)\) are \(\k\)-linear for all objects \(x,y\) in \(\C\), and the composition functor preserves the \(\k\)-linear structure. Similarly, an additive \(\k\)-linear 2-category is a \(\k\)-linear 2-category in which the categories \(\C(x, y)\) are additive. 

A {\em graded linear 2-category} is a 2-category whose hom-spaces form graded linear categories, and the composition map preserves degree.

The following definitions extend notions of  trace, split Grothendieck group and  Karoubi envelope to the 2-categorical setting.
\begin{itemize}
\item Given an additive \(\k\)-linear 2-category $\C$, define the split Grothendieck group $K_0(\C)$ of $\C$ to be the \(\k\)-linear category with $\Ob(K_0(\C)) = \Ob(\C)$ and with $K_0(\C)(x,y) := K_0(\C(x,y))$
for any two objects $x,y \in \Ob(\C)$.  For $[f]_{\cong} \in
K_0(\C)(x,y)$ and $[g]_{\cong} \in K_0(\C)(y,z) $ the
composition in $K_0(\C)$ is defined by $[g]_{\cong} \circ [f]_{\cong}
:= [g \circ f]_{\cong}$.

\item The trace $\Tr \C$ of a \(\k\)-linear 2-category is the \(\k\)-linear category with $\Ob(\Tr \C) = \Ob(\C)$ and with $\Tr \C(x,y) := \Tr \C(x,y)$ for any two objects $x,y \in \Ob(\C)$. For a 2-endomorphism $\sigma$ in $\C(x,y)$ and a 2-endomorphism $\tau$ in $\C(y,z)$, we have
$[\tau ]\circ[\sigma ]=[\tau \circ\sigma ]$. The identity morphism for $x\in \Ob(\Tr \C)=\Ob(\C)$ is given by
$[1_{1_x}]$.

\item The Karoubi envelope $\Kar(\C)$ of a \(\k\)-linear 2-category $\C$ is the \(\k\)-linear 2-category with $\Ob(\Kar(\C)) = \Ob(\C)$ and with hom-categories $\Kar(\C)(x,y) := \Kar(\C(x,y))$. The composition functor $\Kar(\C)(y,z) \times \Kar(\C)(x,y) \to \Kar(\C)(x,z)$ is induced
by the universal property of the Karoubi envelope from the composition
functor in $\C$.  The fully-faithful additive functors  $\C(x,y) \to
\Kar(\C(x,y))$  form an additive $2$-functor $\C \to \Kar(\C)$ that is universal with respect to splitting idempotents in the  hom-categories $\C(x,y)$.
\end{itemize}

A 2-functor \(F \colon \C \to \mathcal{D} \) between \(\k\)-linear 2-categories \(\C\) and \(\mathcal{D}\) is a \(\k\)-linear 2-functor if for objects \(x, y\) in \(\C\)  the functor \(F \colon \C(x, y) \to \mathcal{D} (x, y)\) is \(\k\)-linear. In this case \(F\) induces a \(\k\)-linear functor \[ \Tr F \colon \Tr \C \to \Tr \mathcal{D}, \] such that 
\[ \Tr F = F \colon \text{Ob}(\Tr \C ) \to \text{Ob}(\Tr \cal{D}) \]
on the objects, and for \(x, y \in \text{Ob}(\cal{C})\)
\[ (\Tr F)_{x,y} = \Tr (F_{x,y}) \colon \Tr \C (x,y)  \to \Tr \cal{D}  (F(x), F(y)) . \]

 Let \(\C\) be a \(\k\)-linear 2-category. The {\em center} \(Z(x)\) of an object \( x \in \text{Ob}(\C)\) is the commutative ring of endomorphisms \(\C(1_x, 1_x)\). We call \(Z(\C) = \bigoplus_{x \in \text{Ob}(\C)} Z(x) \) the {\em center of objects} of \(\C\).

We call  \(\, \C\)  {\em a pivotal 2-category} if for every 1-morphism \(u\colon x \to y\) there exists a biadjoint morphism \(u^* \colon y \to x\), together with 2-morphisms

\begin{align} \label{rea}
  \text{ev}_u \colon u^*u \to 1_x, \quad \text{coev}_u\colon 1_{y} \to uu^*, \\
  \label{ree}
 \tilde{\text{ev}}_u \colon uu^* \to 1_{y}, \quad  \tilde{\text{coev}}_u\colon 1_{x} \to u^*u
\end{align}
 such that biadjointness relations 
\begin{align}\label{rea1}
 \left( \tilde{\text{ev}}_u \otimes 1_u\right) \left( 1_u \otimes \tilde{\text{coev}}_u\right) = \left(1_u  \otimes \text{ev}_u  \right) \left(\text{coev}_u \otimes 1_u \right) = 1_u, \\
\label{rea2}
  \left( \text{ev}_u \otimes 1_{u^*} \right) \left( 1_{u^*}  \otimes \text{coev}_u\right) = \left(1_{u^*}  \otimes \tilde{\text{ev}}_u  \right) \left(\tilde{\text{coev}}_u \otimes 1_{u^*}  \right) = 1_{u^*}
   \end{align}
 hold. 
 Pivotal 2-categories admit rich a diagrammatic presentation in the following way. We denote the objects by areas labeled with the objects. Then 1-morphisms are described by the vertices of the diagrams separating two regions, and 2-morphisms are the edges of the diagrams. The horizontal composition of two diagrams \(A \circ B\) places \(A\) to the right of \(B\) in the plane. The vertical composition \(AB\) of two diagrams means the stacking of \(A\) on top of \(B\). 
 The 2-morphisms \eqref{rea} and \eqref{ree} can be depicted as 

\[  \xy 0;/r.25pc/:
    (0,0)*{\bbcef{}};
    (6,2)*{\scs  x};
    (-5,-4.5)*{u^*};(5,-5)*{u};
    \endxy\, , \quad \quad \xy 0;/r.25pc/:
    (0,-3)*{\bbpfe{}};
    (6,-4)*{\scs y};
   (-5,2)*{u};(5,2.5)*{u^*};
    \endxy \, ; \]  
\[  \xy 0;/r.25pc/:
    (0,0)*{\bbcfe{}};
    (6,2)*{\scs  y};
    (-5,-5.5)*{u};(5,-5)*{u^*};
    \endxy \, , \quad \quad \xy 0;/r.25pc/:
    (0,-3)*{\bbpef{}};
    (6,-4)*{\scs  x};
    (-5,2.5)*{u^*};(5,2)*{u};
    \endxy,\]  
and the relations \eqref{rea1} and \eqref{rea2} are 
 \[  \xy   0;/r.25pc/:
    (8,0)*{}="1";
    (0,0)*{}="2";
    (-8,0)*{}="3";
    (8,-10);"1" **\dir{-};
    "1";"2" **\crv{(8,8) & (0,8)} ?(0)*\dir{>} ?(1)*\dir{>};
    "2";"3" **\crv{(0,-8) & (-8,-8)}?(1)*\dir{>};
    "3"; (-8,10) **\dir{-};
      (10,7)*{x}; (-6,-9)*{y};
       (8,-11)*{\scs u};   (-8,12)*{\scs u};
    \endxy
    \; =
  \xy   0;/r.25pc/:
    (-8,0)*{}="1";
    (0,0)*{}="2";
    (8,0)*{}="3";
    (-8,-10);"1" **\dir{-};
    "1";"2" **\crv{(-8,8) & (0,8)} ?(0)*\dir{>} ?(1)*\dir{>};
    "2";"3" **\crv{(0,-8) & (8,-8)}?(1)*\dir{>};
    "3"; (8,10) **\dir{-};
    (9,-6)*{x};
    (-6,9)*{y};
        (-8,-12)*{\scs u}; (8,12)*{\scs u};
    \endxy
    \; =
    \;
    \xy 0;/r.20pc/:
    (8,0)*{}="1";
    (0,0)*{}="2";
    (-8,0)*{}="3";
    (0,-10);(0,10)**\dir{-} ?(.5)*\dir{>};
   (0,-12)*{\scs u}; (0,12)*{\scs u};
    (6,2)*{x};
    (-6,2)*{y};
    \endxy,
 \] 

\[\xy   0;/r.25pc/:
    (-8,0)*{}="1";
    (0,0)*{}="2";
    (8,0)*{}="3";
    (-8,-10);"1" **\dir{-};
    "1";"2" **\crv{(-8,8) & (0,8)} ?(0)*\dir{<} ?(1)*\dir{<};
    "2";"3" **\crv{(0,-8) & (8,-8)}?(1)*\dir{<};
    "3"; (8,10) **\dir{-};
    (9,-6)*{y};
    (-6,9)*{x};
      (-8,-12)*{\scs u^*}; (8,12)*{\scs u^*};
    \endxy
    \; = \xy  0;/r.25pc/:
    (8,0)*{}="1";
    (0,0)*{}="2";
    (-8,0)*{}="3";
    (8,-10);"1" **\dir{-};
    "1";"2" **\crv{(8,8) & (0,8)} ?(0)*\dir{<} ?(1)*\dir{<};
    "2";"3" **\crv{(0,-8) & (-8,-8)}?(1)*\dir{<};
    "3"; (-8,10) **\dir{-};
    (10,7)*{y};
    (-6,-9)*{x};
      (8,-11)*{\scs u^*};   (-8,12)*{\scs u^*};
    \endxy
    \; =
    \;
    \;
\xy   0;/r.20pc/:
    (-8,0)*{}="1";
    (0,0)*{}="2";
    (8,0)*{}="3";
    (0,-10);(0,10)**\dir{-} ?(.5)*\dir{<};
   (6,2)*{y};
    (-6,2)*{x};
     (0,-12)*{\scs u^*}; (0,12)*{\scs u^*};
  \endxy .\] 

 In a pivotal 2-category \(\C\) we can define the left  and right duals of a given 2-morphism \( f \colon u \to v\) as follows:
 \[ f^* := \left(\text{ev}_v \otimes 1_{u^*} \right)  \left(1_v \otimes f \otimes 1_{u^*} \right)  \left(1_{v^*}\otimes \text{coev}_u   \right) \colon v^{*} \to u^*, \]
 \[ {}^*f : = \left(1_{u^*} \otimes \tilde{\text{ev}}_v   \right)  \left(1_{u^*} \otimes f \otimes 1_{v^*} \right)  \left(\tilde{\text{coev}}_u \otimes 1_{v^*}    \right) \colon v^{*} \to u^*. \]

Diagrammatically, the right and left duals of the 2-morphism
   \[\quad  \xy 0;/r.15pc/:
    (8,0)*{}="1";
    (0,0)*{}="2";
    (-8,0)*{}="3";  (0,0)*{\bullet}; (0,-0.5)*{\bullet};(3,0)*{f}; (0,0.5)*{\bullet};
    (0,-10);(0,10)**\dir{-} ?(0.85)*\dir{>};
   (0,-12)*{\scs u}; (0,12)*{\scs v};
    \endxy\]
  can be depicted as 
   \[\xy 0;/r.20pc/:
    (8,0)*{}="1";
    (0,0)*{}="2";
    (-8,0)*{}="3";  (0,0)*{\bullet}; (0,-0.5)*{\bullet}; (0,0.5)*{\bullet};(4,0)*{f^*};
    (0,-10);(0,10)**\dir{-} ?(0.15)*\dir{<};
   (0,-12)*{\scs v^*}; (0,12)*{\scs u^*};
    \endxy \, = \, \xy   0;/r.20pc/:
    (-8,0)*{}="1";
    (0,0)*{}="2";
    (8,0)*{}="3";
    (-8,-10);"1" **\dir{-}; (0,0)*{\bullet}; (0,-0.5)*{\bullet}; (0,0.5)*{\bullet};(2.5,0)*{f};
    "1";"2" **\crv{(-8,8) & (0,8)} ?(0)*\dir{<} ?(1)*\dir{};
    "2";"3" **\crv{(0,-8) & (8,-8)}?(1)*\dir{<};
    "3"; (8,10) **\dir{-};
      (-8,-12)*{\scs v^*}; (8,12)*{\scs u^*};
    \endxy
    \; ,\quad \quad 
   \xy 0;/r.20pc/:
    (8,0)*{}="1";
    (0,0)*{}="2";
    (-8,0)*{}="3";  (0,0)*{\bullet}; (0,-0.5)*{\bullet}; (0,0.5)*{\bullet};(5,0)*{{}^*f };
    (0,-10);(0,10)**\dir{-} ?(0.15)*\dir{<};
   (0,-12)*{\scs v^*}; (0,12)*{\scs u^*};
    \endxy \, =
    \;
  \xy  0;/r.20pc/:
    (8,0)*{}="1";
    (0,0)*{}="2";
    (-8,0)*{}="3";
    (8,-10);"1" **\dir{-};
    "1";"2" **\crv{(8,8) & (0,8)} ?(0)*\dir{<} ?(1)*\dir{};
    "2";"3" **\crv{(0,-8) & (-8,-8)}?(1)*\dir{<}; (0,0)*{\bullet}; (0,-0.5)*{\bullet}; (0,0.5)*{\bullet};(2.5,0)*{f};
    "3"; (-8,10) **\dir{-};
      (8,-11)*{\scs v^*};   (-8,12)*{\scs u^*};
    \endxy \; .\]   
  
 A pivotal 2-category \(\C\) is called {\em cyclic} if the left and right duals of all 2-morphisms are equal. In a cyclic category \(\C\) there is a well-defined action of the trace \(\Tr \C\) on the center of objects \(Z(\C)\) as follows. Let \(u \colon x \to y\) be a 1-morphism,  \([f]\) denote the trace class of a 2-morphism \(f \colon u \to u\), and  \(a \in Z(x)\). Then we define the map
 \( [f] \colon Z(x) \to Z(y)\),
 
 \[ [f] (a) = \tilde{\text{ev}}_u \circ \left( f \otimes a \otimes 1_{u^*} \right) \circ \text{coev}_u .\]
 In terms of diagrams an element in \( Z(x)\) is a closed diagram at a region labeled by \(x\). The action of the trace corresponds to the closure of diagram \(f\) around \(a\), and this will give us a closed diagram at region \(y\):
 \begin{equation} \label{pp1}
 \quad \xy
  (0,0)*{a}+(-3,-3)*{\ast};
  (6,4)*{x};
 \endxy
 \mapsto  \;\;
  \xy
(-12,-0.4)*{\bullet}; (-12,0.3)*{\bullet};(-16,0)*{f};
 (-12,0)*{};
  (12,0)*{};
  (-12,0)*{}="t1";
  (12,0)*{}="t2";
  "t2";"t1" **\crv{(12,14) & (-12,14)}; ?(0)*\dir{} ?(1)*\dir{};
  "t2";"t1" **\crv{(12,-14) & (-12,-14)}; ?(.3)*\dir{>}+(2,-1)*{\scs };
  (0,0)*{a}+(-3,-3)*{\ast};;
  (14,4)*{y};
 \endxy\;\;,
   \end{equation}

\section{The categorified quantum \(\mathfrak{sl}_n \)}

In this section we will set our notation for \(\mathfrak{sl}_n\) weight lattice. We will then introduce the quantum \(\mathfrak{sl}_n\)  and its diagrammatic categorification $\U_Q(\mathfrak{sl}_n)$  by Khovanov and Lauda. There is an alternative categorification of the quantum \(\mathfrak{sl}_n\), which is given by Rouquier \cite{rouq}. Brundan \cite{brun0} proves the isomorphism between Rouquier's \(2\)-category and $\U_Q(\mathfrak{sl}_n)$. 
\subsection{Cartan datum for $\mathfrak{sl}_n$ }  
Let $I=\{1,2,\dots,n-1\}$ be the set of vertices of the Dynkin diagram of type $A_{n-1}$, \(n \geq 2\):
\begin{equation} \label{eq_dynkin_sln}
    \xy
  (-15,0)*{\circ}="1";
  (-5, 0)*{\circ}="2";
  (5,  0)*{\circ}="3";
  (35,  0)*{\circ}="4";
  "1";"2" **\dir{-}?(.55)*\dir{};
  "2";"3" **\dir{-}?(.55)*\dir{};
  "3";(15,0) **\dir{-}?(.55)*\dir{};
  (25,0);"4" **\dir{-}?(.55)*\dir{};
  (-15,2.2)*{\scs 1};
  (-5,2.2)*{\scs 2};
  (5,2.2)*{\scs 3};
  (35,2.2)*{\scs n-1};
  (20,0)*{\cdots };
  \endxy.
\end{equation}
{\em Simple roots} are defined as \(\alpha_i=(0, \dots, 0, 1, -1, 0,\dots, 0) \in \Z^{n} \),  \(i \in I\), where 1 occurs in the \(i\)-th component. Let  $(\cdot , \cdot) $ be the standard scalar product on $\Z^{n}$.  Then
\[ a_{ij}=(\alpha_i, \alpha_j)=
\left\{
\begin{array}{ll}
  2 & \text{if $i=j$},\\
  -1&  \text{if $|i-j|=1$}, \\
  0 & \text{if $|i-j|>1$}
\end{array}
\right.
\]
are the entries of the Cartan matrix associated to $\mathfrak{sl}_n$.

We will call the elements \(\omega_i=(1,\dots 1, 0, \dots, 0) \in \Z_+^n\), \(i \in I\) with \(i\) number of 1's the {\em fundamental weights}.  Notice that we have \((\omega_i, \alpha_j)= \delta_{ij}\).  Fundamental weights generate the integral \(\mathfrak{sl}_n\)-weight lattice \(X\)  over \(\Z\). 

By  \(\cal{P}(n, N)\) we denote \(n\)-compositions of \(N\) -- the set of \(n\)-tuples of non-negative integers which sum up to a positive integer \(N\), and by \(\cal{P}^+(n, N)\)  the set of \(n\)-partitions of \(N\), that is, elements  \(\l = (\l_1, \l_2, \dots, \l_{n}) \in \cal{P}(n, N)\) such that \(\l_i \geq \l_{i+1}\) for all \(i \in I\). For each composition \(\nu= (\nu_1, \nu_2, \dots, \nu_{n}) \in \cal{P}(n, N)\) there is a corresponding \(\mathfrak{sl}_n\) weight \(\bar{\nu}= \sum_{i \in I}(\nu_i-\nu_{i+1})\omega_i \). Note that \(\nub\) does not define \(\nu\) uniquely unless we fix \(N\). 
We will sometimes abuse the notation and write \(\nub=( \nub_1, \nub_2, \dots , \nub_{n-1}) \in  \Z^{n-1}\), \(\nub_i=\nu_i-\nu_{i+1}\) although \(\nub\) is not a composition. If \(\nub_i \geq 0\) for all \(i \in I\), we call \(\nub\) a dominant integral weight.  If \(\l \in \cal{P}^+(n, N)\), then \(\lb\) is a dominant weight.

Define the composition \( \nu+\alpha_i\in \mathcal{P}(n,N) \)  as 
\[\nu+\alpha_i= (\nu_1, \dots, \nu_{i-1}, \nu_i+1, \nu_{i+1}-1, \nu_{i+2}, \dots, \nu_n) \]
if \(\nu_{i+1} >0\) and \(\emptyset\) otherwise. 
Similarly, \[\nu-\alpha_i= (\nu_1, \dots, \nu_{i-1}, \nu_i-1, \nu_{i+1}+1, \nu_{i+2}, \dots, \nu_n) \]
for \(\nu_i >0\) and \(\emptyset\) if \(\nu_{i} =0\). In this case we have
 
\begin{eqnarray*} 
 \overbar{\nu +\alpha_i} &=& \quad \left\{
\begin{array}{ccl}
   (\overbar{\nu}_1+2, \overbar{\nu}_2-1,\overbar{\nu}_3,\dots,\overbar{\nu}_{n-2}, \overbar{\nu}_{n-1}) & \quad & \text{if $i=1,$} \\
   (\overbar{\nu}_1, \overbar{\nu}_2,\dots,\overbar{\nu}_{n-2},\overbar{\nu}_{n-1}-1, \overbar{\nu}_{n-1}+2) & \quad & \text{if $i=n-1,$} \\
  (\overbar{\nu}_1, \dots, \overbar{\nu}_{i-1}-1, \overbar{\nu}_i+2, \overbar{\nu}_{i+1}-1, \dots,
\overbar{\nu}_{n-1}) & \quad &  \text{if $1<i<n-1$,} 
\end{array}
 \right.
\end{eqnarray*}
 
 \begin{eqnarray*} 
 \overbar{\nu -\alpha_i} &=& \quad \left\{
\begin{array}{ccl}
   (\overbar{\nu}_1-2, \overbar{\nu}_2+1,\overbar{\nu}_3,\dots,\overbar{\nu}_{n-2}, \overbar{\nu}_{n-1}) & \quad & \text{if $i=1,$} \\
   (\overbar{\nu}_1, \overbar{\nu}_2,\dots,\overbar{\nu}_{n-2},\overbar{\nu}_{n-1}+1, \overbar{\nu}_{n-1}-2) & \quad & \text{if $i=n-1,$} \\
  (\overbar{\nu}_1, \dots, \overbar{\nu}_{i-1}+1, \overbar{\nu}_i-2, \overbar{\nu}_{i+1}+1, \dots,
\overbar{\nu}_{n-1}) & \quad &  \text{if $1<i<n-1$.} 
\end{array}
 \right.
\end{eqnarray*}
%
\subsection{The quantum group $\mathbf{U}_q(\mathfrak{sl}_n)$} 
The algebra $\mathbf{U}_q(\mathfrak{sl}_n)$ is the unital \(\Q(q)\)-algebra generated by the elements \(E_i, F_i\) and \(K_i^{}\), \(K_i^{-}\)  for \(i \in I\)  together with the following relations:
\[ K_i^{-1}K_i^{} =K_i^{}K_i^{-1} =1, \quad K_i^{}K_j^{}=K_j^{}K_i^{},\]
\[ K_i^{}E_j K_i^{-1} = q^{a_{ij}} E_j, \quad  K_i^{}F_j K_i^{-1} = q^{-a_{ij}} F_j, \]
\[ E_i F_j -F_jE_i =\delta_{ij} \frac{K_i - K_i^{-1}}{q-q^{-1}},\] 
\[ E_i^2E_j - (q+q^{-1}) E_iE_jE_i + E_j E_i^{2} = 0  \quad  \text{ if } a_{ij} = -1,\]
\[ F_i^2F_j - (q+q^{-1}) F_iF_jF_i + F_j F_i^{2} = 0  \quad \text{ if } a_{ij} = -1,\]
\[ E_iE_j = E_jE_i, \quad F_iF_j = F_jF_i  \quad \text{ if } a_{ij} = 0.\] 
Let $\dot{\mathbf{U}}_q(\mathfrak{sl}_n)$ be Lusztig's idempotented version of $\mathbf{U}_q(\mathfrak{sl}_n)$, where the unit is replaced by a collection of orthogonal idempotents $1_{\nub}$ indexed by the weight lattice $X$ of $\mathfrak{sl}_n$,
\begin{equation}
  1_{\nub}1_{\nub'} = \delta_{\nub\nub'} 1_{\nub},
\end{equation}
such that if \(\nub=( \nub_1, \nub_2, \dots , \nub_{n-1}) \), then
\begin{equation} \label{eq_onesubn}
K_i1_{\nub} =1_{\nub}K_i= q^{\nub_i} 1_{\nub}, \quad
E_i^{}1_{\nub} = 1_{\overbar{\nu +\alpha_i}}E_i, \quad F_i1_{\nub} = 1_{\overbar{\nu -\alpha_i}}F_i
.
\end{equation}

$\dot{\mathbf{U}}_q(\mathfrak{sl}_n)$ can be viewed as a category with objects $\nub \in X$ and morphisms
given  by compositions of $E_i$ and $F_i$ with $1\leq i <n$ modulo the above relations.

For $\cal{A} := \Z[q,q^{-1}]$, the $\cal{A}$-algebra $\UA_q(\mathfrak{sl}_n)$, the integral form of $\dot{\mathbf{U}}_q(\mathfrak{sl}_n)$ is generated by products of divided powers $E^{(a)}_i1_{\nub}:=
\frac{E^{a}_i}{[a]!}1_{\nub}$, $F^{(a)}_i1_{\nub}:=
\frac{F^{a}_i}{[a]!}1_{\nub}$ for $\nub \in X$ and $i \in I$. The subalgebra $\UA^+_q(\mathfrak{sl}_n)$
generated by $E^{(a)}_i1_{\nub}$, $\nub \in X$ is called the {\em positive half} of  $\UA_q(\mathfrak{sl}_n)$.

\subsection{The 2-category $\U_Q(\mathfrak{sl}_n)$}
Here we describe a categorification of  $\mathbf{U}_q(\mathfrak{sl}_n)$  by M. Khovanov and A. Lauda. From now on we assume that the underlying field is complex field. Fix the following choice of scalars $Q$ consisting of $\{ t_{ij}\}_{i,j \in I}$
such that
\begin{itemize}
\item $t_{ii}=1$ for all $i \in I$ and $t_{ij} \in \mathbb{C}^{*}$ for $i\neq j$,
 \item $t_{ij}=t_{ji}$ if $a_{ij}=0$,
\end{itemize}
 and a choice of  bubble parameters $c_{i,\nub} \in \mathbb{C}^{*}$  for $i\in I$ and a weight $\nub$ such that
\[
  c_{j,\overbar{\nu+\alpha_j}}/c_{i,\nub}=t_{ij}.
\]

\begin{Definition} \label{defU_cat}
The 2-category $\, \U_Q(\mathfrak{sl}_n)$ is the graded linear 2-category consisting of:
\begin{itemize}
\item objects are \(\mathfrak{sl}_n\) weights $\bar{\nu}$.
\item 1-morphisms are formal direct sums of (degree shifts of) compositions of
$$\onel, \quad \onenn{\overbar{\nu+\alpha_i}} \sE_i = \onenn{\overbar{\nu+\alpha_i}} \sE_i\onel = \sE_i \onel, \quad \text{ and }\quad \onenn{\overbar{\nu-\alpha_i}} \sF_i = \onenn{\overbar{\nu-\alpha_i}} \sF_i\onel = \sF_i\onel$$
for $i \in I$ and a weight $\bar{\nu} $.
\item 2-morphisms are $\mathbb{C}$-vector spaces spanned by compositions of the following tangle-like diagrams label by \(i \in I\):
\begin{align}
  \xy 0;/r.17pc/:
 (0,7);(0,-7); **\dir{-} ?(.75)*\dir{>};
 (5,3)*{ \scs \bar{\nu}};
 (-9,3)*{\scs  \overbar{\nu+\alpha_i}};
 (-2.5,-6)*{\scs i};
 (-10,0)*{};(10,0)*{};
 \endxy  &\maps \cal{E}_i\onel \to \cal{E}_i\onel,  & \quad
 &
    \xy 0;/r.17pc/:
 (0,7);(0,-7); **\dir{-} ?(.75)*\dir{<};
 (5,3)*{ \scs \nub};
 (-9,3)*{\scs  \overbar{\nu-\alpha_i}};
 (-2.5,-6)*{\scs i};
 (-10,0)*{};(10,0)*{};
 \endxy \maps \cal{F}_i\onel \to \cal{F}_i\onel,  \label{idmorph}\\
   & & & \nn \\
  \xy 0;/r.17pc/:
 (0,7);(0,-7); **\dir{-} ?(.75)*\dir{>};
 (0,0)*{\bullet};
 (5,3)*{ \scs \bar{\nu}};
 (-9,3)*{\scs  \overbar{\nu+\alpha_i}};
 (-2.5,-6)*{\scs i};
 (-10,0)*{};(10,0)*{};
 \endxy &\maps \cal{E}_i\onel \to \cal{E}_i\onel\la 2 \ra,  & \quad
 &
    \xy 0;/r.17pc/:
 (0,7);(0,-7); **\dir{-} ?(.75)*\dir{<};
 (0,0)*{\bullet};
 (5,3)*{ \scs \nub};
 (-9,3)*{\scs  \overbar{\nu-\alpha_i}};
 (-2.5,-6)*{\scs i};
 (-10,0)*{};(10,0)*{};
 \endxy\maps \cal{F}_i\onel \to \cal{F}_i\onel\la 2 \ra,   \\
   & & & \nn \\
   \xy 0;/r.17pc/:
  (0,0)*{\xybox{
    (-4,-4)*{};(4,4)*{} **\crv{(-4,-1) & (4,1)}?(1)*\dir{>} ;
    (4,-4)*{};(-4,4)*{} **\crv{(4,-1) & (-4,1)}?(1)*\dir{>};
    (-5.5,-3)*{\scs i};
     (5.5,-3)*{\scs j};
     (9,1)*{\scs  \nub};
     (-10,0)*{};(10,0)*{};
     }};
  \endxy \;\;&\maps \cal{E}_i\cal{E}_j\onel  \to \cal{E}_j\cal{E}_i\onel\la - a_{ij} \ra,  &
  &
   \xy 0;/r.17pc/:
  (0,0)*{\xybox{
    (-4,4)*{};(4,-4)*{} **\crv{(-4,1) & (4,-1)}?(1)*\dir{>} ;
    (4,4)*{};(-4,-4)*{} **\crv{(4,1) & (-4,-1)}?(1)*\dir{>};
    (-6.5,-3)*{\scs i};
     (6.5,-3)*{\scs j};
     (9,1)*{\scs  \nub};
     (-10,0)*{};(10,0)*{};
     }};
  \endxy\;\; \maps \cal{F}_i\cal{F}_j\onel  \to \cal{F}_j\cal{F}_i\onel\la - a_{ij} \ra,   \\
  & & & \nn \\
     \xy 0;/r.17pc/:
    (0,-3)*{\bbpef{i}};
    (8,-5)*{\scs  \nub};
    (-10,0)*{};(10,0)*{};
    \endxy &\maps \onel  \to \cal{F}_i\cal{E}_i\onel\la 1 + \nub_i \ra,   &
    &
   \xy 0;/r.17pc/:
    (0,-3)*{\bbpfe{i}};
    (8,-5)*{\scs \nub};
    (-10,0)*{};(10,0)*{};
    \endxy \maps \onel  \to\cal{E}_i\cal{F}_i\onel\la 1 - \nub_i \ra,   \\
      & & & \nn \\
  \xy 0;/r.17pc/:
    (0,0)*{\bbcef{i}};
    (8,4)*{\scs  \nub};
    (-10,0)*{};(10,0)*{};
    \endxy & \maps \cal{F}_i\cal{E}_i\onel \to\onel\la 1 + \nub_i \ra,  &
    &
 \xy 0;/r.17pc/:
    (0,0)*{\bbcfe{i}};
    (8,4)*{\scs  \nub};
    (-10,0)*{};(10,0)*{};
    \endxy \maps \cal{E}_i\cal{F}_i\onel  \to\onel \la 1 - \nub_i \ra. 
\end{align}
\end{itemize}
\end{Definition}
 
The two 2-morphisms \eqref{idmorph} are identity 2-morphisms. We read diagrams from right to left and from bottom to top. That is, the horizontal composition of the 2-morphisms corresponds to drawing the respective diagrams side by side from right to left, and vertical composition means stacking diagrams on top of each other. If the labels of the strands do not match in the vertical composition, then we set the composition to zero. Isotopies are allowed as long as they do not change the combinatorial type of a diagram.

The 2-morphisms satisfy the following relations:
\begin{enumerate}
\item \label{item_cycbiadjoint} The identity morphisms of $\cal{E}_i \onel$ and $\cal{F}_i \onel$ are biadjoint up to a degree shift:
    \begin{equation} \label{eq_biadjoint1}
 \xy   0;/r.17pc/:
    (-8,0)*{}="1";
    (0,0)*{}="2";
    (8,0)*{}="3";
    (-8,-10);"1" **\dir{-};
    "1";"2" **\crv{(-8,8) & (0,8)} ?(0)*\dir{>} ?(1)*\dir{>};
    "2";"3" **\crv{(0,-8) & (8,-8)}?(1)*\dir{>};
    "3"; (8,10) **\dir{-};
    (12,-9)*{\nub};
    (-6,9)*{\overbar{\nu+\alpha_i}};
     (-6,-8)*{\scs i};
    \endxy
    \; =
    \;
\xy   0;/r.17pc/:
    (-8,0)*{}="1";
    (0,0)*{}="2";
    (8,0)*{}="3";
    (0,-10);(0,10)**\dir{-} ?(.5)*\dir{>};
    (5,6)*{\nub};
    (-9,6)*{\overbar{\nu+\alpha_i}};
     (-2,-8)*{\scs i};
    \endxy, 
\qquad \quad \xy  0;/r.17pc/:
    (8,0)*{}="1";
    (0,0)*{}="2";
    (-8,0)*{}="3";
    (8,-10);"1" **\dir{-};
    "1";"2" **\crv{(8,8) & (0,8)} ?(0)*\dir{<} ?(1)*\dir{<};
    "2";"3" **\crv{(0,-8) & (-8,-8)}?(1)*\dir{<};
    "3"; (-8,10) **\dir{-};
    (12,9)*{\nub};
    (-6,-9)*{\overbar{\nu-\alpha_i}};
      (10,-8)*{\scs i};
    \endxy
    \; =
    \;
\xy  0;/r.17pc/:
    (8,0)*{}="1";
    (0,0)*{}="2";
    (-8,0)*{}="3";
    (0,-10);(0,10)**\dir{-} ?(.5)*\dir{<};
    (6,6)*{\nub};
    (-9,6)*{\overbar{\nu-\alpha_i}};
    (2,-8)*{\scs i};
    \endxy,
\end{equation}

\begin{equation}\label{eq_biadjoint2}
 \xy   0;/r.17pc/:
    (8,0)*{}="1";
    (0,0)*{}="2";
    (-8,0)*{}="3";
    (8,-10);"1" **\dir{-};
    "1";"2" **\crv{(8,8) & (0,8)} ?(0)*\dir{>} ?(1)*\dir{>};
    "2";"3" **\crv{(0,-8) & (-8,-8)}?(1)*\dir{>};
    "3"; (-8,10) **\dir{-};
    (12,9)*{\nub};
    (-5,-9)*{\overbar{\nu+\alpha_i}};
    (10,-8)*{\scs i};
    \endxy
    \; =
    \;
    \xy 0;/r.17pc/:
    (8,0)*{}="1";
    (0,0)*{}="2";
    (-8,0)*{}="3";
    (0,-10);(0,10)**\dir{-} ?(.5)*\dir{>};
    (5,6)*{\nub};
    (-9,6)*{\overbar{\nu+\alpha_i}};
     (2,-8)*{\scs i};
    \endxy,
\qquad \quad \xy   0;/r.17pc/:
    (-8,0)*{}="1";
    (0,0)*{}="2";
    (8,0)*{}="3";
    (-8,-10);"1" **\dir{-};
    "1";"2" **\crv{(-8,8) & (0,8)} ?(0)*\dir{<} ?(1)*\dir{<};
    "2";"3" **\crv{(0,-8) & (8,-8)}?(1)*\dir{<};
    "3"; (8,10) **\dir{-};
    (10,-7)*{\nub};
    (-6,9)*{\overbar{\nu-\alpha_i}};
      (-6,-8)*{\scs i};
    \endxy
    \; =
    \;
\xy   0;/r.17pc/:
    (-8,0)*{}="1";
    (0,0)*{}="2";
    (8,0)*{}="3";
    (0,-10);(0,10)**\dir{-} ?(.5)*\dir{<};
   (6,5)*{\nub};
    (-9,5)*{\overbar{\nu-\alpha_i}};
     (2,-8)*{\scs i};
    \endxy.
\end{equation}

  \item The 2-morphisms are cyclic with respect to this biadjoint structure:
\begin{equation} \label{eq_cyclic_dot}
 \xy 0;/r.17pc/:
    (-8,5)*{}="1";
    (0,5)*{}="2";
    (0,-5)*{}="2'";
    (8,-5)*{}="3";
    (-8,-10);"1" **\dir{-};
    "2";"2'" **\dir{-} ?(.5)*\dir{<};
    "1";"2" **\crv{(-8,12) & (0,12)} ?(0)*\dir{<};
    "2'";"3" **\crv{(0,-12) & (8,-12)}?(1)*\dir{<};
    "3"; (8,10) **\dir{-};
    (12,-7)*{\nub};
    (-18,8)*{\overbar{\nu-\alpha_i}};
    (0,4)*{\bullet};
    (10,8)*{\scs };
    (-10,-8)*{\scs };
     (-6,-8)*{\scs i};
    \endxy
    \quad = \quad
      \xy 0;/r.17pc/:
 (0,10);(0,-10); **\dir{-} ?(.75)*\dir{<}+(2.3,0)*{\scriptstyle{}}
 ?(.1)*\dir{ }+(2,0)*{\scs };
 (0,0)*{\bullet};
 (-10,7)*{\overbar{\nu-\alpha_i}};
 (5,7)*{\nub};
  (2,-8)*{\scs i};
 (-10,0)*{};(10,0)*{};(-2,-8)*{\scs };
 \endxy
    \quad = \quad
   \xy 0;/r.17pc/:
    (8,5)*{}="1";
    (0,5)*{}="2";
    (0,-5)*{}="2'";
    (-8,-5)*{}="3";
    (8,-10);"1" **\dir{-};
    "2";"2'" **\dir{-} ?(.5)*\dir{<};
    "1";"2" **\crv{(8,12) & (0,12)} ?(0)*\dir{<};
    "2'";"3" **\crv{(0,-12) & (-8,-12)}?(1)*\dir{<};
    "3"; (-8,10) **\dir{-};
    (12,7)*{\nub};
    (-17,-7)*{\overbar{\nu-\alpha_i}};
     (10,-8)*{\scs i};
    (0,4)*{\bullet};
    (-10,8)*{\scs };
    (10,-8)*{\scs };
    \endxy.
\end{equation}
The cyclicity for crossings are given by
\begin{equation} \label{eq_almost_cyclic}
   \xy 0;/r.17pc/:
  (0,0)*{\xybox{
    (-4,4)*{};(4,-4)*{} **\crv{(-4,1) & (4,-1)}?(1)*\dir{>} ;
    (4,4)*{};(-4,-4)*{} **\crv{(4,1) & (-4,-1)}?(1)*\dir{>};
    (-6.5,-3)*{\scs i};
     (6.5,-3)*{\scs j};
     (9,1)*{\scs  \nub};
     (-10,0)*{};(10,0)*{};
     }};
  \endxy \quad = \quad
 \xy 0;/r.17pc/:
  (0,0)*{\xybox{
    (4,-4)*{};(-4,4)*{} **\crv{(4,-1) & (-4,1)}?(1)*\dir{>};
    (-4,-4)*{};(4,4)*{} **\crv{(-4,-1) & (4,1)};
     (-4,4)*{};(18,4)*{} **\crv{(-4,16) & (18,16)} ?(1)*\dir{>};
     (4,-4)*{};(-18,-4)*{} **\crv{(4,-16) & (-18,-16)} ?(1)*\dir{<}?(0)*\dir{<};
     (-18,-4);(-18,12) **\dir{-};(-12,-4);(-12,12) **\dir{-};
     (18,4);(18,-12) **\dir{-};(12,4);(12,-12) **\dir{-};
     (8,1)*{ \nub};
     (-10,0)*{};(10,0)*{};
     (-4,-4)*{};(-12,-4)*{} **\crv{(-4,-10) & (-12,-10)}?(1)*\dir{<}?(0)*\dir{<};
      (4,4)*{};(12,4)*{} **\crv{(4,10) & (12,10)}?(1)*\dir{>}?(0)*\dir{>};
      (-20,11)*{\scs j};(-10,11)*{\scs i};
      (20,-11)*{\scs j};(10,-11)*{\scs i};
     }};
  \endxy
\quad =  
\xy 0;/r.17pc/:
  (0,0)*{\xybox{
    (-4,-4)*{};(4,4)*{} **\crv{(-4,-1) & (4,1)}?(1)*\dir{>};
    (4,-4)*{};(-4,4)*{} **\crv{(4,-1) & (-4,1)};
     (4,4)*{};(-18,4)*{} **\crv{(4,16) & (-18,16)} ?(1)*\dir{>};
     (-4,-4)*{};(18,-4)*{} **\crv{(-4,-16) & (18,-16)} ?(1)*\dir{<}?(0)*\dir{<};
     (18,-4);(18,12) **\dir{-};(12,-4);(12,12) **\dir{-};
     (-18,4);(-18,-12) **\dir{-};(-12,4);(-12,-12) **\dir{-};
     (8,1)*{ \nub};
     (-10,0)*{};(10,0)*{};
      (4,-4)*{};(12,-4)*{} **\crv{(4,-10) & (12,-10)}?(1)*\dir{<}?(0)*\dir{<};
      (-4,4)*{};(-12,4)*{} **\crv{(-4,10) & (-12,10)}?(1)*\dir{>}?(0)*\dir{>};
      (20,11)*{\scs i};(10,11)*{\scs j};
      (-20,-11)*{\scs i};(-10,-11)*{\scs j};
     }};
  \endxy.
\end{equation}
Sideways crossings are defined by the equalities:
\begin{equation} \label{eq_crossl-gen}
  \xy 0;/r.18pc/:
  (0,0)*{\xybox{
    (-4,-4)*{};(4,4)*{} **\crv{(-4,-1) & (4,1)}?(1)*\dir{>} ;
    (4,-4)*{};(-4,4)*{} **\crv{(4,-1) & (-4,1)}?(0)*\dir{<};
    (-5,-3)*{\scs j};
     (6.5,-3)*{\scs i};
     (9,2)*{ \nub};
     (-12,0)*{};(12,0)*{};
     }};
  \endxy
\quad = \quad
 \xy 0;/r.17pc/:
  (0,0)*{\xybox{
    (4,-4)*{};(-4,4)*{} **\crv{(4,-1) & (-4,1)}?(1)*\dir{>};
    (-4,-4)*{};(4,4)*{} **\crv{(-4,-1) & (4,1)};
     (-4,4);(-4,12) **\dir{-};
     (-12,-4);(-12,12) **\dir{-};
     (4,-4);(4,-12) **\dir{-};(12,4);(12,-12) **\dir{-};
     (16,1)*{\nub};
     (-10,0)*{};(10,0)*{};
     (-4,-4)*{};(-12,-4)*{} **\crv{(-4,-10) & (-12,-10)}?(1)*\dir{<}?(0)*\dir{<};
      (4,4)*{};(12,4)*{} **\crv{(4,10) & (12,10)}?(1)*\dir{>}?(0)*\dir{>};
      (-14,11)*{\scs i};(-2,11)*{\scs j};
      (14,-11)*{\scs i};(2,-11)*{\scs j};
     }};
  \endxy
  \quad = \,  \;\;
 \xy 0;/r.17pc/:
  (0,0)*{\xybox{
    (-4,-4)*{};(4,4)*{} **\crv{(-4,-1) & (4,1)}?(1)*\dir{<};
    (4,-4)*{};(-4,4)*{} **\crv{(4,-1) & (-4,1)};
     (4,4);(4,12) **\dir{-};
     (12,-4);(12,12) **\dir{-};
     (-4,-4);(-4,-12) **\dir{-};(-12,4);(-12,-12) **\dir{-};
     (16,1)*{\nub};
     (10,0)*{};(-10,0)*{};
     (4,-4)*{};(12,-4)*{} **\crv{(4,-10) & (12,-10)}?(1)*\dir{>}?(0)*\dir{>};
      (-4,4)*{};(-12,4)*{} **\crv{(-4,10) & (-12,10)}?(1)*\dir{<}?(0)*\dir{<};
     }};
     (12,11)*{\scs j};(0,11)*{\scs i};
      (-17,-11)*{\scs j};(-5,-11)*{\scs i};
  \endxy,
\end{equation}
\begin{equation} \label{eq_crossr-gen}
  \xy 0;/r.18pc/:
  (0,0)*{\xybox{
    (-4,-4)*{};(4,4)*{} **\crv{(-4,-1) & (4,1)}?(0)*\dir{<} ;
    (4,-4)*{};(-4,4)*{} **\crv{(4,-1) & (-4,1)}?(1)*\dir{>};
    (5.1,-3)*{\scs i};
     (-6.5,-3)*{\scs j};
     (9,2)*{ \nub};
     (-12,0)*{};(12,0)*{};
     }};
  \endxy
\quad = \quad
 \xy 0;/r.17pc/:
  (0,0)*{\xybox{
    (-4,-4)*{};(4,4)*{} **\crv{(-4,-1) & (4,1)}?(1)*\dir{>};
    (4,-4)*{};(-4,4)*{} **\crv{(4,-1) & (-4,1)};
     (4,4);(4,12) **\dir{-};
     (12,-4);(12,12) **\dir{-};
     (-4,-4);(-4,-12) **\dir{-};(-12,4);(-12,-12) **\dir{-};
     (16,-6)*{\nub};
     (10,0)*{};(-10,0)*{};
     (4,-4)*{};(12,-4)*{} **\crv{(4,-10) & (12,-10)}?(1)*\dir{<}?(0)*\dir{<};
      (-4,4)*{};(-12,4)*{} **\crv{(-4,10) & (-12,10)}?(1)*\dir{>}?(0)*\dir{>};
      (14,11)*{\scs j};(2,11)*{\scs i};
      (-14,-11)*{\scs j};(-2,-11)*{\scs i};
     }};
  \endxy
  \quad = \,  \;\;
  \xy 0;/r.17pc/:
  (0,0)*{\xybox{
    (4,-4)*{};(-4,4)*{} **\crv{(4,-1) & (-4,1)}?(1)*\dir{<};
    (-4,-4)*{};(4,4)*{} **\crv{(-4,-1) & (4,1)};
     (-4,4);(-4,12) **\dir{-};
     (-12,-4);(-12,12) **\dir{-};
     (4,-4);(4,-12) **\dir{-};(12,4);(12,-12) **\dir{-};
     (16,6)*{\nub};
     (-10,0)*{};(10,0)*{};
     (-4,-4)*{};(-12,-4)*{} **\crv{(-4,-10) & (-12,-10)}?(1)*\dir{>}?(0)*\dir{>};
      (4,4)*{};(12,4)*{} **\crv{(4,10) & (12,10)}?(1)*\dir{<}?(0)*\dir{<};
      (-14,11)*{\scs i};(-2,11)*{\scs j};(14,-11)*{\scs i};(2,-11)*{\scs j};
     }};
  \endxy,
\end{equation}
where the second equalities in \eqref{eq_crossl-gen} and \eqref{eq_crossr-gen} follow from \eqref{eq_almost_cyclic}.

\item The following local relations hold for upwards oriented strands:
\begin{enumerate}[i)]

\item
\begin{equation}
 \vcenter{\xy 0;/r.17pc/:
    (-4,-4)*{};(4,4)*{} **\crv{(-4,-1) & (4,1)}?(1)*\dir{};
    (4,-4)*{};(-4,4)*{} **\crv{(4,-1) & (-4,1)}?(1)*\dir{};
    (-4,4)*{};(4,12)*{} **\crv{(-4,7) & (4,9)}?(1)*\dir{};
    (4,4)*{};(-4,12)*{} **\crv{(4,7) & (-4,9)}?(1)*\dir{};
    (9,6)*{\nub};
    (4,12); (4,13) **\dir{-}?(1)*\dir{>};
    (-4,12); (-4,13) **\dir{-}?(1)*\dir{>};
  (-5.5,-3)*{\scs i};
     (5.5,-3)*{\scs j};
 \endxy}
 \qquad = \qquad
 \left\{
 \begin{array}{ccc}
 0 &  &  \text{if $a_{ij}=2$,}\\ \\
 
     t_{ij}\;\xy 0;/r.17pc/:
  (3,9);(3,-9) **\dir{-}?(0)*\dir{<}+(2.3,0)*{};
  (-3,9);(-3,-9) **\dir{-}?(0)*\dir{<}+(2.3,0)*{};
  (9,3)*{\nub};
  (-5,-6)*{\scs i};     (5.1,-6)*{\scs j};
 \endxy &  &  \text{if $a_{ij}=0$,}\\ \\
 t_{ij} \vcenter{\xy 0;/r.17pc/:
  (3,9);(3,-9) **\dir{-}?(0)*\dir{<}+(2.3,0)*{};
  (-3,9);(-3,-9) **\dir{-}?(0)*\dir{<}+(2.3,0)*{};
   (9,3)*{\nub};
  (-3,4)*{\bullet};(-6.5,5)*{};
  (-5,-6)*{\scs i};     (5.1,-6)*{\scs j};
 \endxy} \;\; + \;\; t_{ji}
  \vcenter{\xy 0;/r.17pc/:
  (3,9);(3,-9) **\dir{-}?(0)*\dir{<}+(2.3,0)*{};
  (-3,9);(-3,-9) **\dir{-}?(0)*\dir{<}+(2.3,0)*{};
    (9,3)*{\nub};
  (3,4)*{\bullet};(7,5)*{};
  (-5,-6)*{\scs i};     (5.1,-6)*{\scs j};
 \endxy}
   &  & \text{if $a_{ij} =-1$.}
 \end{array}
 \right. \label{eq_r2_ij-gen}
\end{equation}

\item The nilHecke dot sliding relations
\begin{equation} \label{eq_dot_slide_ii-gen-cyc}
\xy 0;/r.22pc/:
  (0,0)*{\xybox{
    (-4,-4)*{};(4,6)*{} **\crv{(-4,-1) & (4,1)}?(1)*\dir{>}?(.25)*{\bullet};
    (4,-4)*{};(-4,6)*{} **\crv{(4,-1) & (-4,1)}?(1)*\dir{>};
    (-5,-3)*{\scs i};
     (5.1,-3)*{\scs i};
   (9,3)*{\nub};
     (-10,0)*{};(10,0)*{};
     }};
  \endxy
 \;\; -
\xy 0;/r.22pc/:
  (0,0)*{\xybox{
    (-4,-4)*{};(4,6)*{} **\crv{(-4,-1) & (4,1)}?(1)*\dir{>}?(.75)*{\bullet};
    (4,-4)*{};(-4,6)*{} **\crv{(4,-1) & (-4,1)}?(1)*\dir{>};
    (-5,-3)*{\scs i};
     (5.1,-3)*{\scs i};
     (9,3)*{\nub};
     (-10,0)*{};(10,0)*{};
     }};
  \endxy
\;\; =  \xy 0;/r.22pc/:
  (0,0)*{\xybox{
    (-4,-4)*{};(4,6)*{} **\crv{(-4,-1) & (4,1)}?(1)*\dir{>};
    (4,-4)*{};(-4,6)*{} **\crv{(4,-1) & (-4,1)}?(1)*\dir{>}?(.75)*{\bullet};
    (-5,-3)*{\scs i};
     (5.1,-3)*{\scs i};
     (9,3)*{\nub};
     (-10,0)*{};(10,0)*{};
     }};
  \endxy
\;\;  -
  \xy 0;/r.22pc/:
  (0,0)*{\xybox{
    (-4,-4)*{};(4,6)*{} **\crv{(-4,-1) & (4,1)}?(1)*\dir{>} ;
    (4,-4)*{};(-4,6)*{} **\crv{(4,-1) & (-4,1)}?(1)*\dir{>}?(.25)*{\bullet};
    (-5,-3)*{\scs i};
     (5.1,-3)*{\scs i};
     (9,3)*{\nub};
     (-10,0)*{};(12,0)*{};
     }};
  \endxy
  \;\; = \;\;
    \xy 0;/r.22pc/:
  (0,0)*{\xybox{
    (-4,-4)*{};(-4,6)*{} **\dir{-}?(1)*\dir{>} ;
    (4,-4)*{};(4,6)*{} **\dir{-}?(1)*\dir{>};
    (-5,-3)*{\scs i};
      (9,3)*{\nub};
     (5.1,-3)*{\scs i};
          (-10,0)*{};(12,0)*{};
     }};
  \endxy
\end{equation}
hold.

\item For $i \neq j$ the dot sliding relations
\begin{equation} \label{eq_dot_slide_ij-gen}
\xy 0;/r.22pc/:
  (0,0)*{\xybox{
    (-4,-4)*{};(4,6)*{} **\crv{(-4,-1) & (4,1)}?(1)*\dir{>}?(.75)*{\bullet};
    (4,-4)*{};(-4,6)*{} **\crv{(4,-1) & (-4,1)}?(1)*\dir{>};
    (-5,-3)*{\scs i};
     (5.1,-3)*{\scs j};
      (9,3)*{\nub};
     (-10,0)*{};(10,0)*{};
     }};
  \endxy
 \;\; =
\xy 0;/r.22pc/:
  (0,0)*{\xybox{
    (-4,-4)*{};(4,6)*{} **\crv{(-4,-1) & (4,1)}?(1)*\dir{>}?(.25)*{\bullet};
    (4,-4)*{};(-4,6)*{} **\crv{(4,-1) & (-4,1)}?(1)*\dir{>};
    (-5,-3)*{\scs i};
     (5.1,-3)*{\scs j};
       (9,3)*{\nub};
     (-10,0)*{};(10,0)*{};
     }};
  \endxy,
\qquad  \xy 0;/r.22pc/:
  (0,0)*{\xybox{
    (-4,-4)*{};(4,6)*{} **\crv{(-4,-1) & (4,1)}?(1)*\dir{>};
    (4,-4)*{};(-4,6)*{} **\crv{(4,-1) & (-4,1)}?(1)*\dir{>}?(.75)*{\bullet};
    (-5,-3)*{\scs i};
     (5.1,-3)*{\scs j};
     (9,3)*{\nub};
     (-10,0)*{};(10,0)*{};
     }};
  \endxy
\;\;  =
  \xy 0;/r.22pc/:
  (0,0)*{\xybox{
    (-4,-4)*{};(4,6)*{} **\crv{(-4,-1) & (4,1)}?(1)*\dir{>} ;
    (4,-4)*{};(-4,6)*{} **\crv{(4,-1) & (-4,1)}?(1)*\dir{>}?(.25)*{\bullet};
    (-5,-3)*{\scs i};
     (5.1,-3)*{\scs j};
  (9,3)*{\nub};
     (-10,0)*{};(12,0)*{};
     }};
  \endxy
\end{equation}
hold.

\item If $i \neq k$ and $a_{ij} \geq 0$, the relation
\begin{equation}
\vcenter{\xy 0;/r.17pc/:
    (-4,-4)*{};(4,4)*{} **\crv{(-4,-1) & (4,1)}?(1)*\dir{};
    (4,-4)*{};(-4,4)*{} **\crv{(4,-1) & (-4,1)}?(1)*\dir{};
    (4,4)*{};(12,12)*{} **\crv{(4,7) & (12,9)}?(1)*\dir{};
    (12,4)*{};(4,12)*{} **\crv{(12,7) & (4,9)}?(1)*\dir{};
    (-4,12)*{};(4,20)*{} **\crv{(-4,15) & (4,17)}?(1)*\dir{};
    (4,12)*{};(-4,20)*{} **\crv{(4,15) & (-4,17)}?(1)*\dir{};
    (-4,4)*{}; (-4,12) **\dir{-};
    (12,-4)*{}; (12,4) **\dir{-};
    (12,12)*{}; (12,20) **\dir{-};
    (4,20); (4,21) **\dir{-}?(1)*\dir{>};
    (-4,20); (-4,21) **\dir{-}?(1)*\dir{>};
    (12,20); (12,21) **\dir{-}?(1)*\dir{>};
  (18,12)*{\nub}; 
   (-6,-3)*{\scs i};
  (6,-3)*{\scs j};
  (15,-3)*{\scs k};
\endxy}
 \;\; =\;\;
\vcenter{\xy 0;/r.17pc/:
    (4,-4)*{};(-4,4)*{} **\crv{(4,-1) & (-4,1)}?(1)*\dir{};
    (-4,-4)*{};(4,4)*{} **\crv{(-4,-1) & (4,1)}?(1)*\dir{};
    (-4,4)*{};(-12,12)*{} **\crv{(-4,7) & (-12,9)}?(1)*\dir{};
    (-12,4)*{};(-4,12)*{} **\crv{(-12,7) & (-4,9)}?(1)*\dir{};
    (4,12)*{};(-4,20)*{} **\crv{(4,15) & (-4,17)}?(1)*\dir{};
    (-4,12)*{};(4,20)*{} **\crv{(-4,15) & (4,17)}?(1)*\dir{};
    (4,4)*{}; (4,12) **\dir{-};
    (-12,-4)*{}; (-12,4) **\dir{-};
    (-12,12)*{}; (-12,20) **\dir{-};
    (4,20); (4,21) **\dir{-}?(1)*\dir{>};
    (-4,20); (-4,21) **\dir{-}?(1)*\dir{>};
    (-12,20); (-12,21) **\dir{-}?(1)*\dir{>};
  (10,12)*{\nub};
  (-14,-3)*{\scs i};
  (-6,-3)*{\scs j};
  (6,-3)*{\scs k};
\endxy}
 \label{eq_r3_easy-gen}
\end{equation}
holds. Otherwise if $a_{ij} =-1$, we have
\begin{equation}
\vcenter{\xy 0;/r.17pc/:
    (-4,-4)*{};(4,4)*{} **\crv{(-4,-1) & (4,1)}?(1)*\dir{};
    (4,-4)*{};(-4,4)*{} **\crv{(4,-1) & (-4,1)}?(1)*\dir{};
    (4,4)*{};(12,12)*{} **\crv{(4,7) & (12,9)}?(1)*\dir{};
    (12,4)*{};(4,12)*{} **\crv{(12,7) & (4,9)}?(1)*\dir{};
    (-4,12)*{};(4,20)*{} **\crv{(-4,15) & (4,17)}?(1)*\dir{};
    (4,12)*{};(-4,20)*{} **\crv{(4,15) & (-4,17)}?(1)*\dir{};
    (-4,4)*{}; (-4,12) **\dir{-};
    (12,-4)*{}; (12,4) **\dir{-};
    (12,12)*{}; (12,20) **\dir{-};
    (4,20); (4,21) **\dir{-}?(1)*\dir{>};
    (-4,20); (-4,21) **\dir{-}?(1)*\dir{>};
    (12,20); (12,21) **\dir{-}?(1)*\dir{>};
   (20,12)*{\nub};
    (-6,-3)*{\scs i};
  (6,-3)*{\scs j};
  (15,-3)*{\scs i};
\endxy}
\quad - \quad
\vcenter{\xy 0;/r.17pc/:
    (4,-4)*{};(-4,4)*{} **\crv{(4,-1) & (-4,1)}?(1)*\dir{};
    (-4,-4)*{};(4,4)*{} **\crv{(-4,-1) & (4,1)}?(1)*\dir{};
    (-4,4)*{};(-12,12)*{} **\crv{(-4,7) & (-12,9)}?(1)*\dir{};
    (-12,4)*{};(-4,12)*{} **\crv{(-12,7) & (-4,9)}?(1)*\dir{};
    (4,12)*{};(-4,20)*{} **\crv{(4,15) & (-4,17)}?(1)*\dir{};
    (-4,12)*{};(4,20)*{} **\crv{(-4,15) & (4,17)}?(1)*\dir{};
    (4,4)*{}; (4,12) **\dir{-};
    (-12,-4)*{}; (-12,4) **\dir{-};
    (-12,12)*{}; (-12,20) **\dir{-};
    (4,20); (4,21) **\dir{-}?(1)*\dir{>};
    (-4,20); (-4,21) **\dir{-}?(1)*\dir{>};
    (-12,20); (-12,21) **\dir{-}?(1)*\dir{>};
  (10,12)*{\nub};
  (-14,-3)*{\scs i};
  (-6,-3)*{\scs j};
  (6,-3)*{\scs i};
\endxy}
 \;\; =\;\;
 t_{ij} \;\;
\xy 0;/r.17pc/:
  (4,12);(4,-12) **\dir{-}?(0)*\dir{<};
  (-4,12);(-4,-12) **\dir{-}?(0)*\dir{<}?(.25)*\dir{};
  (12,12);(12,-12) **\dir{-}?(0)*\dir{<}?(.25)*\dir{};
  (18,3)*{\nub};
  (-6,-9)*{\scs i};     (6.1,-9)*{\scs j};
  (14,-9)*{\scs i};
 \endxy.
 \label{eq_r3_hard-gen}
\end{equation}
\end{enumerate}

\item When $i \ne j$ one has the mixed relations
\begin{equation} \label{mixed_rel}
 \vcenter{   \xy 0;/r.18pc/:
    (-4,-4)*{};(4,4)*{} **\crv{(-4,-1) & (4,1)}?(1)*\dir{>};
    (4,-4)*{};(-4,4)*{} **\crv{(4,-1) & (-4,1)}?(1)*\dir{<};?(0)*\dir{<};
    (-4,4)*{};(4,12)*{} **\crv{(-4,7) & (4,9)};
    (4,4)*{};(-4,12)*{} **\crv{(4,7) & (-4,9)}?(1)*\dir{>};
  (9,7)*{\nub};
  (-6,-3)*{\scs i};
     (6,-3)*{\scs j};
 \endxy}
 \;\; = \;\; 
\xy 0;/r.18pc/:
  (3,9);(3,-9) **\dir{-}?(.55)*\dir{>}+(2.3,0)*{};
  (-3,9);(-3,-9) **\dir{-}?(.5)*\dir{<}+(2.3,0)*{};
  (9,3)*{\nub};
  (-5,-6)*{\scs i};     (5.1,-6)*{\scs j};
 \endxy,
\qquad \quad
    \vcenter{\xy 0;/r.18pc/:
    (-4,-4)*{};(4,4)*{} **\crv{(-4,-1) & (4,1)}?(1)*\dir{<};?(0)*\dir{<};
    (4,-4)*{};(-4,4)*{} **\crv{(4,-1) & (-4,1)}?(1)*\dir{>};
    (-4,4)*{};(4,12)*{} **\crv{(-4,7) & (4,9)}?(1)*\dir{>};
    (4,4)*{};(-4,12)*{} **\crv{(4,7) & (-4,9)};
   (9,7)*{\nub};
  (-6,-3)*{\scs i};
     (6,-3)*{\scs j};
 \endxy}
 \;\;=\;\;  
\xy 0;/r.18pc/:
  (3,9);(3,-9) **\dir{-}?(.5)*\dir{<}+(2.3,0)*{};
  (-3,9);(-3,-9) **\dir{-}?(.55)*\dir{>}+(2.3,0)*{};
  (9,3)*{\nub};
  (-5,-6)*{\scs i};     (5.1,-6)*{\scs j};
 \endxy.
\end{equation}

\item \label{item_positivity} Negative degree bubbles are zero. That is, for all $m \in \Z_+$ one has
\begin{equation} \label{eq_positivity_bubbles}
\xy 0;/r.18pc/:
 (-12,0)*{\icbub{m}{i}};
 (-8,8)*{\nub};
 \endxy
  = 0
 \qquad  \text{if $m<\nub_i-1$,} \qquad \xy 0;/r.18pc/: (-12,0)*{\iccbub{m}{i}};
 (-8,8)*{\nub};
 \endxy = 0\quad
  \text{if $m< -\nub_i-1$.}
\end{equation}
A dotted bubble of degree zero is just the scaled identity 2-morphism:

\[
\xy 0;/r.18pc/:
 (0,0)*{\icbub{\nub_i-1}{i}};
  (4,8)*{\nub};
 \endxy
  = c_{i, \nub} \, \Id_{\onenn{\nub}} \quad \text{for $\nub_i \geq 1$,}
  \qquad \quad
  \xy 0;/r.18pc/:
 (0,0)*{\iccbub{-\nub_i-1}{i}};
  (4,8)*{\nub};
 \endxy  = c^{-1}_{i, \nub} \, \Id_{\onenn{\nub}} \quad \text{for $\nub_i \leq -1$.}\]

\item \label{item_highersl2} 
It is convenient to introduce so called fake bubbles. These diagrams are dotted bubbles where the number of dots is negative, but the total degree of the dotted bubble taken with these negative dots is still positive.  
\begin{itemize}
 \item Degree zero fake bubbles are equal to the identity 2-morphisms
\[
 \xy 0;/r.18pc/:
    (2,0)*{\icbub{\nub_i-1}{i}};
  (10, 6)*{\nub};
 \endxy
  =  c_{i, \nub} \, \Id_{\onenn{\nub}} \quad \text{if $\nub_i \leq 0$,}
  \qquad \quad
\xy 0;/r.18pc/:
    (2,0)*{\iccbub{-\nub_i-1}{i}};
   (10, 6)*{\nub};
 \endxy = c^{-1}_{i, \nub}\, \Id_{\onenn{\nub}} \quad  \text{if $\nub_i \geq 0$}.\]

 \item Higher degree fake bubbles for $\nub_i<0$ are defined inductively as
  \begin{equation} \label{eq_fake_nleqz}
  \vcenter{\xy 0;/r.18pc/:
    (2,-11)*{\icbub{\nub_i-1+j}{i}};
  (10,-4)*{\nub};
 \endxy} \;\; =
 \left\{
 \begin{array}{cl}
  \;\; -\;c_{i, \nub} \,
\xsum{\xy (0,6)*{};  (0,1)*{\scs a+b=j}; (0,-2)*{\scs b\geq 1}; 
\endxy}
\;\; \vcenter{\xy 0;/r.18pc/:
    (2,0)*{\cbub{\nub_i-1+a}{}};
    (20,0)*{\ccbub{-\nub_i-1+b}{}};
  (12,8)*{\nub};
 \endxy}  & \text{if $0 \leq j < -\nub_i+1$} \\ & \\
   0 & \text{if $j < 0$. }
 \end{array}
\right.
 \end{equation}

  \item Higher degree fake bubbles for $\nub_i>0$ are defined inductively as
   \begin{equation} \label{eq_fake_ngeqz}
  \vcenter{\xy 0;/r.18pc/:
    (2,-11)*{\iccbub{-\nub_i-1+j}{i}};
  (10,-4)*{\nub};
 \endxy} \;\; =
 \left\{
 \begin{array}{cl}
  \;\; -\;c^{-1}_{i, \nub} \,
\xsum{\xy (0,6)*{}; (0,1)*{\scs a+b=j}; (0,-2)*{\scs a\geq 1}; \endxy}
\;\; \vcenter{\xy 0;/r.18pc/:
    (2,0)*{\cbub{\nub_i-1+a}{}};
    (20,0)*{\ccbub{-\nub_i-1+b}{}};
  (12,8)*{\nub};
 \endxy}  & \text{if $0 \leq j < \nub_i+1$} \\ & \\
   0 & \text{if $j < 0$. }
 \end{array}
\right.
\end{equation}
\end{itemize}
These equations arise from the homogeneous terms in $t$ of the Grassmannian equation
\begin{center}
 \begin{eqnarray}
 \makebox[0pt]{ $
\left( \xy 0;/r.15pc/:
 (0,0)*{\iccbub{-\nub_i-1}{i}};
  (4,8)*{\nub};
 \endxy
 +
 \xy 0;/r.15pc/:
 (0,0)*{\iccbub{-\nub_i-1+1}{i}};
  (4,8)*{\nub};
 \endxy t
 + \cdots +
\xy 0;/r.15pc/:
 (0,0)*{\iccbub{-\nub_i-1+q}{i}};
  (4,8)*{\nub};
 \endxy t^q
 + \cdots
\right)
\left(\xy 0;/r.15pc/:
 (0,0)*{\icbub{\nub_i-1}{i}};
  (4,8)*{\nub};
 \endxy
 + \xy 0;/r.15pc/:
 (0,0)*{\icbub{\nub_i-1+1}{i}};
  (4,8)*{\nub};
 \endxy t
 +\cdots +
\xy 0;/r.15pc/:
 (0,0)*{\icbub{\nub_i-1+p}{i}};
 (4,8)*{\nub};
 \endxy t^{p}
 + \cdots
\right) = \Id_{\onel}.$ } \nn \\ \label{eq_infinite_Grass}
\end{eqnarray}
\end{center}
 
Using the Grassmannian relation, we can express clockwise bubbles in terms of fake and real counterclockwise bubbles.  We will use the following notation in order to emphasize the degree of bubble: 
\[
    \xy 0;/r.18pc/:
  (4,8)*{\nub};
  (2,-2)*{\icbub{\spadesuit+r}{i}};
 \endxy \; := \;
   \xy 0;/r.18pc/:
  (4,8)*{\nub};
  (2,-2)*{\icbub{\nub_i-1+r}{i}};
 \endxy,
 \qquad
 \qquad
    \xy 0;/r.18pc/:
  (4,8)*{\nub};
  (2,-2)*{\iccbub{\spadesuit+r}{i}};
 \endxy \; := \;
   \xy 0;/r.18pc/:
  (4,8)*{\nub};
  (2,-2)*{\iccbub{-\nub_i-1+r}{i}};
 \endxy.
\]

\item The following set of relations are sometimes called extended \(\mathfrak{sl}_2\) relations:
\begin{equation} \label{eq_curl}
  \xy 0;/r.17pc/:
  (14,8)*{\nub};
  (-3,-10)*{};(3,5)*{} **\crv{(-3,-2) & (2,1)}?(1)*\dir{>};?(.15)*\dir{>};
    (3,-5)*{};(-3,10)*{} **\crv{(2,-1) & (-3,2)}?(.85)*\dir{>} ?(.1)*\dir{>};
  (3,5)*{}="t1";  (9,5)*{}="t2";
  (3,-5)*{}="t1'";  (9,-5)*{}="t2'";
   "t1";"t2" **\crv{(4,8) & (9, 8)};
   "t1'";"t2'" **\crv{(4,-8) & (9, -8)};
   "t2'";"t2" **\crv{(10,0)} ;
   (-6,-8)*{\scs i};
 \endxy \;\; = \;\; -
   \sum_{ \xy  (0,3)*{\scs f_1+f_2}; (0,0)*{\scs =-1};\endxy}
 \xy 0;/r.17pc/:
  (19,4)*{\nub};
  (0,0)*{\bbe{}};(-2,-8)*{\scs };
  (-2,-8)*{\scs i};
  (12,-2)*{\icbub{f_2}{i}};
  (0,6)*{\bullet}+(3,-1)*{\scs f_1};
 \endxy,
\qquad
  \xy 0;/r.17pc/:
  (-14,8)*{\nub};
  (3,-10)*{};(-3,5)*{} **\crv{(3,-2) & (-2,1)}?(1)*\dir{>};?(.15)*\dir{>};
    (-3,-5)*{};(3,10)*{} **\crv{(-2,-1) & (3,2)}?(.85)*\dir{>} ?(.1)*\dir{>};
  (-3,5)*{}="t1";  (-9,5)*{}="t2";
  (-3,-5)*{}="t1'";  (-9,-5)*{}="t2'";
   "t1";"t2" **\crv{(-4,8) & (-9, 8)};
   "t1'";"t2'" **\crv{(-4,-8) & (-9, -8)};
   "t2'";"t2" **\crv{(-10,0)} ;
   (6,-8)*{\scs i};
 \endxy \;\; = \;\;
 \sum_{ \xy  (0,3)*{\scs g_1+g_2}; (0,0)*{\scs =-1};\endxy}
  \xy 0;/r.17pc/:
  (5,-8)*{\scs i};
  (-12,8)*{\nub};
  (3,0)*{\bbe{}};(2,-8)*{\scs};
  (-12,-2)*{\iccbub{g_2}{i}};
  (3,6)*{\bullet}+(3,-1)*{\scs g_1};
 \endxy,
\end{equation}

\begin{equation}
\vcenter{   \xy 0;/r.17pc/:
    (-4,-4)*{};(4,4)*{} **\crv{(-4,-1) & (4,1)}?(1)*\dir{>};
    (4,-4)*{};(-4,4)*{} **\crv{(4,-1) & (-4,1)}?(1)*\dir{<};?(0)*\dir{<};
    (-4,4)*{};(4,12)*{} **\crv{(-4,7) & (4,9)};
    (4,4)*{};(-4,12)*{} **\crv{(4,7) & (-4,9)}?(1)*\dir{>};
  (8,8)*{\nub};
     (-6,-3)*{\scs i};
     (6.5,-3)*{\scs i};
 \endxy}
\;\; = \;\; -\;
  \vcenter{\xy 0;/r.17pc/:
  (-8,0)*{};
  (8,0)*{};
  (-4,10)*{}="t1";
  (4,10)*{}="t2";
  (-4,-10)*{}="b1";
  (4,-10)*{}="b2";(-6,-8)*{\scs i};(6,-8)*{\scs i};
  "t1";"b1" **\dir{-} ?(.5)*\dir{<};
  "t2";"b2" **\dir{-} ?(.5)*\dir{>};
  (10,2)*{\nub};
  \endxy}
  \;\; + \;\;
   \sum_{ \xy  (0,3)*{\scs f_1+f_2+f_3}; (0,0)*{\scs =-2};\endxy}
    \vcenter{\xy 0;/r.17pc/:
    (-12,10)*{\nub};
    (-8,0)*{};
  (8,0)*{};
  (-4,-15)*{}="b1";
  (4,-15)*{}="b2";
  "b2";"b1" **\crv{(5,-8) & (-5,-8)}; ?(.05)*\dir{<} ?(.93)*\dir{<}
  ?(.8)*\dir{}+(0,-.1)*{\bullet}+(-3,2)*{\scs f_3};
  (-4,15)*{}="t1";
  (4,15)*{}="t2";
  "t2";"t1" **\crv{(5,8) & (-5,8)}; ?(.15)*\dir{>} ?(.95)*\dir{>}
  ?(.4)*\dir{}+(0,-.2)*{\bullet}+(3,-2)*{\scs \; f_1};
  (0,0)*{\iccbub{\scs \quad f_2}{i}};
  (7,-13)*{\scs i};
  (-7,13)*{\scs i};
  \endxy}, \label{eq_ident_decomp-ngeqz}
\end{equation}

\begin{equation}
\vcenter{\xy 0;/r.17pc/:
    (-4,-4)*{};(4,4)*{} **\crv{(-4,-1) & (4,1)}?(1)*\dir{<};?(0)*\dir{<};
    (4,-4)*{};(-4,4)*{} **\crv{(4,-1) & (-4,1)}?(1)*\dir{>};
    (-4,4)*{};(4,12)*{} **\crv{(-4,7) & (4,9)}?(1)*\dir{>};
    (4,4)*{};(-4,12)*{} **\crv{(4,7) & (-4,9)};
  (8,8)*{\nub};(-6.5,-3)*{\scs i};  (6,-3)*{\scs i};
 \endxy}
\;\; = \;\;
  -\;\;
  \vcenter{\xy 0;/r.17pc/:
  (-8,0)*{};(-6,-8)*{\scs i};(6,-8)*{\scs i};
  (8,0)*{};
  (-4,10)*{}="t1";
  (4,10)*{}="t2";
  (-4,-10)*{}="b1";
  (4,-10)*{}="b2";
  "t1";"b1" **\dir{-} ?(.5)*\dir{>};
  "t2";"b2" **\dir{-} ?(.5)*\dir{<};
  (10,2)*{\nub};
  \endxy}
  \;\; + \;\;
    \sum_{ \xy  (0,3)*{\scs g_1+g_2+g_3}; (0,0)*{\scs =-2};\endxy}
    \vcenter{\xy 0;/r.17pc/:
    (-8,0)*{};
  (8,0)*{};
  (-4,-15)*{}="b1";
  (4,-15)*{}="b2";
  "b2";"b1" **\crv{(5,-8) & (-5,-8)}; ?(.1)*\dir{>} ?(.95)*\dir{>}
  ?(.8)*\dir{}+(0,-.1)*{\bullet}+(-3,2)*{\scs g_3};
  (-4,15)*{}="t1";
  (4,15)*{}="t2";
  "t2";"t1" **\crv{(5,8) & (-5,8)}; ?(.15)*\dir{<} ?(.9)*\dir{<}
  ?(.4)*\dir{}+(0,-.2)*{\bullet}+(3,-2)*{\scs g_1};
  (0,0)*{\icbub{\scs \quad\;  g_2}{i}};
    (7,-13)*{\scs i};
  (-7,13)*{\scs i};
  (-10,10)*{\nub};
  \endxy}. \label{eq_ident_decomp-nleqz}
\end{equation}
\end{enumerate}
We prove the following relation which we will need for our main result.
\begin{Proposition} \label{dotslide}
Let  \(k, l \geq 0\) be integers such that \(k+l >0\). We have
\begin{equation}
 \vcenter{\xy 0;/r.18pc/:
    (-4,-4)*{};(4,4)*{} **\crv{(-4,-1) & (4,1)}?(1)*\dir{};
    (4,-4)*{};(-4,4)*{} **\crv{(4,-1) & (-4,1)}?(1)*\dir{};
    (-4,4)*{};(4,12)*{} **\crv{(-4,7) & (4,9)}?(1)*\dir{};
    (4,4)*{};(-4,12)*{} **\crv{(4,7) & (-4,9)}?(1)*\dir{};
    (8,9)*{\nub};
    (4,12); (4,13) **\dir{-}?(1)*\dir{>};
    (-4,12); (-4,13) **\dir{-}?(1)*\dir{>};
  (-5.5,-3)*{\scs i};
     (5.5,-3)*{\scs i};
       (7,4)*{\scs l};
      (-7,4)*{\scs k};
      (4,4)*{\bullet};
      (-4,4)*{\bullet};
 \endxy}
 \;\;= 
 \;\; \sum_{s=0}^{k-1}
  \; \vcenter{\xy 0;/r.18pc/:
  (0,0)*{\xybox{
    (-4,-4)*{};(4,6)*{} **\crv{(-4,-1) & (4,1)}?(1)*\dir{>}?;
    (4,-4)*{};(-4,6)*{} **\crv{(4,-1) & (-4,1)}?(1)*\dir{>};
    (-5,-3)*{\scs i};
     (5.1,-3)*{\scs i};
     (9,0)*{\nub};
     (-10,0)*{};(10,0)*{};
      (-3.2,3)*{\bullet};
       (3.2,3)*{\bullet};
       (-13,3)*{\scs k+l-1-s};
        (6,3)*{\scs s};
     }};
  \endxy}\;\; -\;\; \sum_{s=0}^{l-1}
  \; \vcenter{\xy 0;/r.18pc/:
  (0,0)*{\xybox{
    (-4,-4)*{};(4,6)*{} **\crv{(-4,-1) & (4,1)}?(1)*\dir{>}?;
    (4,-4)*{};(-4,6)*{} **\crv{(4,-1) & (-4,1)}?(1)*\dir{>};
    (-5,-3)*{\scs i};
     (5.1,-3)*{\scs i};
     (9,0)*{\nub};
     (-10,0)*{};(10,0)*{};
      (-3.2,3)*{\bullet};
       (3.2,3)*{\bullet};
       (-13,3)*{\scs k+l-1-s};
        (6,3)*{\scs s};
     }};
  \endxy}\; \; =  
  \end{equation}
  \[ = 
   \;\; \sum_{s=0}^{k-1}
  \; \vcenter{\xy 0;/r.18pc/:
  (0,0)*{\xybox{
    (-4,-4)*{};(4,6)*{} **\crv{(-4,-1) & (4,1)}?(1)*\dir{>}?;
    (4,-4)*{};(-4,6)*{} **\crv{(4,-1) & (-4,1)}?(1)*\dir{>};
    (-5,3)*{\scs i};
     (5.1,3)*{\scs i};
     (9,3)*{\nub};
     (-10,0)*{};(10,0)*{};
      (-3.0,-2.5)*{\bullet};
       (3.2,-2.5)*{\bullet};
       (-13,-2.5)*{\scs k+l-1-s};
        (6,-2.5)*{\scs s};
     }};
  \endxy}\;\;- \;\;
  \sum_{s=0}^{l-1}
  \; \vcenter{\xy 0;/r.18pc/:
  (0,0)*{\xybox{
    (-4,-4)*{};(4,6)*{} **\crv{(-4,-1) & (4,1)}?(1)*\dir{>}?;
    (4,-4)*{};(-4,6)*{} **\crv{(4,-1) & (-4,1)}?(1)*\dir{>};
    (-5,3)*{\scs i};
     (5.1,3)*{\scs i};
     (9,3)*{\nub};
     (-10,0)*{};(10,0)*{};
      (-3.0,-2.5)*{\bullet};
       (3.2,-2.5)*{\bullet};
       (-13,-2.5)*{\scs k+l-1-s};
        (6,-2.5)*{\scs s};
     }};
  \endxy}\, . 
  \]
  \end{Proposition}

\begin{proof}
We apply the following two nilHecke relations for the proof of the statement:
\[
\xy 0;/r.18pc/:
  (0,0)*{\xybox{
    (-4,-4)*{};(4,6)*{} **\crv{(-4,-1) & (4,1)}?(1)*\dir{>}?(.25)*{\bullet};
    (4,-4)*{};(-4,6)*{} **\crv{(4,-1) & (-4,1)}?(1)*\dir{>};
    (-5,-3)*{\scs i};
     (5.1,-3)*{\scs i};
   (9,3)*{\nub};
     (-10,0)*{};(10,0)*{};
     }};
  \endxy
 \; -
\xy 0;/r.18pc/:
  (0,0)*{\xybox{
    (-4,-4)*{};(4,6)*{} **\crv{(-4,-1) & (4,1)}?(1)*\dir{>}?(.75)*{\bullet};
    (4,-4)*{};(-4,6)*{} **\crv{(4,-1) & (-4,1)}?(1)*\dir{>};
    (-5,-3)*{\scs i};
     (5.1,-3)*{\scs i};
     (9,3)*{\nub};
     (-10,0)*{};(10,0)*{};
     }};
  \endxy
\; =  \xy 0;/r.18pc/:
  (0,0)*{\xybox{
    (-4,-4)*{};(4,6)*{} **\crv{(-4,-1) & (4,1)}?(1)*\dir{>};
    (4,-4)*{};(-4,6)*{} **\crv{(4,-1) & (-4,1)}?(1)*\dir{>}?(.75)*{\bullet};
    (-5,-3)*{\scs i};
     (5.1,-3)*{\scs i};
     (9,3)*{\nub};
     (-10,0)*{};(10,0)*{};
     }};
  \endxy
\; -
  \xy 0;/r.18pc/:
  (0,0)*{\xybox{
    (-4,-4)*{};(4,6)*{} **\crv{(-4,-1) & (4,1)}?(1)*\dir{>} ;
    (4,-4)*{};(-4,6)*{} **\crv{(4,-1) & (-4,1)}?(1)*\dir{>}?(.25)*{\bullet};
    (-5,-3)*{\scs i};
     (5.1,-3)*{\scs i};
     (9,3)*{\nub};
     (-10,0)*{};(12,0)*{};
     }};
  \endxy
  \; = \;
    \xy 0;/r.18pc/:
  (0,0)*{\xybox{
    (-4,-4)*{};(-4,6)*{} **\dir{-}?(1)*\dir{>} ;
    (4,-4)*{};(4,6)*{} **\dir{-}?(1)*\dir{>};
    (-5,-3)*{\scs i};
      (9,3)*{\nub};
     (5.1,-3)*{\scs i};
          (-10,0)*{};(12,0)*{};
     }};
  \endxy, \quad \quad \quad 
   \vcenter{ \xy 0;/r.14pc/:
    (-4,-4)*{};(4,4)*{} **\crv{(-4,-1) & (4,1)}?(1)*\dir{};
    (4,-4)*{};(-4,4)*{} **\crv{(4,-1) & (-4,1)}?(1)*\dir{};
    (-4,4)*{};(4,12)*{} **\crv{(-4,7) & (4,9)}?(1)*\dir{};
    (4,4)*{};(-4,12)*{} **\crv{(4,7) & (-4,9)}?(1)*\dir{};
    (8,9)*{\nub};
    (4,12); (4,13) **\dir{-}?(1)*\dir{>};
    (-4,12); (-4,13) **\dir{-}?(1)*\dir{>};
  (-5.5,-3)*{\scs i};
     (5.5,-3)*{\scs i};
 \endxy} \; = 0 .
\]

First assume \(l=0\). Then by sliding the \(k\) dots upwards through the crossing, we get
\[
 \vcenter{\xy 0;/r.17pc/:
    (-4,-4)*{};(4,4)*{} **\crv{(-4,-1) & (4,1)}?(1)*\dir{};
    (4,-4)*{};(-4,4)*{} **\crv{(4,-1) & (-4,1)}?(1)*\dir{};
    (-4,4)*{};(4,12)*{} **\crv{(-4,7) & (4,9)}?(1)*\dir{};
    (4,4)*{};(-4,12)*{} **\crv{(4,7) & (-4,9)}?(1)*\dir{};
    (8,9)*{\nub};
    (4,12); (4,13) **\dir{-}?(1)*\dir{>};
    (-4,12); (-4,13) **\dir{-}?(1)*\dir{>};
  (-5.5,-3)*{\scs i};
     (5.5,-3)*{\scs i};
      (-7,4)*{\scs k};
      (-4,4)*{\bullet};
 \endxy}
 \, =
 \; \vcenter{\xy 0;/r.17pc/:
    (-4,-4)*{};(4,4)*{} **\crv{(-4,-1) & (4,1)}?(1)*\dir{};
    (4,-4)*{};(-4,4)*{} **\crv{(4,-1) & (-4,1)}?(1)*\dir{};
    (-4,4)*{};(4,12)*{} **\crv{(-4,7) & (4,9)}?(1)*\dir{};
    (4,4)*{};(-4,12)*{} **\crv{(4,7) & (-4,9)}?(1)*\dir{};
    (9,9)*{\nub};
    (4,12); (4,13) **\dir{-}?(1)*\dir{>};
    (-4,12); (-4,13) **\dir{-}?(1)*\dir{>};
  (-5.5,-3)*{\scs i};
     (5.5,-3)*{\scs i};
      (-9,4)*{\scs k-1};
      (-4,4)*{\bullet};
      (3,10)*{\bullet};
 \endxy}
 \;\; + \; \;
\vcenter{\xy 0;/r.17pc/:
  (0,0)*{\xybox{
    (-4,-4)*{};(4,6)*{} **\crv{(-4,-1) & (4,1)}?(1)*\dir{>}?;
    (4,-4)*{};(-4,6)*{} **\crv{(4,-1) & (-4,1)}?(1)*\dir{>};
    (-5,-3)*{\scs i};
     (5.1,-3)*{\scs i};
     (9,3)*{\nub};
     (-10,0)*{};(10,0)*{};
       (-3.2,3)*{\bullet};
        (-8,3)*{\scs k-1};
     }};
  \endxy}
  \;\; =  \; \vcenter{\xy 0;/r.17pc/:
    (-4,-4)*{};(4,4)*{} **\crv{(-4,-1) & (4,1)}?(1)*\dir{};
    (4,-4)*{};(-4,4)*{} **\crv{(4,-1) & (-4,1)}?(1)*\dir{};
    (-4,4)*{};(4,12)*{} **\crv{(-4,7) & (4,9)}?(1)*\dir{};
    (4,4)*{};(-4,12)*{} **\crv{(4,7) & (-4,9)}?(1)*\dir{};
    (9,9)*{\nub};
    (4,12); (4,13) **\dir{-}?(1)*\dir{>};
    (-4,12); (-4,13) **\dir{-}?(1)*\dir{>};
  (-5.5,-3)*{\scs i};
     (5.5,-3)*{\scs i};
      (-9,4)*{\scs k-2};
      (-4,4)*{\bullet};
      (3,10)*{\bullet};
      (5.5,10)*{\scs 2};
 \endxy}
 \;\; + \; \;
\vcenter{\xy 0;/r.17pc/:
  (0,0)*{\xybox{
    (-4,-4)*{};(4,6)*{} **\crv{(-4,-1) & (4,1)}?(1)*\dir{>}?;
    (4,-4)*{};(-4,6)*{} **\crv{(4,-1) & (-4,1)}?(1)*\dir{>};
    (-5,-3)*{\scs i};
     (5.1,-3)*{\scs i};
     (9,3)*{\nub};
     (-10,0)*{};(10,0)*{};
       (-3.2,3)*{\bullet};
        (-8,3)*{\scs k-1};
     }};
  \endxy}
  + \; \;
\vcenter{\xy 0;/r.17pc/:
  (0,0)*{\xybox{
    (-4,-4)*{};(4,6)*{} **\crv{(-4,-1) & (4,1)}?(1)*\dir{>}?;
    (4,-4)*{};(-4,6)*{} **\crv{(4,-1) & (-4,1)}?(1)*\dir{>};
    (-5,-3)*{\scs i};
     (5.1,-3)*{\scs i};
     (9,3)*{\nub};
     (-10,0)*{};(10,0)*{};
       (-3.2,3)*{\bullet};
       (3.2,3)*{\bullet};
        (-8,3)*{\scs k-2};
     }};
  \endxy}
  \;\; = \, \dots 
  \]
  \begin{equation} \label{eqql}
  = \; \vcenter{\xy 0;/r.17pc/:
    (-4,-4)*{};(4,4)*{} **\crv{(-4,-1) & (4,1)}?(1)*\dir{};
    (4,-4)*{};(-4,4)*{} **\crv{(4,-1) & (-4,1)}?(1)*\dir{};
    (-4,4)*{};(4,12)*{} **\crv{(-4,7) & (4,9)}?(1)*\dir{};
    (4,4)*{};(-4,12)*{} **\crv{(4,7) & (-4,9)}?(1)*\dir{};
    (9,9)*{\nub};
    (4,12); (4,13) **\dir{-}?(1)*\dir{>};
    (-4,12); (-4,13) **\dir{-}?(1)*\dir{>};
  (-5.5,-3)*{\scs i};
     (5.5,-3)*{\scs i};
      (3,10)*{\bullet};
      (6,10)*{\scs k};
 \endxy}
 \;\; + \, \sum_{s=0}^{k-1}\; 
\vcenter{\xy 0;/r.17pc/:
  (0,0)*{\xybox{
    (-4,-4)*{};(4,6)*{} **\crv{(-4,-1) & (4,1)}?(1)*\dir{>}?;
    (4,-4)*{};(-4,6)*{} **\crv{(4,-1) & (-4,1)}?(1)*\dir{>};
    (-5,-3)*{\scs i};
     (5.1,-3)*{\scs i};
     (9,3)*{\nub};
     (-10,0)*{};(10,0)*{};
       (-3.2,3)*{\bullet};
       (3.2,3)*{\bullet};
        (-11,3)*{\scs k-1-s};
        (6,3)*{\scs s};
     }};
  \endxy}
  \;\;  = \; \sum_{s=0}^{k-1}\; 
\vcenter{\xy 0;/r.20pc/:
  (0,0)*{\xybox{
    (-4,-4)*{};(4,6)*{} **\crv{(-4,-1) & (4,1)}?(1)*\dir{>}?;
    (4,-4)*{};(-4,6)*{} **\crv{(4,-1) & (-4,1)}?(1)*\dir{>};
    (-5,-3)*{\scs i};
     (5.1,-3)*{\scs i};
     (9,3)*{\nub};
     (-10,0)*{};(10,0)*{};
       (-3.2,3)*{\bullet};
       (3.2,3)*{\bullet};
        (-11,3)*{\scs k-1-s};
        (6,3)*{\scs s};
     }};
  \endxy}
  \;.
  \end{equation}
   We now let \(l \geq 0\) and slide the \(l\) dots upwards:
   \[\vcenter{\xy 0;/r.17pc/:
    (-4,-4)*{};(4,4)*{} **\crv{(-4,-1) & (4,1)}?(1)*\dir{};
    (4,-4)*{};(-4,4)*{} **\crv{(4,-1) & (-4,1)}?(1)*\dir{};
    (-4,4)*{};(4,12)*{} **\crv{(-4,7) & (4,9)}?(1)*\dir{};
    (4,4)*{};(-4,12)*{} **\crv{(4,7) & (-4,9)}?(1)*\dir{};
    (8,9)*{\nub};
    (4,12); (4,13) **\dir{-}?(1)*\dir{>};
    (-4,12); (-4,13) **\dir{-}?(1)*\dir{>};
  (-5.5,-3)*{\scs i};
     (5.5,-3)*{\scs i};
       (7,4)*{\scs l};
      (-7,4)*{\scs k};
     (4,4)*{\bullet};
      (-4,4)*{\bullet};
 \endxy}
 \; \; = \; \;
 \vcenter{\xy 0;/r.17pc/:
    (-4,-4)*{};(4,4)*{} **\crv{(-4,-1) & (4,1)}?(1)*\dir{};
    (4,-4)*{};(-4,4)*{} **\crv{(4,-1) & (-4,1)}?(1)*\dir{};
    (-4,4)*{};(4,12)*{} **\crv{(-4,7) & (4,9)}?(1)*\dir{};
    (4,4)*{};(-4,12)*{} **\crv{(4,7) & (-4,9)}?(1)*\dir{};
    (8,9)*{\nub};
    (4,12); (4,13) **\dir{-}?(1)*\dir{>};
    (-4,12); (-4,13) **\dir{-}?(1)*\dir{>};
  (-5.5,-3)*{\scs i};
     (5.5,-3)*{\scs i};
       (9,4)*{\scs l-1};
      (-7,4)*{\scs k};
     (4,4)*{\bullet};
      (-4,4)*{\bullet};
      (-3,10)*{\bullet};
 \endxy}
 \; \; - \; \vcenter{\xy 0;/r.18pc/:
  (0,0)*{\xybox{
    (-4,-4)*{};(4,6)*{} **\crv{(-4,-1) & (4,1)}?(1)*\dir{>}?;
    (4,-4)*{};(-4,6)*{} **\crv{(4,-1) & (-4,1)}?(1)*\dir{>};
    (-5,-3)*{\scs i};
     (5.1,-3)*{\scs i};
     (9,0)*{\nub};
     (-10,0)*{};(10,0)*{};
      (-3.2,3)*{\bullet};
       (3.2,3)*{\bullet};
       (-6,3)*{\scs k};
        (8,3)*{\scs l-1};
     }};
  \endxy}\; \; =\;\;
 \vcenter{\xy 0;/r.17pc/:
    (-4,-4)*{};(4,4)*{} **\crv{(-4,-1) & (4,1)}?(1)*\dir{};
    (4,-4)*{};(-4,4)*{} **\crv{(4,-1) & (-4,1)}?(1)*\dir{};
    (-4,4)*{};(4,12)*{} **\crv{(-4,7) & (4,9)}?(1)*\dir{};
    (4,4)*{};(-4,12)*{} **\crv{(4,7) & (-4,9)}?(1)*\dir{};
    (8,9)*{\nub};
    (4,12); (4,13) **\dir{-}?(1)*\dir{>};
    (-4,12); (-4,13) **\dir{-}?(1)*\dir{>};
  (-5.5,-3)*{\scs i};
     (5.5,-3)*{\scs i};
       (9,4)*{\scs l-2};
      (-7,4)*{\scs k};
     (4,4)*{\bullet};
      (-4,4)*{\bullet};
      (-3,10)*{\bullet};
      (-5.7,10)*{\scs 2};
 \endxy}
 \; \; -\; \vcenter{\xy 0;/r.18pc/:
  (0,0)*{\xybox{
    (-4,-4)*{};(4,6)*{} **\crv{(-4,-1) & (4,1)}?(1)*\dir{>}?;
    (4,-4)*{};(-4,6)*{} **\crv{(4,-1) & (-4,1)}?(1)*\dir{>};
    (-5,-3)*{\scs i};
     (5.1,-3)*{\scs i};
     (9,0)*{\nub};
     (-10,0)*{};(10,0)*{};
      (-3.2,3)*{\bullet};
       (3.2,3)*{\bullet};
       (-6,3)*{\scs k};
        (8,3)*{\scs l-1};
     }};
  \endxy} \;\; -
  \; \vcenter{\xy 0;/r.18pc/:
  (0,0)*{\xybox{
    (-4,-4)*{};(4,6)*{} **\crv{(-4,-1) & (4,1)}?(1)*\dir{>}?;
    (4,-4)*{};(-4,6)*{} **\crv{(4,-1) & (-4,1)}?(1)*\dir{>};
    (-5,-3)*{\scs i};
     (5.1,-3)*{\scs i};
     (9,0)*{\nub};
     (-10,0)*{};(10,0)*{};
      (-3.2,3)*{\bullet};
       (3.2,3)*{\bullet};
       (-8,3)*{\scs k+1};
        (8,3)*{\scs l-2};
     }};
  \endxy}\; = \dots 
 \]
 
 \[=\;\;  \vcenter{\xy 0;/r.17pc/:
    (-4,-4)*{};(4,4)*{} **\crv{(-4,-1) & (4,1)}?(1)*\dir{};
    (4,-4)*{};(-4,4)*{} **\crv{(4,-1) & (-4,1)}?(1)*\dir{};
    (-4,4)*{};(4,12)*{} **\crv{(-4,7) & (4,9)}?(1)*\dir{};
    (4,4)*{};(-4,12)*{} **\crv{(4,7) & (-4,9)}?(1)*\dir{};
    (8,9)*{\nub};
    (4,12); (4,13) **\dir{-}?(1)*\dir{>};
    (-4,12); (-4,13) **\dir{-}?(1)*\dir{>};
  (-5.5,-3)*{\scs i};
     (5.5,-3)*{\scs i};
      (-7,4)*{\scs k};
      (-4,4)*{\bullet};
      (-3,10)*{\bullet};
      (-5.7,10)*{\scs l};
 \endxy}
 \; \;  -\;\; \sum_{s=0}^{l-1}
  \; \vcenter{\xy 0;/r.18pc/:
  (0,0)*{\xybox{
    (-4,-4)*{};(4,6)*{} **\crv{(-4,-1) & (4,1)}?(1)*\dir{>}?;
    (4,-4)*{};(-4,6)*{} **\crv{(4,-1) & (-4,1)}?(1)*\dir{>};
    (-5,-3)*{\scs i};
     (5.1,-3)*{\scs i};
     (9,0)*{\nub};
     (-10,0)*{};(10,0)*{};
      (-3.2,3)*{\bullet};
       (3.2,3)*{\bullet};
       (-13,3)*{\scs k+l-1-s};
        (6,3)*{\scs s};
     }};
  \endxy}\; \stackrel{\eqref{eqql}}{=} \;\;
 \sum_{s=0}^{k-1}
  \; \vcenter{\xy 0;/r.18pc/:
  (0,0)*{\xybox{
    (-4,-4)*{};(4,6)*{} **\crv{(-4,-1) & (4,1)}?(1)*\dir{>}?;
    (4,-4)*{};(-4,6)*{} **\crv{(4,-1) & (-4,1)}?(1)*\dir{>};
    (-5,-3)*{\scs i};
     (5.1,-3)*{\scs i};
     (9,0)*{\nub};
     (-10,0)*{};(10,0)*{};
      (-3.2,3)*{\bullet};
       (3.2,3)*{\bullet};
       (-13,3)*{\scs k+l-1-s};
        (6,3)*{\scs s};
     }};
  \endxy}\;\; -
  \;\; \sum_{s=0}^{l-1}
  \; \vcenter{\xy 0;/r.18pc/:
  (0,0)*{\xybox{
    (-4,-4)*{};(4,6)*{} **\crv{(-4,-1) & (4,1)}?(1)*\dir{>}?;
    (4,-4)*{};(-4,6)*{} **\crv{(4,-1) & (-4,1)}?(1)*\dir{>};
    (-5,-3)*{\scs i};
     (5.1,-3)*{\scs i};
     (9,0)*{\nub};
     (-10,0)*{};(10,0)*{};
      (-3.2,3)*{\bullet};
       (3.2,3)*{\bullet};
       (-13,3)*{\scs k+l-1-s};
        (6,3)*{\scs s};
     }};
  \endxy}\;.
 \]
The proof of the second equality is entirely similar, we slide the dots downwards through the crossing.
\end{proof}

Let \(\dot{\U}_Q(\mathfrak{sl}_n)\) be the Karoubi envelope of \(\,\U_Q(\mathfrak{sl}_n)\), the smallest 2-category which contains \(\U_Q(\mathfrak{sl}_n)\) and has splitting idempotents. 
The diagrammatic description of \(\dot{\U}_Q(\mathfrak{sl}_2)\) is given in \cite{exten}, but has not been worked out yet in full generality for any \(n\). 
The split Grothendieck group  \(K_0(\dot{\U}_Q(\mathfrak{sl}_n))\) is the additive category with objects \(\nub \in X\), and the abelian group of morphisms \(\nub \to \mub \), \(\mub \in X\)  is the split Grothendieck group \(K_0 \left(\dot{\U}_Q(\mathfrak{sl}_n) (\nub, \mub) \right) \) of the additive category \(\dot{\U}_Q(\mathfrak{sl}_n)(\nub, \mub) \). We can view \(K_0(\dot{\U}_Q(\mathfrak{sl}_n))\) as a \(\cal{A}=\Z[q, q^{-1}]\) module, since the degree shift functor on 
\(\dot{\U}_Q(\mathfrak{sl}_n)\) induces a multiplication by \(q\) in \(K_0(\dot{\U}_Q(\mathfrak{sl}_n))\).  The following theorem is the categorification  of quantum \(\mathfrak{sl}_n\) by Khovanov and Lauda. 
\begin{Theorem} \cite{lauda} \label{klau}
There is \(\cal{A}\)-module isomorphism 
\begin{equation}
\gamma : \UA_q(\mathfrak{sl}_n)  \to K_0(\dot{\U}_Q(\mathfrak{sl}_n)),
\end{equation}
such that \(\gamma(1_{\nub}) = [\onel ]_{\cong} \), \( \gamma(E_i 1_{\nub})  =[ \mathcal{E}_i \onel]_{\cong}  \), and \( \gamma(F_i 1_{\nub})  =[ \mathcal{F}_i \onel]_{\cong}  \).
\end{Theorem}
We define $\U =\U^{\star}_Q(\mf{sl}_n)$  to be the 2-category with the same objects and 1-morphisms as $\Ucat_Q(\mf{sl}_n)$ and with 2-morphisms
\[\U(\nub,  \bar{\mu}) (f,g) = \bigoplus_{t \in \Z} \, \U_Q (\mf{sl}_n)(\nub,  \bar{\mu}) (f, g \la t \ra) \]
for all  \(\nub,  \bar{\mu} \in \text{Ob}(\U) )\) and \(f,g \colon \nub \to \bar{\mu}\). Horizontal composition in  $\U$ is induced from the horizontal composition in  $\Ucat_Q(\mf{sl}_n)$. We will work mainly with the 2-category  $\U$ from now on. \(K_0\) decategorification of \(\U\) is the same as that of $\U^{}_Q(\mf{sl}_n)$, however, it has more interesting trace decategorification as we will see next.

\subsection{The current algebra and the trace of \(\U\)  }
Here we will compute the trace  of the 2-category \(\U\) and compare it to the split Grothendieck group. Let \(E_i, F_i, H_i, \, i \in I \) be the  Chevalley generators of \(\mathfrak{sl}_n\).  By \(\mathfrak{sl}_n[t]= \mathfrak{sl}_n \otimes \k[t]\), we denote the extension of \(\mathfrak{sl}_n\) by polynomials in \(t\). Then \(\mathfrak{sl}_n[t]\) is a Lie algebra with the Lie bracket is defined as \( [a \otimes t^r, b \otimes t^s]= [a, b] \otimes t^{r+s}\), for \(a,b \in \mathfrak{sl}_n\) and \(r,s \geq 0 \). 
We call  \(\mathbf{U}(\mathfrak{sl}_n[t])\) the current algebra. Current algebra is graded with  \(\deg (a \otimes t^k) = 2k\). 
This algebra is generated over $\k$ by \(E_{i,r}=E_{i}\otimes t^r, F_{i,r}=F_i\otimes t^r, H_{i,r}=H_i\otimes t^r\) for $r\geq 0$ and $i \in I$
modulo the following relations:
\begin{description}
  \item[C1] For $i,j \in I$ and $r,s \geq 0$
\[
[ H_{i,r}, H_{j,s} ] = 0,
\]

  \item[C2] For any $i,j \in I$ and $r,s \geq 0$,
\[
 [H_{i,r}, E_{j,s} ] =  a_{ij} E_{j, r+s}, \quad [H_{i,r}, F_{j,s} ] = - a_{ij} F_{j, r+s},
\]

  \item[C3] For any $i,j \in I$ and $r, s  \geq 0$,
\[
[E_{i, r+1}, E_{j,s}]=[E_{i,r}, E_{j,s+1}], \quad [F_{i, r+1}, F_{j,s}]=[F_{i,r}, F_{j,s+1}],
\]
  \item[C4] For any $i,j \in I$ and $r,s \geq 0$
\[
 [E_{i,r}, F_{j,s}] = \delta_{i,j} H_{i, r+s},
\]
  \item[C5] Let $i \neq j$ and set $m=1-a_{ij}$.  For every $r_1, \dots, r_m, s \geq 0$ 
\[
 \sum_{\pi \in S_m} \sum_{l=0}^{m} (-1)^l
 \binom{m}{l} E_{i, k_{\pi(1)}} \dots E_{i, k_{\pi(l)}}
E_{j, s} E_{i, k_{\pi(l+1)}} \dots E_{i, k_{\pi(m)}} = 0, \] 
\[ \sum_{\pi \in S_m} \sum_{l=0}^{m} (-1)^l \binom{m}{l} F_{i, k_{\pi(1)}} \dots F_{i, k_{\pi(l)}}
F_{j, s} F_{i, k_{\pi(l+1)}} \dots F_{i, k_{\pi(m)}} = 0.
\]
\end{description}

 We identify \(\mathbf{U}(\mathfrak{sl}_n)\) as a subalgebra of \(\mathbf{U}(\mathfrak{sl}_n[t])\) generated by degree zero elements \(E_{i,0}, F_{i,0}, H_{i,0}, \, i \in I \). Let \(\dot{\mathbf{U}}(\mathfrak{sl}_n[t])\) be the idempotented version of the current algebra  \(\mathbf{U}(\mathfrak{sl}_n[t])\), where the unit is replaced by the collection of mutually orthogonal idempotents \(1_{\nub}\) for each \(\mathfrak{sl}_n\)-weight \(\nub\), such that
\[ 1_{\nub}1_{\nub'}= \delta_{\nub. \nub'}1_{\nub}, \quad 1_{\nub}H_{i,0} =H_{i,0}1_{\nub}= \nub_i 1_{\nub},  
\quad 1_{\overbar{\nu+\alpha_i}}E_{i,j} =E_{i,j}1_{\nub}, \quad 1_{\overbar{\nu-\alpha_i}}F_{i,j} =F_{i,j}1_{\nub}.
\]


Now we describe the trace \( \Tr \U\) of the 2-category \(\U\). The trace class of a 2-morphisms in \(\U\) can be depicted by closing a 2-morphism diagrams to the right:
 \[\left[ \vcenter{ \xy 0;/r.30pc/:
 (-8,0)*{\ecross};
(-3,-3)*{\scs i}; (-13,-3)*{\scs i};
 (-10,1)*{\bullet}; (-2,0)*{\nub};
 \endxy} 
   \right]
 \qquad = \qquad
 \vcenter{ \xy 0;/r.18pc/:
  (8,8)*{\lcap};
  (8,14)*{\xlcap};
(9,-4)*{
 \hackcenter{\begin{tikzpicture}
  \path[draw, blue, very thick, fill=white]
    (-0.15,0) .. controls ++(0,.2) and ++(0,.2) .. (0.15,0)
            .. controls ++(0,-.2) and ++(0,-.2) .. (-0.15,0);
  \end{tikzpicture}}
 };
 (8,-14)*{\xlcupef};
 (-8,0)*{\ecross};
 (20,0)*{\sfline};
 (28,0)*{\sfline};
 (8,-10)*{\lcupef}; (5,2)*{\nub};
(-2,-3)*{\scs i}; 
(-15,-3)*{\scs i};
 (-10,1)*{\bullet};
 \endxy}\, \, .
\]

Let \(\mathsf{E}_{i,r}1_{\nub}, \mathsf{F}_{j,r}1_{\nub}, \mathsf{H}_{i,r}1_{\nub}\) denote
the following trace classes:
\[\mathsf{E}_{i,r}1_{\nub}=  \left[
\xy 0;/r.18pc/:
  (10,6)*{\nub};
  (0,0)*{\bbe{}};
  (0,2)*{\bullet}+(3,1)*{r};
  (-8,6)*{ }; (12,-7)*{}; (2,-7)*{\scs i};
 \endxy
\right],  \quad \quad
\mathsf{F}_{j,r}1_{\nub}=  \left[
\xy 0;/r.18pc/:
  (10,6)*{\nub};
  (0,0)*{\bbf{}};
  (2,-7)*{\scs j};
  (0,4)*{\bullet}+(3,1)*{r};
   (-8,6)*{ }; (12,-7)*{};
 \endxy
\right],\quad \quad
  \mathsf{H}_{i,r}1_{\nub}= \left[
 \pi_{i,r}(\nub)\; \Id_{\onel}
\right],\]
where 
\begin{equation*} \label{eq_defpil}
 \pi_{i,r}(\nub) = \sum_{l=0}^{r} (l+1)
 \xy
   (0,-2)*{\icbub{\spadesuit+l}{i}};
   (12,-2)*{\iccbub{\spadesuit+r-l}{i}};
  (8,8)*{\nub};
 \endxy
\end{equation*}
for \(r>0\) and  $\pi_{i,0}(\nub) = \nub_i$. 
 
 The following theorem is a fundamental  result proved by A. Beliakova, K. Habiro, A. Lauda, B. Webster \cite{lcurr}. 

 \begin{Theorem} \label{prop-cur}
  There is an algebra isomorphism 
\begin{equation} \label{eq_sln-homom}
 \rho \maps \dot{\bfU}(\mf{sl}_n[t]) \longrightarrow \Tr \U,
\end{equation}
given by
\begin{equation*} 
E_{i, r} 1_{\nub} \mapsto
\mathsf{E}_{i,r} 1_{\nub}
 ,\quad \quad
F_{i, r} 1_{\nub}\mapsto
\mathsf{F}_{i,r} 1_{\nub},
\quad \quad
H_{i, r} 1_{\nub}\mapsto
 \mathsf{H}_{i,r}1_{\nub}.
\end{equation*}
\end{Theorem}

Let \(Z({\nub})\),  the center of an object \(\nub \in \text{Ob}(\U)\),  to be the endomorphism ring \(\End (1_{\nub})\). The ring \(Z({\nub})\) is a commutative ring, diagrammatically given by closed diagrams  label by \(\nub\) on the outside of the diagram. 
By \cite{lauda2},  \(Z({\nub})\) is freely generated by counterclockwise oriented fake and real bubbles. We will define the center of objects  \(Z(\U_{})\) of $\U$ as
\[ Z(\U_{}) = \bigoplus_{\nub \in \Ob (\U)} Z({\nub}). \]
Cyclicity relations show that \(\U\) is cyclic as a 2-category. Hence, we can define an action of \( \Tr \U \) on the center of objects \(Z (\U) \), and thus an action of the current algebra. 
\begin{Corollary} \label{thm_main}
The vector space \(Z (\U)\) is a $\bfU(\fsl_n[t])$-module with
\begin{equation} \label{mpp1}
  E_{i,j}  \maps Z({\nub}) \to  Z({\overbar{\nu+\alpha_i}}),   
\quad \quad  \quad \xy
  (0,0)*{\ast};
  (6,4)*{\nub};
 \endxy
 \longmapsto  \;\;
  \xy
(0,0)*{\lcbub{i}};
  (-8,-8)*{\bullet}+(-3,-2)*{j};
  (0,0)*{\ast};
  (4,6)*{\nub};
  (22,6)*{\overbar{\nu+\alpha_i}};
 \endxy\;\;,
   \end{equation}
   
\begin{equation} \label{mpp2}
 F_{i,j} \maps Z({\nub}) \to  Z({\overbar{\nu-\alpha_i}}), \quad \quad \quad
 \xy
  (0,0)*{\ast};
  (6,4)*{\nub};
 \endxy
 \longmapsto   \;\;
 \xy
(0,0)*{\lccbub{i}};
(-8,-8)*{\bullet}+(-3,-2)*{j};
  (0,0)*{\ast};
  (4,6)*{\nub};
  (22,6)*{{\overbar{\nu-\alpha_i}}};
 \endxy,
 \end{equation}
 
 \begin{equation} \label{mpp3}
   H_{i,j} \colon Z({\nub}) \to  Z({\nub}),  \\ \quad \quad  \quad
   \xy
  (0,0)*{\ast};
  (6,4)*{\nub};
 \endxy  \longmapsto 
     \xy
     (0,-6)*{};
  (-9,0)*{\pi_{i,j}(\nub)};
  (0,0)*{\ast};
  (6,4)*{\nub};
 \endxy.
\end{equation}

\end{Corollary}

\begin{proof}
This theorem follows immediately from Theorem \ref{prop-cur}.
In particular, the fact that the current algebra relations hold in the trace $\Tr \U$ imply that these relations hold under the action described above.
\end{proof}

Notice that the maps  \( F_{i,j}\) and \( E_{i,j}\) send an element of degree \(d\) to an element of degree \(d+2(j+\nub_i -1)\) and \(d+2(j-\nub_i -1)\), respectively. In order to see this for \( F_{i,0}\), let \(*\) the empty diagram and \(j=0\). Then \(F_{i,0}\) sends it to a counterclockwise bubble of degree \(2(\nub_i -1)\):
 \begin{equation} \label{mpw2}
 \xy
  (0,0)*{\ast};
  (6,4)*{\nub};
 \endxy
 \xrightarrow[]{ F_{i,0} }   \;\;
 \xy
(0,0)*{\lccbub{i}};
  (4,6)*{\nub};
  (18,6)*{{\overbar{\nu-\alpha_i}}};
 \endxy =  \;
   \xy 0;/r.18pc/:
  (4,8)*{{\overbar{\nu-\alpha_i}}};
  (2,-2)*{\iccbub{-(\overbar{\nu-\alpha_i})_i-1+ \nub_i-1}{i}};
 \endxy =  \xy 0;/r.18pc/:
  (6,8)*{{\overbar{\nu-\alpha_i}}};
  (2,-2)*{\iccbub{\spadesuit+\nub_i-1}{i}};
 \endxy,
 \end{equation}
 where we use the identity \((\overbar{\nu-\alpha_i})_i=\nub_i-2\).

\section{Local Weyl modules of current algebra}
\subsection{Local Weyl modules}

Let \(X\) be the integral \(\mathfrak{sl}_n\) weight lattice and  \(\lb \in X\) be a dominant integral weight. There is a partial order on \(X\), defined as \(\bar{\l} \geq \nub \) if \(\bar{\l} - \nub \) is a positive linear combination of simple roots. A \(\mathfrak{sl}_n\)-module is called a {\em weight module} if it is a direct sum of its weight spaces, and an {\em integrable module} if \(E_i\) and \(F_i\) act nilpotently for all \(i \in I\).

We define \(V(\lb)\) to be the \(\mathbf{U}(\mathfrak{sl}_n)\)-module generated by a vector \(v_{\lb}\) over  \(\mathbf{U}(\mathfrak{sl}_n)\), subject to the following relations for all \(i \in I\):
\[ E_iv_{\lb}=F_i ^{\lb_i+1}v_{\lb}=0, \qquad H_iv_{\lb}=\lb_i v_{\lb}.\]
\(V(\lb)\)  is the unique, up to isomorphism, finite-dimensional irreducible \(\mathbf{U}(\mathfrak{sl}_n)\)-module with the highest weight \(\lb\).  We can also view \(V(\lb)\) as a   \(\mathbf{U}(\mathfrak{sl}_n[t])\)-module, where the action of positive degree elements of \(\mathbf{U}(\mathfrak{sl}_n[t])\) are trivial.  Let \(\Z[X ]\) be the integral group ring spanned by elements \(e(\nub)\), \(\nub \in X\). We call the element 
\[ \text{ch} \, V(\lb) = \sum_{\lb \geq \nub } \dim_{\mathbb{C}} \, V_{\nub}(\lb) \, e(\nub)\]
in \(\Z[X ]\) the {\em  character} of \(V(\lb)\), where \( {V}_{\nub}(\lb) = \{ v \in {V}_{}(\lb) \,| \, H_{i} v = \nub_i v \}\).

In this section we give a short introduction to   \(\mathbf{U}(\mathfrak{sl}_n[t])\)-modules, called local Weyl modules. The simple \(\mathfrak{sl}_n\)-modules \(V(\lb)\) are sometimes also called Weyl modules. In this work, however, we apply the term only to the following current algebra modules. 
\begin{Definition}
The local Weyl module \(W(\lb)\) is the  \(\mathbf{U}(\mathfrak{sl}_n[t])\)-module generated by the element \(v_{\lb}\) over  \(\mathbf{U}(\mathfrak{sl}_n[t])\) together with the following relations for all  \(i\in I\) and \( j \geq 0\):
\begin{equation} \label{defrel}
E_{i,j} v_{\lb}=F^{\lb_i+1}_{i,0} v_{\lb}=0, \quad \quad H_{i,0}v_{\lb}=\delta_{j,0 }\, \lb_i v_{\lb}.
\end{equation}    
\end{Definition}

 \({W}(\lb)\)  is a weight module, that is, it decomposes into a direct sum of weight spaces:
 \begin{equation}
  {W}(\lb) = \bigoplus_{\lb \geq \nub} {W}_{\nub}(\lb),
 \end{equation}
where \( {W}_{\nub}(\lb) = \{ v \in {W}_{}(\lb) \,| \, H_{i,0} v = \nub_i v \}\).

In the simplest case \(n=2\) and \(\lb = 2\omega_1\), the local Weyl module and the current algebra action can be illustrated by the following graph. The upward oriented edge describes the action of \(H_1:=H_{1,1}\), the left and right oriented edges shows the actions of \(E_j:=E_{1,j}\) and \(F_j:=F_{1,j}\), respectively.

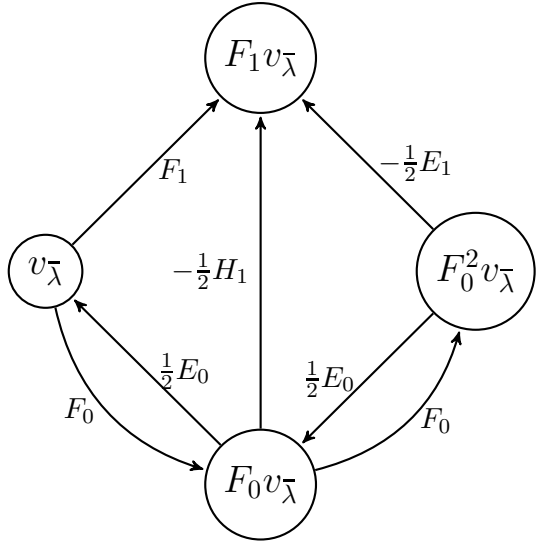
\begin{center}
\begin{tikzpicture}[->,>=stealth',shorten >=1pt,auto,node distance=4cm,
                    thick,main node/.style={circle, draw, font=\sffamily\large\bfseries}]
  \node[main node] (1) {\(F_{1} v_{\lb}\) };
  \node[main node] (2) [below left  of = 1] {\(v_{\lb}\)};
  \node[main node] (3) [below right of=2] {\(F_{0} v_{\lb}\)};
  \node[main node] (4) [below right of=1] {\(F^2_{0} v_{\lb}\)};

  \path[every node/.style={font=\sffamily\small}]
    (1)
    (2) edge node [right] {\(F_{1}\)} (1)
        edge [bend right] node[left] {\(F_{0}\)} (3)
    (3) edge node [ right] {\(\frac{1}{2} E_{0}\)} (2)
        edge [bend right] node[right] {\(F_{0}\)} (4)
         edge node  {\(-\frac{1}{2} H_{1}\)} (1)
   (4) edge node [ left] {\(\frac{1}{2} E_{0}\)} (3)
        edge [ right] node[right] {\(-\frac{1}{2} E_{1}\)} (1);
\end{tikzpicture}
\end{center}

As a module generated by a single element \(v_{\lb}\) over a graded algebra, \({W}(\lb)\) inherits the degree from the current algebra, where we set the degree of \(v_{\lb}\) to zero. Since the degree of \(t\)-parameter of the current algebra is 2, we can write \(W(\lb) \)  as 
\[W(\lb) = \bigoplus_{r \geq 0} W(\lb) \{ 2r \} ,\]
where each homogeneous  \(2r\)-graded piece \(W(\lb) \{ 2r \} = \{ v \in W(\lb), \deg (v)=2r\}\)  is a \(\mathfrak{sl}_n\)-module. Then the character  of a local Weyl module \(W(\lb)\) is 
\begin{equation}
\text{ch}_t \, W(\lb) = \sum_{r \geq 0} \text{ch} \, W(\lb) \{ 2r \} \,  t^{r}.
\end{equation}
The local Weyl modules have the following universal property.
\begin{Theorem} \cite{chp} \label{unip}
Any finite-dimensional current algebra module generated by an element \(v_{\lb}\) satisfying the relations \eqref{defrel} is a quotient of the local Weyl module \(W(\lb)\).
\end{Theorem}

\subsection{Duals of local Weyl modules}
Let \(\mathcal{P}^{}(n, N)\) be the set of \(n\)-compositions of \(N\). We define a partial order, called the {\em dominance order}, on \(\mathcal{P}^{}(n, N)\) by setting \(\mu \geq \nu\) for \( \nu, \mu \in \mathcal{P}^{}(n, N)\) if \(\mu -\nu\) is a positive linear composition of simple roots. 

For any compositions \( \nu, \mu \in \mathcal{P}(n, N)\) define the integer
\begin{equation} \label{dd}
d^{\mu}_{\nu} = \max\{ (\mu, \mu)-(\nu, \nu), 0\}.
\end{equation}
Notice that \(d^{\mu}_{\nu}\)  is an even non-negative integer, and it depends on  \(\bar{\mu}\) and \(\nub\) rather than \(\mu_{}\) and \(\nu\). To see this, take the partition \(\mu' = \mu +(m, m, \dots, m) \in \mathcal{P}+(n, N+nm) \) and \(\nu' = \nu +(m, m, \dots, m) \in \mathcal{P}+(n, N+nm) \) for any \(m \in \Z_+\). Then we have \(d^{\mu'}_{\nu'} = d^{\mu}_{\nu}\).

 We now define the duals of \(V(\lb)\) and \(W(\lb)\) as follows. 
 Let \(V^*_{\nub}(\lb)\)  be the dual vector space of the  weight space \(V_{\nub}(\lb)\). Namely, for each \(w \in V_{\nub}(\lb)\), \(V^*_{\nub}(\lb)\)  is spanned by linear maps \(\delta_w \colon V_{\nub}(\lb) \to \mathbb{ C}\) such that \(\delta_w(v)\) equals the Kronecker delta function \( \delta_{w,v}\).
 
The dual of \(V(\lb)\) is the  \(\mathbf{U}(\mathfrak{sl}_n)\) -module
 \[V^*_{}(\lb) = \bigoplus_{\lb \geq \nub}  V^*_{\nub}(\lb)\] 
  in which the action is defined as  
\begin{equation} \label{dual}
(x \cdot \delta_w)v = \delta_w(\omega(x) v).
\end{equation} 

In the equation \eqref{dual},  \(v, w \in {V}(\lb)\), \(\delta_w \in {V}^*(\lb)\), \(x \in \mathbf{U}(\mathfrak{sl}_n)\) and \(\omega\) is the anti-involution on \(\mathbf{U}(\mathfrak{sl}_n)\) which sends \(F_{i}\) to \(E_{i}\) and fixes \(H_{i}\). Under this definition the dual \(V^*(\lb)\) is isomorphic to \(V(\lb)\) as a \(\mathbf{U}(\mathfrak{sl}_n)\)-module.

As we will see next, the dual local Weyl module  \(W^*(\lb)\)  is defined similarly, by dualizing each weight space and taking the direct sum. However, we have to consider the fact that the weight spaces \(W_{\nub}(\lb)\), \(\nub \in X\) are graded. We will need the following statement which is proven in Proposition 4.2 in \cite{ragh}.
\begin{Proposition} 
The top degree elements of the weight space \(W_{\nub}(\lb)\) have degree \(d^{\l}_{\nu}\).
\end{Proposition}

\begin{Definition}
Let  \(W_{\nub}^*(\lb)\)  be the  dual space of the vector space \(W_{\nub}(\lb)\), spanned by elements \(\{ \delta_w\}_{w \in W_{\nub}(\lb)}\). We define the grading on \(W_{\nub}^*(\lb)\) as follows:
 \[ 
 W_{\nub}^*(\lb) \{ 2r\} := \left( W_{\nub}(\lb) \{ d^{\l}_{\nu} - 2r\} \right)^*,
 \]
 and let \( W_{\nub}^*(\lb) = \bigoplus_{r \geq 0}  W_{\nub}^*(\lb) \{ 2r\} \).
 Then \begin{equation}
 W^*(\lb)  := \bigoplus_{\lb \geq \nub} W_{\nub}^*(\lb),
  \end{equation}
  is called the dual Weyl module of highest weight \(\lb\), together with a current algebra action
  \begin{equation} \label{dual}
(x \cdot \delta_w)v = \delta_w(\omega(x) v),
\end{equation}
where  \(v,w \in {W}(\lb)\),  \(x \in \mathbf{U}(\mathfrak{sl}_n[t])\), and \(\omega\) is the anti-involution on \(\mathbf{U}(\mathfrak{sl}_n[t])\) which sends \(F_{i,j}\) to \(E_{i,j}\) and fixes \(H_{i,j}\).
\end{Definition}

There is a non-degenerate bilinear pairing \( (\delta_w, v) \mapsto \delta_w(v)\) of degree \(d^{\l}_{\nu}\) between \(W_{\nub}^*(\lb)\)  and \(W_{\nub}(\lb)\).  Note that current algebra action on local Weyl modules preserves the degree, but the generators \(E_{i,j}\) and \(F_{i,j} \) does not preserve the degree on dual Weyl module. The action of \(H_{i,j}\colon W_{\nub}^*(\lb) \to W_{\nub}^*(\lb) \) increases the degree by \(2j\), while \(E_{i,j}\colon W_{\nub}^*(\lb) \to W_{\overline{\nu + \alpha_i}}^*(\lb)\) increases the degree by \(d^{\l}_{\nu} -d^{\l}_{\nu+\alpha_i}+2j = 2(-\nub_{i}-1+j)\), and \(F_{i,j}\colon W_{\nub}^*(\lb) \to W_{\overline{\nu- \alpha_i}}^*(\lb) \) increases the degree by \(d^{\l}_{\nu} -d^{\l}_{\nu-\alpha_i} +2j= 2(\nub_{i}-1+j)\).
  
  For the case \(n=2\) and \(\lb =2\omega\), the dual Weyl module \(W^*(\lb)\) and the current algebra action has the following graphical description:
 
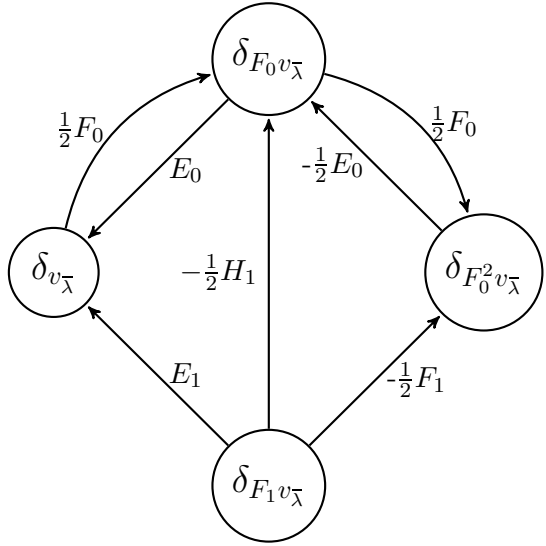
\begin{center}
\begin{tikzpicture} 
[->,>=stealth',shorten >=1pt,auto,node distance=4cm,
                    thick,main node/.style={circle, draw, font=\sffamily\large\bfseries}]
  \node[main node] (1) {\( \delta_{ F_{0} v_{\lb}}\) };
  \node[main node] (2) [below left  of = 1] {\(  \delta_{ v_{\lb}}\)};
  \node[main node] (3) [below right of=2] {\(\delta_{F_{1} v_{\lb}}\)};
  \node[main node] (4) [below right of=1] {\( \delta_{F^2_{0} v_{\lb} }\)};

  \path[every node/.style={font=\sffamily\small}]
    (1) edge node [right] {\(E_{0}\)} (2)
        edge [bend left] node[right] {\(\frac{1}{2} F_{0}\)} (4)
    (2)  
       edge [bend left] node[left] {\(\frac{1}{2} F_{0}\)} (1)
    (3) edge node [right] {\(E_{1}\)} (2)
        edge [ right] node[right] {-\(\frac{1}{2} F_{1}\)} (4)
         edge node  {\(-\frac{1}{2} H_{1}\)} (1)
   (4) 
        edge [ right] node[left] {-\(\frac{1}{2} E_{0}\)} (1);
\end{tikzpicture}
 \end{center}  

 
\section{2-representations}
While classical representation theory studies actions on vector spaces, 2-representation theory concerns with actions on \(\k\)-linear categories. Instead of linear maps, 2-representation theory studies linear functors and natural transformations between them. There has been significant progress in understanding 2-representations of \(\U\). In this section we define some important examples of 2-representations of \(\U\). 
\subsection{Integrable 2-representations }

\begin{Definition}
A 2-representation of \(\, \U\) is a graded, additive 2-functor \(\Phi \colon \U \to \cal{K}\) for some graded, additive 2-category \(\K\).  
\end{Definition}
We will denote the image of an object \(\nub \in  \Ob(\U)\) under the 2-functor by \(\cal{K}_{\nub}\). A 2-representation \(\Phi \colon \U \to \cal{K}\) is called {\em integrable} if the 1-morphisms \(\Phi (\mathcal{ E}_i)\) and \(\Phi (\mathcal{F}_i)\) are locally nilpotent for all \(i \in I\). Integrability condition implies that only finite number of objects can be non-zero in \(\cal{K}\). 
\begin{Definition} \label{equibb}
Let \( \Phi_1 \colon \U \to \cal{K}_1\) and \(\Phi_2 \colon \U \to \cal{K}_2\) be two representations.  A strongly equivariant functor \(\gamma  \colon \cal{K}_1 \to \cal{K}_2\)  is a collection of functors \(\gamma(\nub) \colon\Phi_1(\nub) \to \Phi_2(\nub) \) for each \(\nub \in \Ob(\U)\) together with natural isomorphisms of functors \(c_u \colon \gamma \circ \Phi_1(u) \cong  \Phi_2(u) \circ \gamma\) for every 1-morphism  \(u \in \U\) such that 
\begin{equation}\label{2mor}
c_v ( \Id_{\gamma} \circ \Phi_1(\alpha) ) \cong (\Phi_2(\alpha) \circ  \Id_{\gamma} ) c_u
\end{equation}
for every 2-morphism \(\alpha\colon u \to v\) in  \(\U\). In \eqref{2mor},  we use \(\circ\) for horizontal composition, and multiplication for vertical composition of 2-morphisms. We call \(\gamma\)  a strongly equivariant equivalence if \(\gamma(\nub) \) is an equivalence for each \(\nub \in \Ob(\U)\). 
\end{Definition}

Fix a highest weight \(\lb\). We define Rouquier's universal categorification \({\cal L}(\lb)\) of the simple \(\mathfrak{sl}_n\)-module \(V(\lb)\)  following Brundan-Davidson \cite{BD}.   Let \(\Phi_{\lb}\colon \U \to {\cal R}(\lb)\) be the 2-representation defined by \({\cal R}(\lb)_{\nub} = \text{Hom}_{\U}(\lb, \nub)\), in words, \(\Phi_{\lb}\) sends each \(\nub \in \Ob(\U)\) to a the graded category of 1-morphisms \(\text{Hom}_{\U}(\lb, \nub)\) in \(\U\). The 1-morphisms \(\Phi_{\lb}({\cal E_i})\)  and \(\Phi_{\lb}({\cal F_i})\)  are composing 1-morphisms on the left by \({\cal E_i}\)  and \({\cal F_i}\), and 2-morphisms horizontally on the left by  \( \xy 0;/r.15pc/:
  (0,0)*{\xybox{
    (-2,-4)*{};(-2,6)*{}  **\dir{-}?(1)*\dir{>} ;
    (-.5,-3)*{\scs i}; (1,1)*{\scs \nub};
       (-4,0)*{};(3,0)*{};
     }};
  \endxy \) and \( \xy 0;/r.15pc/:
  (0,0)*{\xybox{
    (-2,-4)*{};(-2,6)*{}  **\dir{-} ?(.75)*\dir{<}; 
    (-.5,-3)*{\scs i}; (1,1)*{\scs \nub};
       (-4,0)*{};(3,0)*{};
     }};
  \endxy \)  respectively. The 2-morphisms are generated by the 2-morphisms in \(\U\).
  
  Now let \({\cal I}(\lb)\) be the full, additive 2-subcategory of \({\cal R}(\lb)\) in which the objects are generated by objects of the form \(R \,{\cal E_i}\), \(i \in I\) for a 1-morphism \(R\) in \(\U\) from \(\overbar{\l+\alpha_i}\) to \(\nub\).
We define the 2-category \( {\cal V}(\lb) = {\cal R}(\lb) /{\cal I}(\lb) \) as a quotient of additive 2-categories. \( {\cal V}(\lb) \) has the same objects as in \( {\cal R}(\lb) \) and \(\text{Hom}_{{\cal V}(\lb) } (a,b)= \text{Hom}_{{\cal R}(\lb) }(a,b)/\text{Hom}_{{\cal I}(\lb) } (a,b) \). \(\Phi_{\lb}\colon \U \to {\cal R}(\lb)\) induces a well-defined 2-representation \(\Phi_{\lb}\colon \U \to {\cal V}(\lb)\). 

There is another construction of a 2-representation called the {\em minimal categorification } \(\Phi^{\min}_{\lb}\colon\U \to {\cal V}_{\min}(\lb) \) of the irreducible \(\mathfrak{sl}_n\)-module \(V(\lb)\). The image \(\Phi^{\min}_{\lb} (\lb) \) of the object \(\lb\) is isomorphic to the ground field \(\k\), and the image of 2-morphisms of \(\U\) in \({\cal V}_{\min}(\lb) \) is finite-dimensional. Actually, Rouquier \cite{rouq, rouq1} proves that any 2-representation \(\Phi\) which categorifies \(V(\lb)\) such that \(\Phi(\lb) \simeq \k \) is strongly equivariantly equivalent to  \(\Phi^{\min}_{\lb}\).

The 2-representations \(\Phi^{}_{\lb}\) and \(\Phi^{\min}_{\lb}\) are integrable. They are related to module categories of cyclotomic KLR algebras as we will see next.

\subsection{Cyclotomic KLR algebras}

In this subsection we define the cyclotomic quotients of KLR algebras. Their module category is a 2-representation and has an interesting trace decategorification. We start by defining KLR algebras using the diagrammatics that is given in \cite{lauda0}.

Let \(\beta = \sum_{i \in I}^{} \beta_i \alpha_i \) be a linear combination of simple roots over \(\Z^+\). We  define  \(R(\beta)\) to be the graded \(\k\)-algebra generated by diagrams
\begin{equation} \label{gens}
x_i^r\; = \; \xy 0;/r.18pc/:
  (0,0)*{\xybox{
    (-2,-4)*{};(-2,6)*{} **\dir{-}?(1)*\dir{>} ;
    (-1,-3)*{\scs i}; (-2,2)*{\bullet}; (1,2)*{\scs r};
       (-10,0)*{};(6,0)*{};
     }};
  \endxy,
   \;\;  i \in I, r \geq 0 \quad \quad \quad \;\;
\tau _{ij}\;=\;\xy 0;/r.18pc/:
  (0,0)*{\xybox{
    (-4,-4)*{};(4,6)*{} **\crv{(-4,-1) & (4,1)}?(1)*\dir{>};
    (4,-4)*{};(-4,6)*{} **\crv{(4,-1) & (-4,1)}?(1)*\dir{>};
    (-5.5,-3)*{\scs i};
     (5.5,-3)*{\scs j};
     (-10,0)*{};(10,0)*{};
     }};
  \endxy,  \;\;  i,j \in I 
\end{equation}
with \(\beta_i\) number of \(i\) labeled strands subject to the relations \eqref{eq_r2_ij-gen}-\eqref{eq_r3_hard-gen}. The multiplication is given by stacking two diagrams on top of each other from bottom to top if the labels of strands at the conjunction match, and zero otherwise. The diagrams are subject to the following relations:

\begin{equation}
 \vcenter{\xy 0;/r.17pc/:
    (-4,-4)*{};(4,4)*{} **\crv{(-4,-1) & (4,1)}?(1)*\dir{};
    (4,-4)*{};(-4,4)*{} **\crv{(4,-1) & (-4,1)}?(1)*\dir{};
    (-4,4)*{};(4,12)*{} **\crv{(-4,7) & (4,9)}?(1)*\dir{};
    (4,4)*{};(-4,12)*{} **\crv{(4,7) & (-4,9)}?(1)*\dir{};
    (4,12); (4,13) **\dir{-}?(1)*\dir{>};
    (-4,12); (-4,13) **\dir{-}?(1)*\dir{>};
  (-5.5,-3)*{\scs i};
     (5.5,-3)*{\scs j};
 \endxy}
 \qquad = \qquad
 \left\{
 \begin{array}{ccc}
 0 &  &  \text{if $a_{ij}=2$,}\\ \\
     \xy 0;/r.17pc/:
  (3,6);(3,-6) **\dir{-}?(0)*\dir{<}+(2.3,0)*{};
  (-3,6);(-3,-6) **\dir{-}?(0)*\dir{<}+(2.3,0)*{};
  (-5,-6)*{\scs i};     (5.1,-6)*{\scs j};
 \endxy &  &  \text{if $a_{ij}=0$,}\\ \\
 \vcenter{\xy 0;/r.17pc/:
  (3,6);(3,-6) **\dir{-}?(0)*\dir{<}+(2.3,0)*{};
  (-3,6);(-3,-6) **\dir{-}?(0)*\dir{<}+(2.3,0)*{};
  (-3,2)*{\bullet};(-6.5,5)*{};
  (-5,-6)*{\scs i};     (5.1,-6)*{\scs j};
 \endxy} \;\; + 
\;\;  \vcenter{\xy 0;/r.17pc/:
  (3,6);(3,-6) **\dir{-}?(0)*\dir{<}+(2.3,0)*{};
  (-3,6);(-3,-6) **\dir{-}?(0)*\dir{<}+(2.3,0)*{};
  (3,2)*{\bullet};(7,5)*{};
  (-5,-6)*{\scs i};     (5.1,-6)*{\scs j};
 \endxy}
   &  & \text{if $a_{ij} =-1$,}
 \end{array}
 \right. \label{e_r2_ij-gen}
\end{equation}
\begin{equation} \label{e_dot_slide_ii-gen-cyc}
\xy 0;/r.17pc/:
  (0,0)*{\xybox{
    (-4,-4)*{};(4,6)*{} **\crv{(-4,-1) & (4,1)}?(1)*\dir{>}?(.25)*{\bullet};
    (4,-4)*{};(-4,6)*{} **\crv{(4,-1) & (-4,1)}?(1)*\dir{>};
    (-5,-3)*{\scs i};
     (5.1,-3)*{\scs i};
     (-10,0)*{};(10,0)*{};
     }};
  \endxy
 \;\; -
\xy 0;/r.17pc/:
  (0,0)*{\xybox{
    (-4,-4)*{};(4,6)*{} **\crv{(-4,-1) & (4,1)}?(1)*\dir{>}?(.75)*{\bullet};
    (4,-4)*{};(-4,6)*{} **\crv{(4,-1) & (-4,1)}?(1)*\dir{>};
    (-5,-3)*{\scs i};
     (5.1,-3)*{\scs i};
     (-10,0)*{};(10,0)*{};
     }};
  \endxy
\;\; =  \xy 0;/r.17pc/:
  (0,0)*{\xybox{
    (-4,-4)*{};(4,6)*{} **\crv{(-4,-1) & (4,1)}?(1)*\dir{>};
    (4,-4)*{};(-4,6)*{} **\crv{(4,-1) & (-4,1)}?(1)*\dir{>}?(.75)*{\bullet};
    (-5,-3)*{\scs i};
     (5.1,-3)*{\scs i};
     (-10,0)*{};(10,0)*{};
     }};
  \endxy
\;\;  -
  \xy 0;/r.17pc/:
  (0,0)*{\xybox{
    (-4,-4)*{};(4,6)*{} **\crv{(-4,-1) & (4,1)}?(1)*\dir{>} ;
    (4,-4)*{};(-4,6)*{} **\crv{(4,-1) & (-4,1)}?(1)*\dir{>}?(.25)*{\bullet};
    (-5,-3)*{\scs i};
     (5.1,-3)*{\scs i};
     (-10,0)*{};(12,0)*{};
     }};
  \endxy
  \;\; = \;\;
    \xy 0;/r.17pc/:
  (0,0)*{\xybox{
    (-4,-4)*{};(-4,6)*{} **\dir{-}?(1)*\dir{>} ;
    (4,-4)*{};(4,6)*{} **\dir{-}?(1)*\dir{>};
    (-5,-3)*{\scs i};
     (5.1,-3)*{\scs i};
          (-10,0)*{};(12,0)*{};
     }};
  \endxy.
\end{equation}

For $i \neq j$ we also have
\begin{equation} \label{e_dot_slide_ij-gen}
\xy 0;/r.17pc/:
  (0,0)*{\xybox{
    (-4,-4)*{};(4,6)*{} **\crv{(-4,-1) & (4,1)}?(1)*\dir{>}?(.75)*{\bullet};
    (4,-4)*{};(-4,6)*{} **\crv{(4,-1) & (-4,1)}?(1)*\dir{>};
    (-5,-3)*{\scs i};
     (5.1,-3)*{\scs j};
     (-10,0)*{};(10,0)*{};
     }};
  \endxy
 \;\; =
\xy 0;/r.17pc/:
  (0,0)*{\xybox{
    (-4,-4)*{};(4,6)*{} **\crv{(-4,-1) & (4,1)}?(1)*\dir{>}?(.25)*{\bullet};
    (4,-4)*{};(-4,6)*{} **\crv{(4,-1) & (-4,1)}?(1)*\dir{>};
    (-5,-3)*{\scs i};
     (5.1,-3)*{\scs j};
     (-10,0)*{};(10,0)*{};
     }};
  \endxy,
\qquad  \xy 0;/r.17pc/:
  (0,0)*{\xybox{
    (-4,-4)*{};(4,6)*{} **\crv{(-4,-1) & (4,1)}?(1)*\dir{>};
    (4,-4)*{};(-4,6)*{} **\crv{(4,-1) & (-4,1)}?(1)*\dir{>}?(.75)*{\bullet};
    (-5,-3)*{\scs i};
     (5.1,-3)*{\scs j};
     (-10,0)*{};(10,0)*{};
     }};
  \endxy
\;\;  =
  \xy 0;/r.17pc/:
  (0,0)*{\xybox{
    (-4,-4)*{};(4,6)*{} **\crv{(-4,-1) & (4,1)}?(1)*\dir{>} ;
    (4,-4)*{};(-4,6)*{} **\crv{(4,-1) & (-4,1)}?(1)*\dir{>}?(.25)*{\bullet};
    (-5,-3)*{\scs i};
     (5.1,-3)*{\scs j};
     (-10,0)*{};(12,0)*{};
     }};
  \endxy.
\end{equation}

If $i \neq k$ and $a_{ij} \geq 0$, the relation
\begin{equation}
\vcenter{\xy 0;/r.17pc/:
    (-4,-4)*{};(4,4)*{} **\crv{(-4,-1) & (4,1)}?(1)*\dir{};
    (4,-4)*{};(-4,4)*{} **\crv{(4,-1) & (-4,1)}?(1)*\dir{};
    (4,4)*{};(12,12)*{} **\crv{(4,7) & (12,9)}?(1)*\dir{};
    (12,4)*{};(4,12)*{} **\crv{(12,7) & (4,9)}?(1)*\dir{};
    (-4,12)*{};(4,20)*{} **\crv{(-4,15) & (4,17)}?(1)*\dir{};
    (4,12)*{};(-4,20)*{} **\crv{(4,15) & (-4,17)}?(1)*\dir{};
    (-4,4)*{}; (-4,12) **\dir{-};
    (12,-4)*{}; (12,4) **\dir{-};
    (12,12)*{}; (12,20) **\dir{-};
    (4,20); (4,21) **\dir{-}?(1)*\dir{>};
    (-4,20); (-4,21) **\dir{-}?(1)*\dir{>};
    (12,20); (12,21) **\dir{-}?(1)*\dir{>};
   (-6,-3)*{\scs i};
  (6,-3)*{\scs j};
  (15,-3)*{\scs k};
\endxy}
 \;\; =\;\;
\vcenter{\xy 0;/r.17pc/:
    (4,-4)*{};(-4,4)*{} **\crv{(4,-1) & (-4,1)}?(1)*\dir{};
    (-4,-4)*{};(4,4)*{} **\crv{(-4,-1) & (4,1)}?(1)*\dir{};
    (-4,4)*{};(-12,12)*{} **\crv{(-4,7) & (-12,9)}?(1)*\dir{};
    (-12,4)*{};(-4,12)*{} **\crv{(-12,7) & (-4,9)}?(1)*\dir{};
    (4,12)*{};(-4,20)*{} **\crv{(4,15) & (-4,17)}?(1)*\dir{};
    (-4,12)*{};(4,20)*{} **\crv{(-4,15) & (4,17)}?(1)*\dir{};
    (4,4)*{}; (4,12) **\dir{-};
    (-12,-4)*{}; (-12,4) **\dir{-};
    (-12,12)*{}; (-12,20) **\dir{-};
    (4,20); (4,21) **\dir{-}?(1)*\dir{>};
    (-4,20); (-4,21) **\dir{-}?(1)*\dir{>};
    (-12,20); (-12,21) **\dir{-}?(1)*\dir{>};
  (-14,-3)*{\scs i};
  (-6,-3)*{\scs j};
  (6,-3)*{\scs k};
\endxy}
 \label{eq_r3_easy-gen}
\end{equation}
holds. Otherwise if $a_{ij} =-1$, we have
\begin{equation}
\vcenter{\xy 0;/r.17pc/:
    (-4,-4)*{};(4,4)*{} **\crv{(-4,-1) & (4,1)}?(1)*\dir{};
    (4,-4)*{};(-4,4)*{} **\crv{(4,-1) & (-4,1)}?(1)*\dir{};
    (4,4)*{};(12,12)*{} **\crv{(4,7) & (12,9)}?(1)*\dir{};
    (12,4)*{};(4,12)*{} **\crv{(12,7) & (4,9)}?(1)*\dir{};
    (-4,12)*{};(4,20)*{} **\crv{(-4,15) & (4,17)}?(1)*\dir{};
    (4,12)*{};(-4,20)*{} **\crv{(4,15) & (-4,17)}?(1)*\dir{};
    (-4,4)*{}; (-4,12) **\dir{-};
    (12,-4)*{}; (12,4) **\dir{-};
    (12,12)*{}; (12,20) **\dir{-};
    (4,20); (4,21) **\dir{-}?(1)*\dir{>};
    (-4,20); (-4,21) **\dir{-}?(1)*\dir{>};
    (12,20); (12,21) **\dir{-}?(1)*\dir{>};
    (-6,-3)*{\scs i};
  (6,-3)*{\scs j};
  (15,-3)*{\scs i};
\endxy}
\quad - \quad
\vcenter{\xy 0;/r.17pc/:
    (4,-4)*{};(-4,4)*{} **\crv{(4,-1) & (-4,1)}?(1)*\dir{};
    (-4,-4)*{};(4,4)*{} **\crv{(-4,-1) & (4,1)}?(1)*\dir{};
    (-4,4)*{};(-12,12)*{} **\crv{(-4,7) & (-12,9)}?(1)*\dir{};
    (-12,4)*{};(-4,12)*{} **\crv{(-12,7) & (-4,9)}?(1)*\dir{};
    (4,12)*{};(-4,20)*{} **\crv{(4,15) & (-4,17)}?(1)*\dir{};
    (-4,12)*{};(4,20)*{} **\crv{(-4,15) & (4,17)}?(1)*\dir{};
    (4,4)*{}; (4,12) **\dir{-};
    (-12,-4)*{}; (-12,4) **\dir{-};
    (-12,12)*{}; (-12,20) **\dir{-};
    (4,20); (4,21) **\dir{-}?(1)*\dir{>};
    (-4,20); (-4,21) **\dir{-}?(1)*\dir{>};
    (-12,20); (-12,21) **\dir{-}?(1)*\dir{>};
  (-14,-3)*{\scs i};
  (-6,-3)*{\scs j};
  (6,-3)*{\scs i};
\endxy}
 \;\; =\;\;
 t_{ij} \;\;
\xy 0;/r.17pc/:
  (4,12);(4,-12) **\dir{-}?(0)*\dir{<};
  (-4,12);(-4,-12) **\dir{-}?(0)*\dir{<}?(.25)*\dir{};
  (12,12);(12,-12) **\dir{-}?(0)*\dir{<}?(.25)*\dir{};
  (-6,-9)*{\scs i};     (6.1,-9)*{\scs j};
  (14,-9)*{\scs i};
 \endxy.
 \label{e_r3_hard-gen}
\end{equation}
The multiplication is zero if the labels at the conjunction of diagrams do not match. We will call both  \(R(\beta)\) and \(R = \bigoplus_{\beta} R(\beta)\) KLR algebra when there is no confusion.
Let \(R(\beta) \text{-} \mathrm{ pMod}\) be the category of graded, finitely-generated, projective \(R(\beta)\)-modules. In \cite{lauda0} Proposition 3.18, Khovanov and Lauda showed that 
\[ R\text{-} \mathrm{ pMod}=\bigoplus_{\beta} R(\beta)\text{-}\mathrm{ pMod}\]
categorifies the positive half $\UA^+_q(\mathfrak{sl}_n)$ of $\UA_q(\mathfrak{sl}_n)$.

\begin{Definition} \label{ccc}
For a dominant   \(\mf{sl}_n\) weight \(\lb \in X\), the cyclotomic quotient \(R^{\lb}(\beta)\) of the KLR algebra \(R^{}(\beta)\)  is the quotient, in which leftmost strand with color \(i\) and \(\lb_i\) dots is set to be zero for all \(i \in I\):
 \begin{equation} 
\vcenter{\xy 0;/r.17pc/:
  (-3,6);(-3,-6) **\dir{-}?(0)*\dir{<}+(2.3,0)*{};
   (5,0)*{\cdots};  
  (-3,2)*{\bullet};(-6.5,5)*{};
  (-5,-4)*{\scs i};    (-7,3)*{\scs \lb_i}; 
 \endxy}\, \,  =\, \, 0\; .
\end{equation}
The algebra 
\begin{equation}
R^{\lb} := \bigoplus_{\beta \in \N [I]} R^{\lb}(\beta).
\end{equation}
 is called cyclotomic KLR algebra and its summand \( R^{\lb}(\beta)\) is called cyclotomic KLR algebra of rank \(\beta\). 
\end{Definition}

 \( R^{\lb} \) is a graded, unital and finite dimensional algebra. The algebra \( R^{\lb} \) is interesting for the following reason. There is a 2-representation  \(\U \to R^{\lb}\text{-}\mathrm{ pMod} \) (see Theorem 3.17 in \cite{ben} for the details). \(R^{\lb}\text{-} \mathrm{ pMod}\)  categorifies the irreducible finite dimensional $\, \UA_q(\mathfrak{sl}_n)$-module with the highest weight \(\lb\). This result was first conjectured by Khovanov and Lauda in \cite{lauda0}, and proved later independently by Brundan and Kleshchev \cite{BK2}, Kang and Kashiwara  \cite{kk}, Webster \cite{ben}.  

We now define a deformation of \( R^{\lb} \) according to Webster \cite{ben}.  Let \(\End_{\U} (\oplus_{\mathbf{i}} 1_{\lb} \mathcal{E}_{\mathbf{i}})\) be the endomorphism algebra of   directs sums of 1-morphisms in \(\U\) of form 
\(1_{\lb} \mathcal{E}^{a_1}_{i_1} \dots  \mathcal{E}^{a_l}_{i_l}   \), \(l \geq 0\), \(a_1, \dots, a_l \geq 0\), \(i_1, \dots i_l \in I\). 
\begin{Definition}
Let \(\check{R}^{\lb}\) be the quotient of \(\End_{\U} (\oplus_{\mathbf{i}} 1_{\lb} \mathcal{E}_{\mathbf{i}})\) by the relations 
 \begin{equation} 
\vcenter{\xy 0;/r.20pc/:
  (-3,8);(-3,-8) **\dir{-}?(0)*\dir{<}+(2.3,0)*{};
   (5,0)*{\cdots};  
  (-3,2)*{\bullet};(-6.5,5)*{};
  (-2,-7)*{\scs i};    (0,3)*{\scs \lb_i}; (-10,4)*{\scs \lb}; 
 \endxy}\, \, + \, \, 
 \vcenter{\xy 0;/r.20pc/:
  (-3,8);(-3,-8) **\dir{-}?(0)*\dir{<}+(2.3,0)*{};
   (5,0)*{\cdots};  
  (-3,2)*{\bullet};(-6.5,5)*{};
  (-2,-7)*{\scs i};    (2,3)*{\scs \lb_i-1}; (-12,9)*{\scs \lb}; 
  (-12,0)*{\iccbub{\spadesuit+1}{i}};
 \endxy}\, \, + 
 \dots +\, \, 
 \vcenter{\xy 0;/r.20pc/:
  (-3,8);(-3,-8) **\dir{-}?(0)*\dir{<}+(2.3,0)*{};
   (5,0)*{\cdots};  
  (-2,-7)*{\scs i};  
  (-12,9)*{\scs \lb}; 
  (-12,0)*{\iccbub{\spadesuit+\lb_i-1}{i}};
 \endxy}\, \,
  =\, \, 0\; ,
\end{equation} 
\begin{equation}
\xy 0;/r.20pc/:
  (4,7)*{\scs \lb};
  (2,-2)*{\iccbub{r}{i}};
 \endxy \;= 0 \quad \text{for all } r \geq 0.
\end{equation}
for all \(i \in I\). 
\end{Definition}
The following result is due to Webster \cite{ben}, section 3.5:
\begin{Theorem}
The algebra \(\check{R}^{\lb}\) is a flat deformation of \({R}^{\lb}\). 
\end{Theorem}
Unlike \({R}^{\lb}\), \(\check{R}^{\lb}\)  is infinite-dimensional.  The category \(\check{R}^{\lb}\text{-}\mathrm{ pMod}\)  also admits an action of \(\U\), and this action  gives another categorification of \(V(\lb)\).
The next theorem, proven by Rouquier (see Theorem 4.19 in \cite{brun1}), relates universal categorification 
\({\cal V}(\lb)\) and  \(\check{R}^{\lb}\text{-} pMod\).  
\begin{Theorem}
There is a strongly equivariant equivalence between \(\check{R}^{\lb} \text{-} \mathrm{ pMod}\) and the Karoubi envelope of  \({\cal V}(\lb)\).
\end{Theorem}
The next theorem shows the universality property of \(\check{R}^{\lb}\text{-} \mathrm{ pMod}\) (see Corollary 3.27 in Webster \cite{ben}). It emphasizes the importance of  \(\check{R}^{\lb}\text{-}\mathrm{ pMod}\) in higher representation theory and its relation to other integrable 2-representations.
\begin{Theorem} \label{univv}
Given a 2-representation \(\Phi \colon \U \to \K\) for some additive, idempotent complete 2-category such that \(\Phi(\E_i) (\K_{\lb})=0\) for all \(i \in I\) there is a unique strongly equivariant functor \(\check{R}^{\lb}\text{-} \mathrm{ pMod} \to \K\).
\end{Theorem}

Let  \(\Phi \colon \U \to \cal{K}\) be a 2-representation. By functoriality of the trace,  \(\Phi\)  induces a \(\mathbb{C}\)-algebra homomorphism 
\begin{equation}
 \Tr \U  \xrightarrow{\Tr \Phi } \Tr \K .
\end{equation}
We compose \(\Tr \Phi \) with  \(\dot{\bfU}(\mf{sl}_n[t]) \xrightarrow{\rho} \Tr \U \) to obtain
 a \(\mathbb{C}\)-algebra homomorphism
 \begin{equation} \label{dee}
 \dot{\bfU}(\mf{sl}_n[t])   \xrightarrow{\Tr \Phi } \Tr \K .
\end{equation}

Since \(\K\) is a cyclic 2-category,  \(\Tr \K \) acts on the center of objects \(Z(\K)\). Thus, combining this action with the algebra homomorphism \eqref{dee}, we deduce that  \(Z(\K)\) is a current algebra module. 

Let \(Z ( {R}^{\lb}\text{-}\mathrm{ pMod}) = \bigoplus_{\nub \leq \lb} Z ({R}^{\lb}(\lb -\nub) ) \text{-} \mathrm{ pMod}\) be the center of objects of \( {R}^{\lb} \text{-} \mathrm{ pMod}\), and it is the same as the center 
\(Z  ( {R}^{\lb} ) \) of the cyclotomic KLR algebra \(R^{\lb}\). Hence, by the homomorphism \eqref{dee}, both \(\Tr  {R}^{\lb} \) and \(Z ( {R}^{\lb})\) becomes a current algebra module. 

There is a map \(t \colon R^{\lb} \to \mathbb{C}\) which is a Frobenius trace (see Theorem 3.18, \cite{ben}). This means that the kernel of the map \(t \colon R^{\lb} \to \mathbb{C}\)  contains no nonzero left ideal of \(R^{\lb}\). The map can be realized diagrammatically by closing the diagram representing an element in \(R^{\lb}\) to the right. We obtain an element in \(R^{\lb}(0) \simeq \mathbb{C}\):
\begin{equation}
\hackcenter{\begin{tikzpicture}
    \draw[very thick] (-.55,0) -- (-.55,1.5);
    \draw[very thick] (0,0) -- (0,1.5);
    \draw[very thick] (.55,0) -- (.55,1.5);
    \draw[fill=white!20,] (-.8,.5) rectangle (.8,1);
    \node () at (0,.75) {$a$};
      \node () at (-1.2,.75) {$\lb$};
\end{tikzpicture}}
\quad \longmapsto
\quad
\hackcenter{\begin{tikzpicture}
    \draw[very thick] (-1.65,-.7) -- (-1.65, .7).. controls ++(0,.95) and ++(0,.95) .. (1.65,.7)
        to (1.65,-.7) .. controls ++(0,-.95) and ++(0,-.95) .. (-1.65,-.7);
    \draw[very thick] (-1.1,-.55) -- (-1.1,.55) .. controls ++(0,.65) and ++(0,.65) .. (1.1,.55)
        to (1.1,-.55) .. controls ++(0,-.65) and ++(0,-.65) .. (-1.1, -.55);
    \draw[very thick] (-.55,-.4) -- (-.55,.4) .. controls ++(0,.35) and ++(0,.35) .. (.55,.4)
        to (.55, -.4) .. controls ++(0,-.35) and ++(0,-.35) .. (-.55,-.4);
    \draw[fill=white!20,] (-1.8,-.25) rectangle (-.4,.25);
    \node () at (-1,0) {$a$};  \node () at (-2.2,0) {$\lb$};
\end{tikzpicture}} \; 
 \end{equation}
The map \(t \colon R^{\lb} \to \mathbb{C}\) is symmetric, that is, \(t(aa')=t(a'a)\) for all \(a, a' \in  R^{\lb}\). This map induces a non-degenerate \(Z( R^{\lb}(\beta) )\)-bilinear form \(t\colon Z( R^{\lb}(\beta)) \times \Tr  R^{\lb}(\beta)  \to \mathbb{C} \) of degree \(-d^{\lb}_{\lb - \beta} \) (see \cite{vas}, Prop. 3.10). 

We will need the following crucial result, which is proven by independently Beliakova-Habiro-Lauda-Webster \cite{lcurr}, Theorem 7.3 and Shan-Varagnolo-Vasserot \cite{vas}, Theorem 3.34. 
\begin{Theorem} \label{thmm}
As a current algebra module, \(\Tr {R}^{\lb} \) is isomorphic to the local Weyl module  \(W^{}(\lb)\). Dually,  \(Z ({R}^{\lb} )\) is isomorphic to the dual local Weyl module \(W^*(\lb)\).
\end{Theorem}

\section{The  2-category \(\textbf{Pol}_N\) }
In this section we will construct an integrable  2-representation of \(\U\) using symmetric polynomials. This is very similar to Khovanov-Lauda's construction of the 2-representation using cohomology ring of flag varieties. We will compare this 2-representation  to Rouquier's universal categorification. 
\subsection{Symmetric polynomials}

We denote the set of  \(n\)-compositions and \(n\)-partitions of \(N\) by $\mathcal{P}^{}(n, N)$ and $\mathcal{P}^+(n, N)$, respectively. 
The algebra of polynomials $P=\mathbb{C} [X_1, X_2,\dots, X_N]$  in \(N\) commuting variables admits an action of the symmetric group $S_N$. The action of \( \sigma \in S_N\) on a polynomial \( q(X_1, X_2, \dots, X_N) \in P\) is defined as
\[ \sigma q(X_1, X_2, \dots, X_N) = q(X_{\sigma^{}(1)}, X_{\sigma^{}(2)}, \dots, X_{\sigma^{}(N)} ).\]
We will assume that all variables are of degree 2.  By $Sym_N =P^{S_N }$, we denote the subalgebra of symmetric polynomials.
Consider  a \(n\)-composition \(\nu = (\nu_1, \nu_2, \dots, \nu_n) \in \mathcal{P}^{}(n, N)\) and the parabolic subgroup \(S_\nu = S_{\nu_1} \times S_{\nu_2} \times \dots \times S_{\nu_n} \). To be clear, there are many subgroups of \(S_N\) isomorphic to \(S_\nu\), but all of them are conjugate.  Let \(P_\nu= P^{S_\nu }\) be the subalgebra of \(P\) which is symmetric separately in the first \(\nu_1\) variables, in the second \(\nu_2\) variables and so on.  We have \(P_{(1, 1,\dots, 1)}=P \) and \(P_{(N,0 , \dots, 0)}=Sym_N \). 

For \(r \geq 0\) and \(1 \leq i \leq n\), let \(e_r(\nu; i) \), \(h_r(\nu; i)\) and \(p_r(\nu; i)\) be the \(r\)-th elementary, complete and power sum symmetric polynomials respectively, in variables  
$$\Omega_i =  \left\{X_k \, | \, \nu_1+\dots +\nu_{i-1}+1 \leq k \leq \nu_1+\dots+\nu_i  \right\}.  $$
For example, if \(N=6\), \(n=3\), and \(\nu =(2, 1, 3)\), then \( e_r (\nu; 1) = e_r(X_1, X_2)\), \( e_1 (\nu; 2) = e_1(X_3)=X_3\), and \( e_r (\nu; 3) = e_r(X_4, X_5, X_6)\). The algebra \(P_\nu\) is generated by \(\{e_r(\nu; i)\}_{  r \geq 0, \, 1 \leq i \leq n },\) as well as by \(\{h_r(\nu; i) \}_{  r \geq 0, \, 1 \leq i \leq n }\) or by  \(\{p_r(\nu; i) \}_{  r \geq 0, \, 1 \leq i \leq n }\).

We extend our notation by defining the following symmetric polynomials:
\begin{align} \label{symm}
h_r(\nu;i_1,\dots,i_m) &= \sum_{r_1+\cdots+r_m = r}
h_{r_1}(\nu;i_1)h_{r_2}(\nu;i_2) \cdots h_{r_m}(\nu;i_m),\\
e_r(\nu;i_1,\dots,i_m) &= \sum_{r_1+\cdots+r_m = r}
e_{r_1}(\nu;i_1)e_{r_2}(\nu;i_2) \cdots e_{r_m}(\nu;i_m),
\end{align}
where \(1 \leq i_1 < \cdots < i_m \leq n\), \(m>0\).
It is not difficult to see that  \(e_r(\nu; i_1, i_2, \dots, i_m)\) and  \(h_r(\nu; i_1, i_2, \dots, i_m)\) are the \(r\)-th elementary and complete symmetric polynomials, respectively, in variables \(\Omega_{i_1} \cup \dots \cup \Omega_{i_m}\). This follows from the following well-known identities about symmetric polynomials, which can be found in \cite{mcd}. We have 

\begin{equation}\label{ma1}
\sum_{s=0}^r (-1)^{s} e_{s}(x_1,\dots,x_n) h_{r-s}(x_1,\dots,x_n) = 0 \quad  \text{for all } r \geq 1, 
\end{equation}
and for $r \geq 0$ we  have 
\begin{align}\label{brad122}
h_r(X_1,\dots,X_p,Y_1,\dots,Y_q) &= \sum_{s=0}^r h_{s}(X_1,\dots,X_p) h_{r-s}(Y_{1},\dots,Y_q),\\
e_r(X_1,\dots,X_p,Y_1,\dots,Y_q) &= \sum_{s=0}^r e_{s}(X_1,\dots,X_p) e_{r-s}(Y_{1},\dots,Y_q),\label{brad212}
\end{align}
\begin{align}\label{brad222}
h_r(Y_{1},\dots,Y_q)&=\sum_{s=0}^r (-1)^{s} e_{s}(X_1,\dots,X_p) 
h_{r-s}(X_1,\dots,X_p,Y_{1},\dots,Y_q),\\
e_r(Y_{1},\dots,Y_q)&=\sum_{s=0}^r (-1)^{s} 
h_{s}(X_1,\dots,X_p)e_{r-s}(X_1,\dots,X_p,Y_{1},\dots,Y_q).
\end{align}

By convention, $e_r(X_1,\dots,X_p) = h_r(X_1,\dots,X_p)=p_r(X_1,\dots,X_p) = 0$ for $r < 0$, and $e_0(X_1,\dots,X_p) = h_0(X_1,\dots,X_p) =p_0(X_1,\dots,X_p)= 1$.  Moreover, $e_r(X_1,\dots,X_p)= 0$ holds for all $r > p$.

Let \(\nu=(\nu_1, \dots , \nu_n) \in \mathcal{P}(n, N)\) be a \(n\)-composition of a fixed positive integer \(N\) as before. 
By \( \nu-\alpha_i\in \mathcal{P}(n,N) \) and \( \nu+\alpha_i\in \mathcal{P}(n,N) \), we mean the linear sum 
\[\nu-\alpha_i= (\nu_1, \dots, \nu_{i-1}, \nu_i-1, \nu_{i+1}+1, \nu_{i+2}, \dots, \nu_n), \]
\[\nu+\alpha_i= (\nu_1, \dots, \nu_{i-1}, \nu_i+1, \nu_{i+1}-1, \nu_{i+2}, \dots, \nu_n), \]
if all components are non-negative, otherwise we set them to \(\emptyset\). 
We will also need the following compositions. 
Define the compositions  \((\nu-\alpha_i, \nu) \in \mathcal{P}(n+1,N) \)  and \((\nu+\alpha_i, \nu) \in \mathcal{P}(n+1,N) \)

\[(\nu-\alpha_i, \nu) = (\nu_1, \dots, \nu_{i-1}, \nu_i-1,1, \nu_{i+1}, \nu_{i+2}, \dots, \nu_n), \] 
\[(\nu+\alpha_i, \nu)  = (\nu_1, \dots, \nu_{i-1}, \nu_i,1, \nu_{i+1}-1, \nu_{i+2}, \dots, \nu_n), \] 
if all components are non-negative,  and \(\emptyset\) otherwise.
Recall the definition of the ring \(P_\nu = P^{S_\nu}\). We can also define the rings \(P_{\nu-\alpha_i, \nu}\) and \(P_{\nu+\alpha_i, \nu}\) corresponding to the compositions \((\nu-\alpha_i, \nu)\) and \((\nu+\alpha_i, \nu)\) as well. 
Notice that \(P_{\nu-\alpha_i, \nu}\) is a free \(P_\nu\)-module with basis \(\{ 1, X_{k_i}, \dots ,X_{k_i}^{\nu_i-1} \}\), where \(k_i=\sum_{j=1}^i \nu_j\), and a free \(P_{\nu-\alpha_i}\)-module with basis \(\{ 1, X_{k_i}, \dots, X_{k_i}^{\nu_{i+1}} \}\). This is due to the fact that \( P^{S_{l-1}}\) is a free \(P^{S_l}\) -module with basis 
\( \{ 1, X_l, X_l^2, \dots, X^{l-1}\}\).

\begin{Example}
If \(\nu =(2, 1, 3)\) and \(i=1\), then  \(\nu- \alpha_1  =(1, 2, 3)\), \(k_1 =2\) and \((\nu- \alpha_1, \nu)  =(1, 1,1, 3)\) and we have 
\[P_{\nu} = \la e_r(X_1, X_2), X_3 , e_r(X_4, X_5, X_6)  \ra ,\]
\[P_{\nu- \alpha_1} = \la X_1, e_r(X_2, X_3) , e_r(X_4, X_5, X_6)  \ra ,\]
\[P_{\nu- \alpha_1, \nu} = \la X_1, X_2,  X_3, e_r(X_4, X_5, X_6) \ra .\]
It is easy to see that we have \( X_1= e_1(X_1, X_2)-X_2\), and  \(P_{\nu- \alpha_1, \nu}\) is a free \(P_\nu\)-module with basis \( \left\{ 1, X_2 \right\}\).
\end{Example}

We agree that  \(P_{\nu-\alpha_i, \nu}\) is a right \(P_\nu\)-module and a left \(P_{\nu-\alpha_i}\)-module, although this is artificial, since \(P_{\nu,\nu-\alpha_i}\) is a commutative ring. Similarly, \(P_{\nu+\alpha_i, \nu}\) is a free right \(P_\nu\)-module with basis \(\{ 1, X_{k_i+1}, \dots X_{k_i+1}^{\nu_i} \}\) and a free left \(P_{\nu+\alpha_i}\)-module with basis \(\{ 1, X_{k_i+1}, \dots X_{k_i+1}^{\nu_{i+1}-1} \}\).
The following relations hold in \(P_{\nu-\alpha_i, \nu}\):
\begin{equation} \label{bs1}
e_r(\nu; i+1)  =\sum_{s=0}^r (-1)^{s} e_{r-s}(\nu-\alpha_i ;i+1) X^s_{k_i}.
\end{equation}
\begin{equation} \label{bs2}
e_r(\nu-\alpha_i; i) = \sum_{s=0}^r (-1)^{s} X^s_{k_i}e_{r-s}(\nu;i),
\end{equation}
\begin{equation} \label{bs3}
e_r(\nu-\alpha_i; i+1) = X_{k_i}  e_{r-1}(\nu; i+1) + e_r(\nu; i+1) ,
\end{equation}
\begin{equation} \label{bs4}
e_r(\nu; i) =e_{r-1}(\nu-\alpha_i; i) X_{k_i} + e_r(\nu-\alpha_i; i) ,
\end{equation}

\begin{equation} \label{bs5}
h_r(\nu; i+1)  =  h_{r}(\nu-\alpha_i; i+1) -h_{r-1}(\nu-\alpha_i; i+1)X_{k_i},\\.
\end{equation}
\begin{equation}\label{bs6}
h_r(\nu-\alpha_i; i) = h_{r}(\nu; i) -X_{k_i}h_{r-1}(\nu; i),
\end{equation}
\begin{equation} \label{bs7}
h_r(\nu-\alpha_i; i+1) = \sum_{s=0}^r  X^s_{k_i}h_{r-s}(\nu;i+1),
\end{equation}
\begin{equation} \label{bs8}
h_r(\nu; i) = \sum_{s=0}^r h_{r-s}(\nu-\alpha_i;i) X^s_{k_i},
\end{equation}
which are easily seen to follow from the properties  (\ref{brad122}), (\ref{brad222}) of symmetric polynomials. Similarly, in 
\(P_{\nu+\alpha_i, \nu}\), we have the identities

\begin{equation} \label{s1}
e_r(\nu+\alpha_i; i+1)  =\sum_{s=0}^r (-1)^{s} X^s_{k_i+1}e_{r-s}(\nu ;i+1).
\end{equation}
\begin{equation} \label{s2}
e_r(\nu; i) = \sum_{s=0}^r (-1)^{s} X^s_{k_i+1}e_{r-s}(\nu +\alpha_i;i),
\end{equation}
\begin{equation} \label{s3}
e_r(\nu; i+1) = X_{k_i+1}  e_{r-1}(\nu+\alpha_i; i+1) + e_r(\nu+\alpha_i; i+1) ,
\end{equation}
\begin{equation} \label{s4}
e_r(\nu+\alpha_i; i) = X_{k_i+1}  e_{r-1}(\nu; i) +e_r(\nu; i) ,
\end{equation}

\begin{equation} \label{s5}
h_r(\nu+\alpha_i; i+1)  =  h_{r}(\nu; i+1) -X_{k_i+1}h_{r-1}(\nu; i+1),\\.
\end{equation}
\begin{equation}\label{s6}
h_r(\nu; i) = h_{r}(\nu+\alpha_i; i) -X_{k_i+1}h_{r-1}(\nu+\alpha_i; i),
\end{equation}
\begin{equation} \label{s7}
h_r(\nu; i+1) = \sum_{s=0}^r  X^s_{k_i+1}h_{r-s}(\nu+\alpha_i;i+1),
\end{equation}
\begin{equation} \label{s8}
h_r(\nu+\alpha_i ;i) = \sum_{s=0}^r  X^s_{k_i+1}h_{r-s}(\nu;i).
\end{equation}
The following identities will also be useful:
\begin{equation} \label{sp1}
X^r_{k_i} = \sum_{l=0}^r (-1)^{l}h_{r-l}(\nu-\alpha_i;i+1) e_{l} (\nu;i+1) ,
\end{equation}
\begin{equation} \label{sp2}
X^r_{k_i} = \sum_{l=0}^r (-1)^{l}  e_{l}(\nu-\alpha_i;i)h_{r-l} (\nu;i)  ,
\end{equation}
\begin{equation} \label{spe1}
X^r_{k_i+1} = \sum_{l=0}^r (-1)^{l} e_{l} (\nu+\alpha_i;i+1) h_{r-l}(\nu;i+1),
\end{equation}
\begin{equation} \label{spe2}
X^r_{k_i+1} = \sum_{l=0}^r (-1)^{r-l} h_{l} (\nu+\alpha_i;i) e_{r-l}(\nu;i).
\end{equation}

We can extend the definition of the bimodule \(P_{\nu+\alpha_i, \nu}\) by letting  
\[P_{\nu+\alpha_i \pm \alpha_j, \nu} =P_{\nu+\alpha_i \pm \alpha_j, \nu+\alpha_i} \otimes_{P_{\nu+\alpha_i}} P_{\nu+\alpha_i, \nu}, \]
and hence, inductively define the bimodule \(P_{\nu+\sum m_i\alpha_i, \nu}\) for \(m_i \in \Z\).
Note that these bimodules are graded, and by \(P_{\nu+\alpha_i, \nu} \la s \ra \) we will mean the degree shift of  \(P_{\nu+\alpha_i, \nu} \) by \(s \in \Z\).  \\

\subsection{The 2-functor  \(\U \to \textbf{Pol}_N \) }

Let \( \textbf{Pol}_N\) be the additive 2-category, in which 
\begin{itemize}
\item objects are \(P_\nu\) for all  \(\nu \in \cal{P}(n, N)\),
\item 1-morphisms between the objects  \(P_\nu\) and \(P_{\nu+\sum m_i\alpha_i}\)  are the finite direct sums of bimodules \(P_{\nu+\sum m_i\alpha_i, \nu}\),
\item 2-morphisms are matrices of bimodule homomorphisms.
\end{itemize}

The next theorem is due to Khovanov and Lauda \cite{lauda2}, however, we replace their \(\textbf{EqFlag}_N^* \) 2-category with the 2-category \(\textbf{Pol}_N \). While the 2-category \(\textbf{EqFlag}_N^* \)  is presented in terms  of \(GL_\mathbb{C}(n)\)-equivariant cohomology rings of partial flag varieties in \cite{lauda2}, we can pass to \(\textbf{Pol}_N \) by identifying the Chern classes with elementary symmetric polynomials.

\begin{Theorem} \label{ythm} \cite{lauda2}
There is a 2-representation \(\Theta_N \colon \U \to  \textbf{Pol}_N \), which is defined as follows:
\begin{itemize}
\item On objects
\[
  \nub \quad \mapsto \begin{cases}
P_\nu \quad \text{ if }  \nu \in \mathcal{P}(n, N) \text{ and }  \nub_i = \nu_i - \nu_{i+1} \text{ for each } i \in I,\\
0 \quad \text{otherwise.}
\end{cases}  
\]
\item On 1-morphisms 
\[ \onel \la t \ra \quad \mapsto \quad  P_\nu \la t \ra ,\]

\[ \cal{E}_{i}  \onel\la t \ra \quad  \mapsto  P_{\nu+\alpha_i, \nu} \la t \ra  ,\] 
  
\[  \cal{F}_{i}\onel\la t \ra \quad  \mapsto P_{\nu-\alpha_i, \nu}\la t \ra.\]

\item On 2-morphisms

\begin{eqnarray*}
    \Theta_N\left(\;\xy
    (0,-3)*{\bbpef{}};
    (8,-5)*{ \bfit{ $\overline{\nu}$ }}; (-6,-2)*{i}; 
    (-4,2)*{\scs  };
    (4,2)*{\scs  };
    \endxy \; \right)
    & : &
 \left\{
  \begin{array}{ccl}
    P_{\nu} \ & \longrightarrow &
    \left( P_{\nu, \nu+\alpha_i} \otimes_{ P_{\nu}+\alpha_i}  P_{\nu+\alpha_i, \nu}\right)\la 2\nu_i\ra,
    \\
     1 &\mapsto & c^{-1}_{\nub, i}  \,  \xsum{r=0}{\nu_i}(-1)^{\nu_i-r} X_{k_i+1}^{r} \otimes   e_{\nu_i-r}(\nu;i),  
     
      \end{array}
 \right.
     \label{def_eq_FE_G}\\ \\ \\
     \Theta_N\left(\;\xy
    (0,-3)*{\bbpfe{}};
    (8,-5)*{ \bfit{$\nub$}};(-6,-2)*{i}; 
    (-4,2)*{\scs };
    (4,2)*{\scs };
    \endxy\;\right)
     & : &
     \left\{
  \begin{array}{ccl}
   P_{\nu} \ & \longrightarrow &
    P_{\nu,\nu-\alpha_i} \otimes_{P_{\nu-\alpha_i}} P_{\nu-\alpha_i,\nu} \la 2\nu_{i+1}\ra,
    \\
     1 &\mapsto&
    \xsum{r=0}{\nu_{i+1}}(-1)^{\nu_{i+1}-r}   X_{k_i}^{r}  \otimes  e_{\nu_{i+1}-r}(\nu; i+1) ,
 \end{array}
 \right.
     \label{def_eq_EF_G}
\end{eqnarray*} \\
\begin{equation*}
   \Theta_N\left(\;\xy
    (0,0)*{\bbcef{}};
    (8,2)*{ \bfit{$\overline{\nu}$}}; (-6,-2)*{i}; 
    (-4,-5.5)*{\scs};
    (4,-5.5)*{\scs};
    \endxy\;\right)
 :
 \left\{
  \begin{array}{ccl}
     P_{\nu,\nu+\alpha_i} \otimes_{P_{\nu+\alpha_i}} P_{\nu+\alpha_i,\nu}   \quad \rightarrow \quad  P_{\nu}\la 2(1-\nu_{i+1}) \ra, 
    \\ \\
    X_{k_i+1}^{r_1} \otimes X_{k_i+1}^{r_2}
    \quad \mapsto   \quad 
    h_{r_1+r_2+1-\nu_{i+1}}(\nu;i+1),
 \end{array}
 \right.
   \label{def_FE_cap} \end{equation*} \\
   \begin{equation*}
  \Theta_N\left(\;\xy
    (0,0)*{\bbcfe{}};
    (8,2)*{ \bfit{$\nub$}}; (-6,-2)*{i}; 
    (-4,-5.5)*{\scs };
    (4,-5.5)*{\scs };
    \endxy \right)
 :
 \left\{
  \begin{array}{ccl}
      P_{\nu,\nu-\alpha_i} \otimes_{P_{\nu-\alpha_i}} P_{\nu-\alpha_i,\nu}  
       \quad  \rightarrow   \quad 
     P_{\nu} \la 2(1-\nu_i)\ra,
    \\ \\
     X_{k_i}^{r_1} \otimes X_{k_i}^{r_2}
      \quad \mapsto   \quad 
   c_{i, \nub} \, h_{r_1+r_2+1-\nu_i}(\nu;i),
 \end{array}
 \right.
\label{def_EF_cap}
\end{equation*} \\

 
\begin{eqnarray*} \label{eq_gamma_updot}
 \Theta_N\left(
 \xy
  (0,0)*{\lineu{i}};
  (0,4)*{\txt\large{$\bullet$}};
  (2,5)*{s};
 (4,0)*{ \nub};
 (-8,0)*{ \overbar{\nu +\alpha_i}};
 (-10,0)*{};(10,0)*{};
 \endxy\right)
&: &
 \left\{
\begin{array}{ccc}
  P_{\nu+\alpha_i,\nu} & \to  &  P_{\nu+\alpha_i,\nu}
  \la 2s\ra , \\ \\
  X_{k_i+1}^r& \mapsto &  \quad X_{k_i+1}^{r+s}, \quad s \geq 0,
\end{array}
\right. \\ \\ \\
 \Theta_N\left(
 \xy
  (0,0)*{\lined{i}};
  (0,4)*{\txt\large{$\bullet$}};
   (2,5)*{s};
 (-8,0)*{ \overbar{\nu -\alpha_i} };
 (4,0)*{ \nub};
 (-10,0)*{};(10,0)*{};
 \endxy\right)
& \maps&
 \left\{
\begin{array}{ccc}
   P_{\nu-\alpha_i,\nu} & \to  &  P_{\nu-\alpha_i,\nu} \la 2s \ra, \\ \\
  X_{k_i}^{r} & \mapsto &  \quad X_{k_i}^{r+s},  \quad s \geq 0,
\end{array}
\right.
\end{eqnarray*} \\ \\

\begin{equation*} \label{eq_gamma_dcross}
    \Theta_N \left(
 \xy
  (0,0)*{\xybox{
    (-4,-4)*{};(4,4)*{} **\crv{(-4,-1) & (4,1)}?(1)*\dir{>} ;
    (4,-4)*{};(-4,4)*{} **\crv{(4,-1) & (-4,1)}?(1)*\dir{>};
    (-5,-3)*{\scs i};
     (5.1,-3)*{\scs j};
     (8,1)*{ \nub};
     (-7,0)*{};(9,0)*{};
     }};
  \endxy
  \right)
  \quad \maps \quad
P_{\nu+\alpha_i+\alpha_j, \nu+\alpha_j} \otimes_{P_{\nu+\alpha_j}} P_{\nu+\alpha_j, \nu}
 \to  
 P_{\nu+\alpha_i+\alpha_j, \nu+\alpha_i} \otimes_{P_{\nu+\alpha_i}} P_{\nu+\alpha_i, \nu} \la -a_{ij}\ra, 
     \hspace{0.5in}
\end{equation*} 
\begin{eqnarray*}
    X_{k_i+1}^{r_1} \otimes X_{k_j+1}^{r_2}& \mapsto &
    \left\{
    \begin{array}{ccl} 
    \xsum{f=0}{r_1-1}X_{k_i+2}^{r_1+r_2-1-f} \otimes X_{k_i+1}^{f} - \xsum{g=0}{r_2-1}X_{k_i+2}^{r_1+r_2-1-g} \otimes X_{k_i+1}^{g}&  &
      \text{if  $a_{ij}=2$,} \\
       \left(t_{ij} X_{k_j+1}^{r_2} \otimes X_{k_i+1}^{r_1+1}+t_{ji} X_{k_j+1}^{r_2+1} \otimes
   X_{k_i+1}^{r_1} \right) \la -1\ra  &  &  \text{if  $i=j+1$,} \\
    \left( X_{k_j+1}^{r_2} \otimes X_{k_i+1}^{r_1}\right) \la 1\ra  &  &
  \text{if  $i=j-1$,}\\
  X_{k_j+1}^{r_2} \otimes X_{k_i+1}^{r_1} &  & \text{if  $a_{ij}=0$,}
      \end{array} \right.     
\end{eqnarray*} \\ \\

\begin{equation*} \label{eq_gamma_dcross_down}
    \Theta_N\left(
 \xy
  (0,0)*{\xybox{
    (-4,-4)*{};(4,4)*{} **\crv{(-4,-1) & (4,1)}?(0)*\dir{<} ;
    (4,-4)*{};(-4,4)*{} **\crv{(4,-1) & (-4,1)}?(0)*\dir{<};
    (-6,-3)*{\scs i};
     (6,-3)*{\scs j};
     (8,1)*{ \nub};
     (-7,0)*{};(9,0)*{};
     }};
  \endxy
  \right)
  \quad \maps \quad
P_{\nu-\alpha_i-\alpha_j, \nu-\alpha_j} \otimes_{P_{\nu-\alpha_j}} P_{\nu-\alpha_j, \nu}
 \to 
P_{\nu-\alpha_i-\alpha_j, \nu-\alpha_i} \otimes_{P_{\nu-\alpha_i}} P_{\nu-\alpha_i, \nu} \la -a_{ij}\ra,
     \hspace{0.5in}
\end{equation*}
\begin{eqnarray*}
   X_{k_i}^{r_1} \otimes X_{k_j}^{r_2}& \mapsto &
    \left\{
    \begin{array}{ccl}
    \xsum{f=0}{r_2-1}X_{k_i-1}^{r_1+r_2-1-f} \otimes X_{k_i}^{f} -
      \xsum{g=0}{r_1-1}X_{k_i-1}^{r_1+r_2-1-g} \otimes X_{k_i}^{g}&  &
      \text{if  $a_{ij}=2$,}\\
      t_{ji} \left( X_{k_j}^{r_2} \otimes X_{k_i}^{r_1}\right) \la -1\ra
       &  &  \text{if  $i=j+1$,} \\
  t_{ji}\left( X_{k_j}^{r_2+1} \otimes X_{k_i}^{r_1}+X_{k_j}^{r_2} \otimes
    X_{k_i}^{r_1+1} \right) \la 1\ra    &  &
  \text{if  $i=j-1$,}
      \\
    X_{k_j}^{r_2} \otimes X_{k_i}^{r_1} &  & \text{if  $a_{ij}=0$.} 
    \end{array} \right.
\end{eqnarray*}
\end{itemize}
\end{Theorem}

\begin{proof}
Theorem is proved in \cite{lauda2}. We, however, show a way to do it using symmetric polynomials instead of Chern classes. Clearly, the 2-functor \(\Theta_N\) preserves the degrees of 2-morphisms. We need to check that \(\Phi\) preserves the relations between 2-morphisms. 
\begin{itemize}
\item Biadjointness property of 1-morphisms can be shown the following way:
\[ \Theta_N\left(\;  \xy   0;/r.18pc/:
    (-8,0)*{}="1";
    (0,0)*{}="2";
    (8,0)*{}="3";
    (-8,-10);"1" **\dir{-}; (-1,4)*{\txt\large{$\bullet$}}; (2,4)*{s}; (-6,-5)*{i}; 
    "1";"2" **\crv{(-8,8) & (0,8)} ?(0)*\dir{>} ?(1)*\dir{>};
    "2";"3" **\crv{(0,-8) & (8,-8)}?(1)*\dir{>};
    "3"; (8,10) **\dir{-};
    (14,9)*{\nub};
    (-6,9)*{\overbar{\nu+\alpha_i}};
    \endxy \;\right)
    \; ( 1) =\]\[= \Theta_N\left(\;\xy
    (0,0)*{\bbcfe{}};
    (9,2)*{ \bfit{$\overbar{\nu+\alpha_i}$}};  (-1,1)*{\txt\large{$\bullet$}}; (1,3)*{s}; (-6,-2)*{i}; 
    (-4,-5.5)*{\scs };
    (4,-5.5)*{\scs };
    \endxy \right)
 \left(
 c^{-1}_{\nub,i} \sum_{l=0}^{\nu_i} (-1)^{\nu_i-l} X_{k_i+1}^{s}  \otimes X_{k_i+1}^{l} \otimes  e_{\nu_i-l}(\nu ; i)
 \right)=
   \]
   \[ = X_{k_i+1}^{s}  \sum_{l=0}^{\nu_i} (-1)^{\nu_i-l}  h_{l-\nu_i} (\nu+\alpha_i ; i)   e_{\nu_i-l}(\nu ; i)    = X_{k_i+1}^{s}=  \Theta_N\left(
 \xy
  (0,0)*{\lineu{i}};
  (0,4)*{\txt\large{$\bullet$}};
  (2,5)*{s};
 (4,0)*{ \nub};
 (-8,0)*{ \overbar{\nu +\alpha_i}};
 (-10,0)*{};(10,0)*{};
 \endxy\right) \left( 1 \right)  .\]
 
\item  The following relations are easy computations obtained by applying the images of crossings twice:
 \[
 \Theta_N\left(\; 
 \vcenter{\xy 0;/r.17pc/:
    (-4,-4)*{};(4,4)*{} **\crv{(-4,-1) & (4,1)}?(1)*\dir{};
    (4,-4)*{};(-4,4)*{} **\crv{(4,-1) & (-4,1)}?(1)*\dir{};
    (-4,4)*{};(4,12)*{} **\crv{(-4,7) & (4,9)}?(1)*\dir{};
    (4,4)*{};(-4,12)*{} **\crv{(4,7) & (-4,9)}?(1)*\dir{};
    (4,12); (4,13) **\dir{-}?(1)*\dir{>};
    (-4,12); (-4,13) **\dir{-}?(1)*\dir{>};
  (-5.5,-3)*{\scs i};
     (5.5,-3)*{\scs j}; (8, 3)*{\nub}
 \endxy} \right)  (X_{k_{i}+1}^{r_1} \otimes X_{k_{j}+1}^{r_2}) 
 \,=\]
\[ \] 
 \[ = \,
 \left\{
 \begin{array}{ccc}
 0 &  &  \text{if $a_{ij}=2$,}\\ \\
t_{ij} (X_{k_{i}+1}^{r_1} \otimes X_{k_{j}+1}^{r_2})   &  &  \text{if $a_{ij}=0$,}\\ \\
 t_{ij} (X_{k_{i}+1}^{r_1+1} \otimes X_{k_{j}+1}^{r_2})  \;\; + \;\; t_{ji}
  (X_{k_{i}+1}^{r_1} \otimes X_{k_{j}+1}^{r_2+1}) 
   &  & \text{if $a_{ij} =-1$,}
 \end{array}
 \right.  
\]
 \item  
 We only show the following relation for \(i=j+1\) :
 \begin{equation}
 \Theta_N \left(
\vcenter{\xy 0;/r.17pc/:
    (-4,-4)*{};(4,4)*{} **\crv{(-4,-1) & (4,1)}?(1)*\dir{};
    (4,-4)*{};(-4,4)*{} **\crv{(4,-1) & (-4,1)}?(1)*\dir{};
    (4,4)*{};(12,12)*{} **\crv{(4,7) & (12,9)}?(1)*\dir{};
    (12,4)*{};(4,12)*{} **\crv{(12,7) & (4,9)}?(1)*\dir{};
    (-4,12)*{};(4,20)*{} **\crv{(-4,15) & (4,17)}?(1)*\dir{};
    (4,12)*{};(-4,20)*{} **\crv{(4,15) & (-4,17)}?(1)*\dir{};
    (-4,4)*{}; (-4,12) **\dir{-};
    (12,-4)*{}; (12,4) **\dir{-};
    (12,12)*{}; (12,20) **\dir{-};
    (4,20); (4,21) **\dir{-}?(1)*\dir{>};
    (-4,20); (-4,21) **\dir{-}?(1)*\dir{>};
    (12,20); (12,21) **\dir{-}?(1)*\dir{>};
    (-6,-3)*{\scs i};
  (6,-3)*{\scs j};
  (15,-3)*{\scs i};
\endxy} \right)
\quad - \quad
 \Theta_N \left(
\vcenter{\xy 0;/r.17pc/:
    (4,-4)*{};(-4,4)*{} **\crv{(4,-1) & (-4,1)}?(1)*\dir{};
    (-4,-4)*{};(4,4)*{} **\crv{(-4,-1) & (4,1)}?(1)*\dir{};
    (-4,4)*{};(-12,12)*{} **\crv{(-4,7) & (-12,9)}?(1)*\dir{};
    (-12,4)*{};(-4,12)*{} **\crv{(-12,7) & (-4,9)}?(1)*\dir{};
    (4,12)*{};(-4,20)*{} **\crv{(4,15) & (-4,17)}?(1)*\dir{};
    (-4,12)*{};(4,20)*{} **\crv{(-4,15) & (4,17)}?(1)*\dir{};
    (4,4)*{}; (4,12) **\dir{-};
    (-12,-4)*{}; (-12,4) **\dir{-};
    (-12,12)*{}; (-12,20) **\dir{-};
    (4,20); (4,21) **\dir{-}?(1)*\dir{>};
    (-4,20); (-4,21) **\dir{-}?(1)*\dir{>};
    (-12,20); (-12,21) **\dir{-}?(1)*\dir{>};
  (-14,-3)*{\scs i};
  (-6,-3)*{\scs j};
  (6,-3)*{\scs i};
\endxy} \right)
 \;\; =\;\;
 t_{ij}  \Theta_N \left(\;\;
\xy 0;/r.17pc/:
  (4,12);(4,-12) **\dir{-}?(0)*\dir{<};
  (-4,12);(-4,-12) **\dir{-}?(0)*\dir{<}?(.25)*\dir{};
  (12,12);(12,-12) **\dir{-}?(0)*\dir{<}?(.25)*\dir{};
  (-6,-9)*{\scs i};     (6.1,-9)*{\scs j};
  (14,-9)*{\scs i};
 \endxy \right)\, .
 \end{equation}

 We prove this relation by applying the 2-morphisms on the both sides of the equalities on \( X^{r_3}_{k_i+2} \otimes X^{r_2}_{k_j+1} \otimes X^{r_1}_{k_i+1}\).  Since \[ X^{r_3}_{k_i+2} \otimes X^{r_2}_{k_j+1} \otimes X^{r_1}_{k_i+1} =\Theta_N \left(\;\;
\xy 0;/r.17pc/:
  (4,12);(4,-12) **\dir{-}?(0)*\dir{<};
  (-4,12);(-4,-12) **\dir{-}?(0)*\dir{<}?(.25)*\dir{};
  (12,12);(12,-12) **\dir{-}?(0)*\dir{<}?(.25)*\dir{};
  (-6,-9)*{\scs i};     (6.1,-9)*{\scs j};  (-1.5,0)*{\bullet r_3};  (6.5,0)*{\bullet r_2};  (14.5,0)*{\bullet r_1};
  (14,-9)*{\scs i};
 \endxy \right) ( 1\otimes 1 \otimes 1)\] 
 holds, without loss of generality we can take \(r_1=r_2=r_3=0\):
 \[
  \Theta_N \left(
\vcenter{\xy 0;/r.17pc/:
    (-4,-4)*{};(4,4)*{} **\crv{(-4,-1) & (4,1)}?(1)*\dir{};
    (4,-4)*{};(-4,4)*{} **\crv{(4,-1) & (-4,1)}?(1)*\dir{};
    (4,4)*{};(12,12)*{} **\crv{(4,7) & (12,9)}?(1)*\dir{};
    (12,4)*{};(4,12)*{} **\crv{(12,7) & (4,9)}?(1)*\dir{};
    (-4,12)*{};(4,20)*{} **\crv{(-4,15) & (4,17)}?(1)*\dir{};
    (4,12)*{};(-4,20)*{} **\crv{(4,15) & (-4,17)}?(1)*\dir{};
    (-4,4)*{}; (-4,12) **\dir{-};
    (12,-4)*{}; (12,4) **\dir{-};
    (12,12)*{}; (12,20) **\dir{-};
    (4,20); (4,21) **\dir{-}?(1)*\dir{>};
    (-4,20); (-4,21) **\dir{-}?(1)*\dir{>};
    (12,20); (12,21) **\dir{-}?(1)*\dir{>};
    (-6,-3)*{\scs i};
  (6,-3)*{\scs j};
  (15,-3)*{\scs i}; (17,7)*{\nub}
\endxy} 
\right) (1 \otimes 1 \otimes 1) =
\Theta_N \left( \vcenter{\xy 0;/r.17pc/:
    (4,4)*{};(12,12)*{} **\crv{(4,7) & (12,9)}?(1)*\dir{};
    (12,4)*{};(4,12)*{} **\crv{(12,7) & (4,9)}?(1)*\dir{};
    (-4,12)*{};(4,20)*{} **\crv{(-4,15) & (4,17)}?(1)*\dir{};
    (4,12)*{};(-4,20)*{} **\crv{(4,15) & (-4,17)}?(1)*\dir{};
    (-4,4)*{}; (-4,12) **\dir{-};
    (12,12)*{}; (12,20) **\dir{-};
    (4,20); (4,21) **\dir{-}?(1)*\dir{>};
    (-4,20); (-4,21) **\dir{-}?(1)*\dir{>};
    (12,20); (12,21) **\dir{-}?(1)*\dir{>};
    (-6, 4)*{\scs j};
  (5, 4)*{\scs i};
  (14, 4)*{\scs i}; (17,12)*{\nub}
\endxy} 
\right) 
(t_{ij} \, 1 \otimes X_{k_i+2} \otimes 1 +t_{ji} \, X_{k_j+1} \otimes 1 \otimes 1) \la 1 \ra =\]
\[=\Theta_N \left( \vcenter{\xy 0;/r.17pc/:
    (-4,12)*{};(4,20)*{} **\crv{(-4,15) & (4,17)}?(1)*\dir{};
    (4,12)*{};(-4,20)*{} **\crv{(4,15) & (-4,17)}?(1)*\dir{};
    (12,12)*{}; (12,20) **\dir{-};
    (4,20); (4,21) **\dir{-}?(1)*\dir{>};
    (-4,20); (-4,21) **\dir{-}?(1)*\dir{>};
    (12,20); (12,21) **\dir{-}?(1)*\dir{>};
    (-6, 12)*{\scs j};
  (5, 12)*{\scs i};
  (14, 12)*{\scs i}; (17,16)*{\nub}
\endxy} 
\right) 
(t_{ij} \, 1 \otimes 1 \otimes 1 ) \la 1 \ra
= t_{ij} \, 1\otimes 1 \otimes 1, \]
 \[
  \Theta_N \left(
\vcenter{\xy 0;/r.17pc/:
    (4,-4)*{};(-4,4)*{} **\crv{(4,-1) & (-4,1)}?(1)*\dir{};
    (-4,-4)*{};(4,4)*{} **\crv{(-4,-1) & (4,1)}?(1)*\dir{};
    (-4,4)*{};(-12,12)*{} **\crv{(-4,7) & (-12,9)}?(1)*\dir{};
    (-12,4)*{};(-4,12)*{} **\crv{(-12,7) & (-4,9)}?(1)*\dir{};
    (4,12)*{};(-4,20)*{} **\crv{(4,15) & (-4,17)}?(1)*\dir{};
    (-4,12)*{};(4,20)*{} **\crv{(-4,15) & (4,17)}?(1)*\dir{};
    (4,4)*{}; (4,12) **\dir{-};
    (-12,-4)*{}; (-12,4) **\dir{-};
    (-12,12)*{}; (-12,20) **\dir{-};
    (4,20); (4,21) **\dir{-}?(1)*\dir{>};
    (-4,20); (-4,21) **\dir{-}?(1)*\dir{>};
    (-12,20); (-12,21) **\dir{-}?(1)*\dir{>};
  (-14,-3)*{\scs i};
  (-6,-3)*{\scs j};
  (6,-3)*{\scs i};
\endxy}
\right) (1 \otimes 1 \otimes 1)  
= 0.
 \]

 \end{itemize}
We leave the remaining relations for the readers to check. They all follow from the identities (\ref{bs1}) to (\ref{s8}) via direct computations. 
\end{proof}

Under the 2-functor \(\Theta_N\), the images of the clockwise and counterclockwise bubbles are:
\begin{eqnarray*}  
     \Theta_N \left(\;\xy
  (4,6)*{\nub};
  (0,-2)*{\cbub{\nub_i-1+r}{i}};
 \endxy \; \right) \, ,  \Theta_N \left(\;\xy
  (4,6)*{\nub};
  (0,-2)*{\ccbub{-\nub_i-1+r}{i}};
 \endxy \; \right) : P_\nu \to P_\nu,
\end{eqnarray*}

\begin{eqnarray} \label{eq_gamm_bubbles}   
     \Theta_N \left(\;\xy
  (4,6)*{\nub};
  (0,-2)*{\cbub{\nub_i-1+r}{i}};
 \endxy \; \right) & ( 1) = c_{i, \nub}
\sum_{l=0}^{r} (-1)^l e_l(\nu; i+1) h_{r-l} (\nu; i), 
\end{eqnarray}

\begin{eqnarray} \label{eq_gamma_bubbles}   
     \Theta_N \left(\;\xy
  (4,6)*{\nub};
  (0,-2)*{\ccbub{-\nub_i-1+r}{i}};
 \endxy \; \right) & ( 1) = c^{-1}_{i,\nub}
 \sum_{l=0}^{r} (-1)^l e_l(\nu; i) h_{r-l} (\nu; i+1).
\end{eqnarray}

If we replace the ring \(P_\nu\) with ring \(C_\nu :=  P_\nu /Sym_N\) for each \(\nu \in \Lambda(n,N)\), we can still define a 2-category \(\textbf{CPol}_N\) with objects  \(C_\nu \) and a 2-functor \(\Psi \colon \U \to \textbf{CPol}_N\) in an entirely similar way. Khovanov-Lauda \cite{lauda2} construct a 2-category \(\textbf{Flag}^*_N\) using the ordinary cohomology of partial flag varieties to show the non-degeneracy of \(\U\). The isomorphism between \(\textbf{EqFlag}^*_N\) and \(\textbf{Pol}_N\)  descends to a isomorphism between \(\textbf{Flag}^*_N\) and \(\textbf{CPol}_N\). Taking the quotients of the objects in \(\textbf{Pol}_N\) by \(Sym_N\) corresponds to imposing Grassmannian relation on the objects of \(\textbf{EqFlag}^*_N\). The 2-morphisms in \(\textbf{CPol}_N\) are finite-dimensional, and we have \(C_{\lb_0} :=  P_{\lb_0} /Sym_N \simeq \mathbb{C}\). The following result is also due to Khovanov-Lauda \cite{lauda2}, Theorem 6.14.
 \begin{Theorem}
 The 2-categories \(\textbf{Pol}_N\) and \(\textbf{CPol}_N\) categorify the finite-dimensional, simple \(\mathfrak{sl}_n\) representation \(V(\lb_0)\).
 \end{Theorem}
It is easy to check that both \(\textbf{Pol}_N\) and \(\textbf{CPol}_N\) are integrable. Recall the two categorifications of  \(V(\lb_0)\) -- the categories of graded, finitely generated, projective modules \(\check{R}^{\lb} \text{-} \mathrm{ pMod}\) and \(R^{\lb} \text{-} \mathrm{ pMod} \) over the deformed cyclotomic KLR algebra \(\check{R}^{\lb}\) and the cyclotomic KLR algebra \(R^{\lb} \) respectively.  By Theorem \ref{univv}, they carry a universality property which can depicted in the following commuting diagrams:
 \[ \begin{Large}
 \begin{tikzcd}
  \scs \mathcal{U} \arrow[rd] \arrow[r, "\Theta_N"] & \scs \textbf{Pol}_N \\
  &\scs \check{R}^{\lb_0} \text{-} \mathrm{ pMod} \arrow[u,  "\Theta'_N"] 
  \end{tikzcd},
  \end{Large} \quad \quad  \quad \begin{Large}
 \begin{tikzcd}
  \scs \mathcal{U} \arrow[rd] \arrow[r, "\Psi_N"] & \scs \textbf{CPol}_N \\
  &\scs R^{\lb_0} \text{-} \mathrm{ pMod} \arrow[u,  "\Psi'_N"] 
  \end{tikzcd},
  \end{Large}
  \]
where \(\Theta'_N\) and \(\Psi'_N\) are strongly equivariant 2-functors. The 2-category \(\textbf{CPol}_N\) is a minimal categorification since \(C_{\lb_0} \simeq \mathbb{C}\), i.e., \(\Psi'_N\)  is a strongly equivariant equivalence.

 \subsection{\(Z(\textbf{Pol}_N)\) as a current algebra module}

The 2-functor \(\Theta_N \colon \U \to \textbf{Pol}_N\) induces a \(\mathbb{C}\)-linear functor \(\Tr \Theta_N  \colon \Tr \U \to \Tr \textbf{Pol}_N \). The center of the object \(P_\nu\) of the 2-category \(\textbf{Pol}_N\) is isomorphic to \(P_\nu\), since \(P_\nu\) is a commutative ring and acts on itself by multiplication. Composing the linear isomorphism \(\rho \maps \dot{\bfU}(\mf{sl}_n[t]) \longrightarrow \Tr \U \)  with the functor  \(\Tr \Theta_N \colon \Tr \U \to \Tr \textbf{Pol}_N\) gives the current algebra module structure on \(\bigoplus_{\nu \in \mathcal{P}(n,N)} P_\nu\), where each ring \(P_\nu\) corresponds to the \(\mathfrak{sl}_n\) weight space \(\nub\). Recall the definition of the generators \(\mathsf{F}_{i,j}\),  \(\mathsf{E}_{i,j}\),  \(\mathsf{H}_{i,j}\)   of  \(\Tr \U \) in the equation \eqref{eq_sln-homom}.
  In what follows, we will define the action of \(\Tr \textbf{Pol}_N \) on the the center of objects  \(\bigoplus_{\nu \in \mathcal{P}(n,N)} P_\nu\) to be the images of the trace maps, defined in \eqref{mpp1} to \eqref{mpp3}, under the functor  \(\Theta_N\) :
 \[ \mathtt{F_{i,j}} = \Tr \Theta_N  (\mathsf{F}_{i,j}) \colon P_\nu \to  P_{\nu-\alpha_i} ,\]
  \[\mathtt{E_{i,j}} =\Tr \Theta_N  (\mathsf{E}_{i,j})  \colon P_{\nu} \to  P_{\nu+\alpha_i},\]
   \[\mathtt{H_{i,j}} =\Tr \Theta_N  (\mathsf{H}_{i,j})  \colon P_\nu \to  P_\nu , \]
for all \(i \in I\) and \(j \geq 0\). 
The next theorem defines the current algebra action and generalizes Brundan's formula (\cite{brun1}, Theorem 3.4)  for the action of \(\mathfrak{sl}_n\).
\begin{Theorem} \label{thmf}
\begin{enumerate}
\item The map \(\mathtt{F_{i,j}} \colon P_\nu \to  P_{\nu-\alpha_i} \)  is the \(P_{\nu-\alpha_i}\)-module homomorphism such that 
\[\mathtt{F_{i,j}} (X^{m}_{k_i})= c^{-1}_{i, \nub} \, \sum_{l=0}^{\nu_i-1}(-1)^{l} e_l(\nu-\alpha_i;i) h_{m+j+\nub_i-1-l }(\nu-\alpha_i; i+1).\]
\item The map \(\mathtt{E_{i,j}} \colon P_\nu \to  P_{\nu+\alpha_i} \)  is the \(P_{\nu+\alpha_i}\)-module homomorphism such that 
\[\mathtt{E_{i,j}} (X^{m}_{k_i+1})= c^{}_{i, \nub} \sum_{l=0}^{\nu_{i+1}-1}(-1)^{l} e_l(\nu+\alpha_i;i+1)   h_{m+j- \nub_i-1-l}(\nu+\alpha_i; i) .\]
\item  The map \(\mathtt{H_{i,j}} \colon P_\nu \to P_\nu\) is the multiplication by
\( (-1)^j ( p_j(\nu; i+1) - p_j(\nu; i)) \)  if \(j>0\), and by \(\nub_i\) if  \(j=0\). 
\end{enumerate}
\end{Theorem}

\begin{proof}
\begin{enumerate}
\item We compute the action of  \(\mathtt{F_{i,j}} \)  on \(p\in P_\nu\). To do this, we first apply the map  \(    \Theta_N \left(\;\xy
    (0,-3)*{\bbpef{}};
    (8,-6)*{ \bfit{ $\overbar{\nu-\alpha_i}$ }}; (0,0)*{ \bfit{ $\nub$ }};
    (-4,2)*{\scs  };
    (4,2)*{\scs  };
    \endxy \; \right)\) on \(p\) : 
    
    \[ \Theta_N \left(\;\xy
    (0,-3)*{\bbpef{}};
 (8,-6)*{ \bfit{ $\overbar{\nu-\alpha_i}$ }}; (0,0)*{ \bfit{ $\nub$ }};
    (-4,2)*{\scs  };
    (4,2)*{\scs  };
    \endxy \; \right) (p)=  c^{-1}_{i, \overbar{\nu-\alpha_i}} \, \xsum{r=0}{\nu_i-1}(-1)^{l} e_l(\nu-\alpha_i;i)  p\, \otimes  X_{k_i}^{\nu_i-1-l} \, \in P_{\nu-{\alpha_i}, \nu} \otimes_{P_{\nu}} P_{\nu, \nu-{\alpha_i} }.\]

 Every element \(p \in P_{\nu', \nu}\) can be written as \(p=\sum_{m=0}^{\nu_{i+1}} p'_m X_{k_i}^m\) for some elements \(p'_m \in P_{\nu-\alpha_i}\). 
 
Now we apply the cap \( \Theta_N \left(\;\xy
    (0,0)*{\bbcef{}};
    (8,3)*{ \bfit{$\overbar{\nu-\alpha_i}$}};
    (-4,-5.5)*{\scs}; (0,1.5)*{\txt{$\bullet$}};   (0,5)*{ \bfit{$j$}};
    (4,-5.5)*{\scs};
    \endxy\;\right)\):  
    
    \[ \Theta_N \left(\;\xy
    (0,0)*{\bbcef{}};
    (8,3)*{ \bfit{$\overbar{\nu-\alpha_i}$}};
    (-4,-5.5)*{\scs}; (0,1.5)*{\txt{$\bullet$}};   (0,5)*{ \bfit{$j$}};
    (4,-5.5)*{\scs};
    \endxy\;\right) \left(c^{-1}_{i, \overbar{\nu-\alpha_i}} \sum_{m=0}^{\nu_{i+1}} p'_m \sum_{l=0}^{\nu_i-1}(-1)^{l} e_l(\nu-\alpha_i;i)  X_{k_i}^m\ \otimes  X_{k_i}^{j+\nu_i-1-l} \right) =  \]
    \[ = c^{-1}_{i, \nub} \sum_{m=0}^{\nu_{i+1}} p'_m \sum_{l=0}^{\nu_i-1}(-1)^{l} e_l(\nu-\alpha_i;i)   h_{m+j+\nub_i-1-l }(\nu-\alpha_i; i+1).  \]
Thus, 
\[ \mathtt{F_{i,j}} (X^{m}_{k_i})= c^{-1}_{i, \nub}  \sum_{l=0}^{\nu_i-1}(-1)^{l} e_l(\nu-\alpha_i;i)   h_{m+j+\nub_i-1-l }(\nu-\alpha_i; i+1). \]
\item Similarly,  the action of the \(\mathtt{E_{i,j}} \)  on \(p\in P_\nu\) is defined by closing it with a cup and a cup consecutively, and hence obtaining an element of  \(P_{\nu+\alpha_i}\). 
As an element of  \(P_{\nu, \nu+\alpha_i}\), we can write \(p=\sum_{m=0}^{\nu_{i+1}-1} p'_m X_{k_i+1}^m\) for some elements \(p'_m \in P_{\nu+\alpha_i}\). Then we have 
 \[ \Theta_N\left(\;\xy
    (0,-3)*{\bbpfe{}};
    (8,-6)*{ \bfit{$\overbar{\nu+\alpha_i}$}};(-6,-2)*{i};  (0,0)*{ \bfit{$\overbar{\nu}$}};
    (-4,2)*{\scs };
    (4,2)*{\scs };
    \endxy\;\right) (p) =  \xsum{r=0}{\nu_{i+1}-1}(-1)^{l} e_l(\nu+\alpha_i;i+1)  p\otimes  X_{k_i}^{\nu_{i+1}-1-l}=\]
   
    \[= \sum_{m=0}^{\nu_{i+1}-1} p'_m \sum_{l=0}^{\nu_{i+1}-1}(-1)^{l} e_l(\nu+\alpha_i;i+1)  X_{k_i+1}^m\ \otimes  X_{k_i+1}^{\nu_{i+1}-1-l}  \quad \in P_{\nu+{\alpha_i}, \nu} \otimes_{P_{\nu}} P_{\nu, \nu+{\alpha_i} }.\]
    
    We close the diagram by applying the cap:
    \[  \Theta_N\left(\;\xy
    (0,0)*{\bbcfe{}};
    (8,3)*{ \bfit{$\overbar{\nu+\alpha_i}$}}; (-6,-2)*{i};   (0,5)*{ \bfit{$j$}}; (0,1.5)*{\txt{$\bullet$}};  
    (-4,-5.5)*{\scs };
    (4,-5.5)*{\scs };
    \endxy \right) \left( \sum_{m=0}^{\nu_{i+1}-1} p'_m \sum_{l=0}^{\nu_{i+1}-1}(-1)^{l} e_l(\nu+\alpha_i;i+1)  X_{k_i+1}^m\ \otimes  X_{k_i+1}^{j+\nu_{i+1}-1-l}   \right)=\]
  
    \[ =c^{}_{i, \overbar{\nu+\alpha_i}}\sum_{m=0}^{\nu_{i+1}-1} p'_m \sum_{l=0}^{\nu_{i+1}-1}(-1)^{l} e_l(\nu+\alpha_i;i+1)   h_{m+j-\nub_{i}-1-l}(\nu+\alpha_i; i) .\]
    Thus, we get 
\[ \mathtt{E_{i,j}}  (X^{m}_{k_i})= c^{}_{i, \nub} \sum_{l=0}^{\nu_{i+1}-1}(-1)^{l} e_l(\nu+\alpha_i;i+1)   h_{m+j- \nub_i-1-l}(\nu+\alpha_i; i) .\]

  \item 
 The map \(\mathsf{H}_{i,j}\), \(j \geq 1\) is a multiplication by \(\Tr \Theta_N  (\pi_{i,j} (\nub))\), where 
 \begin{equation} \label{relat}
 \pi_{i,j} (\nub)=\sum_{l=0}^{j} (l+1)
 \xy
   (0,-2)*{\icbub{\spadesuit+l}{i}};
   (12,-2)*{\iccbub{\spadesuit+j-l}{i}};
  (8,8)*{\nub};
 \endxy .
   \end{equation}
   
 We define the following generating functions:
$$
 \xy 0;/r.15pc/:
 (0,0)*{\ccbub{}{i}};
  (4,8)*{\nub}; 
 \endxy (t) =  \sum_{l=0}^{\infty}  \xy 0;/r.15pc/:
 (0,0)*{\ccbub{-\nub_i-1+l}{i}};
  (4,8)*{\nub}; 
 \endxy \, t^l,  
 \quad \quad  \xy 0;/r.15pc/:
 (0,0)*{\cbub{}{i}};
  (4,8)*{\nub}; 
 \endxy (t) =  \sum_{l=0}^{\infty}  \xy 0;/r.15pc/:
 (0,0)*{\cbub{\nub_i-1+l}{i}};
  (4,8)*{\nub}; 
 \endxy \, t^l,     \quad \quad E(\nu; i)(t) = \sum_{l=0}^{\nu_i} e_l(\nu;i) t^l,  $$
  $$H(\nu; i)(t) = \sum_{l=0}^{\infty} h_l(\nu;i) t^l,  \quad \quad 
 p(\nu; i)(t) = \sum_{l=0}^{\infty} p_l(\nu;i) t^{l-1} ,  \quad \quad \pi_i(\nub) (t) = \sum_{l=0}^{\infty} \pi_{i,l}(\nub) t^{l-1} .  $$
 Then the equation  \ref{relat} can be written as 
 \begin{equation}
 \pi_i(\nub) (t) = \left(  \xy 0;/r.15pc/:
 (0,0)*{\cbub{}{i}}; (4,8)*{\nub}; 
 \endxy (t)   \right)'  \xy 0;/r.15pc/:
 (0,0)*{\ccbub{}{i}};
  (4,8)*{\nub}; 
 \endxy (t) = \frac{\left(  \xy 0;/r.15pc/:
 (0,0)*{\cbub{}{i}};
  (4,8)*{\nub}; 
 \endxy (t)   \right)'  }{ \xy 0;/r.15pc/:
 (0,0)*{\cbub{}{i}};
  (4,8)*{\nub}; 
 \endxy (t)    }= \left(  \log \xy 0;/r.15pc/:
 (0,0)*{\cbub{}{i}};
 (4,8)*{\nub}; 
 \endxy (t)   \right)' . 
 \end{equation}

The equation \eqref{eq_gamm_bubbles} implies 
\[ \Tr \Theta_N  \left( \xy 0;/r.15pc/:
 (0,0)*{\cbub{}{i}};
  (4,8)*{\nub}; 
 \endxy (t) \right) =H(\nu; i)(-t) E(\nu; i+1)(t)=\frac{E(\nu; i+1)(t)} {E(\nu; i)(t) } ,\]
 so we have 
  \begin{equation} \label{macc}
 \Tr \Theta_N  \left( \pi_i(\nub) (t) \right) =\left( log   \frac{E(\nu; i)(t)} {E(\nu; i+1)(t) }  \right)' =
 \end{equation} 
 \[= \left( log \, E(\nu; i+1)(t) \right)' - \left( log \,  E(\nu; i)(t)   \right)' =  p(\nu; i+1)(-t) - p(\nu; i)(-t) . \]
 \end{enumerate}
  The last equality in the equation \eqref{macc} is given on page 23 in \cite{mcd}. Thus, \( \Tr \Theta_N \left( \pi_{i,j}(\nub)  \right)   =(-1)^j ( p_j(\nu; i+1) - p_j(\nu; i))\).
\end{proof}

\subsection{The coinvariant algebra \(C^\l_\nu\) }
Fix a partition \(\l \in \mathcal{P}^+(n, N) \). Let \(\nu \in \mathcal{P}(n, N) \) and  \(I^{\l }_{\nu}\) be the ideal of \(P_\nu\) generated by 
\begin{equation} \label{idea}
\left\{h_r(\nu;i_1,\dots,i_m)\:\bigg|
\begin{array}{l}
1 \leq m \leq n, \:\: 1\leq i_1<\dots<i_m \leq n ,\\
r > \l _1+\cdots+\l_m - \nu_{i_1}-\cdots-\nu_{i_m}\end{array}\right\}.
\end{equation}
We will need the following statement for the special case \(\l_0  = (N, 0, \dots, 0) \in \mathcal{P}^+(n, N)\).

\begin{Proposition} \label{sym}
If \(\l_0 = (N, 0, \dots, 0)\),  then \(I^{\l_0 }_{\nu} = Sym_N\). 
\end{Proposition}
\begin{proof}
We have 
\[
I^{\l_0}_\nu = \left\langle h_r(\nu;i_1,\dots,i_m)\:\bigg|
\begin{array}{l}
1 \leq m \leq n, \:\: 1\leq i_1<\dots<i_m \leq n ,\\
r > N - (\nu_{i_1}+\cdots+\nu_{i_m})\end{array} \right\rangle.
\]
Since \(Sym_N = \langle h_r(\nu;1,2,\dots, n) \: | \: r>0 \rangle  \), we have \(Sym_N \subseteq I^{\l_0}_{\nu}\).  To prove \( I^{\l_0}_{\nu} \subseteq Sym_N\), let \( h_r(\nu;i_1,\dots,i_m)\) be a generator of \(I^N_\nu \) with \(r > N - (\nu_{i_1}+\cdots+\nu_{i_m})\). Define \[\{j_1, j_2, \dots, j_{n-m}\} =\{1, 2, \dots n\} \setminus \{i_1, i_2, \dots i_{m}\}. \]
to be the ordered set. Then according to \eqref{brad222}, we have
\[  h_r(\nu;i_1,\dots,i_m) = \sum_{s=0}^r (-1)^{r-s} e_{r-s} (\nu; j_1,  \dots, j_{n-m}) h_{s}(\nu;1,2,\dots, n) =\] \[=(-1)^r e_{r} (\nu; j_1,  \dots, j_{n-m}) +\sum_{s=1}^r (-1)^{r-s} e_{r-s} (\nu; j_1,  \dots, j_{n-m}) h_{s}(\nu;1,2,\dots, n) .\] 
since \(e_{r} (\nu; j_1,  \dots, j_{n-m}) \) is an elementary symmetric polynomial in \( N - (\nu_{i_1}+\cdots+\nu_{i_m})\) variables, it is zero when \(r > N - (\nu_{i_1}+\cdots+\nu_{i_m})\). Thus,  \( I^{\l_0}_{\nu} \subseteq Sym_N\) holds. \qedhere
\end{proof}

Let \(C^\l_\nu = P_\nu/I^{\l}_{\nu}\).  Notice that  \(Sym_N \subseteq I^{\l }_{\nu}\) holds, thus we can regard the ring \(C^\l_\nu\) as a quotient of \(C_\nu=C^{\l_0}_\nu \). Since \(P_\nu\) is a finitely generated free module over \(Sym_N \), \(C_\nu\) is finite dimensional and hence so is \(C^\l_\nu\). Notice that \(C^\l_\nu \neq 0\) if and only if \( \l \geq \nu \) with respect to the dominance order.

\begin{Proposition}\label{dlem}
The ideal  $\bigoplus_\nu I^\l_\nu$ of  $\bigoplus_\nu P_\nu$ is invariant under the action of generators $ E_{i,j}$ and $F_{i,j}$ for all $i \in I$ and \(j \geq 0\).
\end{Proposition}
\begin{proof} 
We will only show $ E_{i,j} \left(I^\l_\nu \right) \subseteq I^\l_{\nu+\alpha_i}$, and  the proof of $F_{i,j} (I^\l_\nu) \subseteq I^\l_{\nu-\alpha_i}$ will be similar.
We  have to show that  images of  all $h_r(\nu;i_1,\dots,i_m)$  under $ E_{i,j}$
are in  $I^\l_{\nu+\alpha_i}$ for all $j \geq 0$, $m \geq 1$, $1 \leq i_1<\dots < i_m \leq n$
and $r > \l_1+\cdots+\l_m - \nu_{i_1}-\cdots-\nu_{i_m}$.
We proceed with the case distinction. 

\begin{enumerate}
\item If   $i\in \{i_1,\dots,i_m\}$ and $i+1 \in \{i_1,\dots,i_m\}$, then $h_r(\nu;i_1,\dots,i_m)=h_r(\nu+\alpha_i;i_1,\dots,i_m)$ holds. In this case, we have   
$$E_{i,j} ( h_r(\nu;i_1,\dots,i_m))= h_r(\nu+\alpha_i;i_1,\dots,i_m)  E_{i,j} (1),
$$
which belongs to $I_{\nu+\alpha_i}^\l$. The same is true when $i\notin \{i_1,\dots,i_m\}$ and $i+1 \notin \{i_1,\dots,i_m\}$ hold simultaneously . 

\item We consider the case $i\in \{i_1,\dots,i_m\}$ and $i+1 \notin \{i_1,\dots,i_m\}$. 
 
We have the identity (\ref{s6}):
$$h_r(\nu;i_1,\dots,i_m)= h_r(\nu+\alpha_i;i_1,\dots,i_{m})
- X_{k_i+1} h_{r-1}(\nu+\alpha_i;i_1,\dots,i_{m}).$$

Since  $r-1 > \l_1+\cdots+\l_m - (\nu+\alpha_i)_{i_1}-\cdots-(\nu+\alpha_i)_{i_m} =  \l_1+\cdots+\l_m - \nu_{i_1}-\cdots-\nu_{i_m}-1$, the element 
$$h_r(\nu+\alpha_i;i_1,\dots,i_{m}) \,  E_{i,j} (1)
- h_{r-1}(\nu+\alpha_i;i_1,\dots,i_{m}) \, E_{i,j}( X_{k_i+1})  $$

belongs to $I_{\nu+\alpha_i}^\l$.

\item
Finally suppose that $i\notin \{i_1,\dots,i_m\}$ and $i+1 =i_m$. 
By (\ref{s7}), we have that
$$h_r(\nu;i_1,\dots,i+1) = \sum_{l=0}^r h_l(\nu+\alpha_i ;i_1,\dots, i+1) X_{k_i+1}^{r-l}$$.

Applying $ E_{i,j}$ on $h_r(\nu;i_1,\dots,i+1)$, we get the element
\begin{equation} \label{prr}
 c_{i, \nub}
\sum_{s=0}^{\nu_{i+1}-1}  (-1)^s e_{s}(\nu+\alpha_i; i+1) \sum_{l=0}^r   h_{r-l+j-\nub_i-1-s}(\nu+\alpha_i; i) h_l(\nu+\alpha_i ;i_1,\dots, i+1). 
\end{equation}

Notice that since $r > \l_1+\cdots+\l_m
- \nu_{i_1}-\cdots-\nu_{i_m} = 
\l_1+\cdots+\l_m-(\nu+\alpha_i)_{i_1}-\cdots-(\nu+\alpha_i)_{i_m}-1$, the second summand in the identity
$$h_{r+j-\nub_i-1-s}(\nu+\alpha_i; i_1,\dots, i+1, i) = \sum_{l=0}^r  h_{r-l+j-\nub_i-1-s}(\nu+\alpha_i; i) h_l(\nu+\alpha_i ;i_1,\dots, i+1)$$
\[ +\sum_{l=r+1}^{r+j-\nub_i-1-s} h_{r-l+j-\nub_i-1-s}(\nu+\alpha_i; i) h_l(\nu+\alpha_i ;i_1,\dots, i+1) \]
is in $I_{\nu+\alpha_i}^\l$. Therefore, 
\[ h_{r+j-\nub_i-1-s}(\nu+\alpha_i; i_1,\dots, i+1, i)   \equiv \] \[ \equiv \sum_{l=0}^r  h_{r-l+j-\nub_i-1-s}(\nu+\alpha_i; i) h_l(\nu+\alpha_i ;i_1,\dots, i+1) \bmod I_{\nu+\alpha_i}^\l.\]
We can insert this equivalence in the element \eqref{prr}. To finish the proof, we need to check that the element
$$ c_{i, \nub}
\sum_{s=0}^{\nu_{i+1}-1}  (-1)^s e_{s}(\nu+\alpha_i; i+1) h_{r+j-\nub_i-1-s}(\nu+\alpha_i; i_1,\dots, i+1, i)
$$\end{enumerate} 
also belongs to $I_{\nu+\alpha_i}^\l$.  Another polynomial identity shows that it is equal to  \(c_{i, \nub} h_{r+j-\nub_i-1}(\nu+\alpha_i; i_1,\dots, i_{m-1}, i)\), which is an element of $I_{\nu+\alpha_i}^\l$, since 
\( r+j-\nub_i-1> \l_1+\cdots+\l_m-(\nu+\alpha_i)_{i_1}-\cdots-(\nu+\alpha_i)_{i_m-1}-(\nu+\alpha_i)_{i+1} \).  \end{proof}
The proof of Proposition \ref{dlem} is a generalization of Lemma 4.1 of Brundan \cite{brun1} for the \(j\) parameter. This proposition implies the following statement.
\begin{Corollary}
The ring \(\bigoplus_\nu C^\l_\nu\) is a finite-dimensional graded current algebra module with the highest weight 
 \(\lb= (\l_1-\l_2, \dots, \l_{n-1}-\l_n)\).
\end{Corollary}
The current algebra action on \(\bigoplus_\nu P_\nu\) descends to \(\bigoplus_\nu C^\l_\nu\). In other words, the following diagrams commute:
$$
\begin{CD}
P_{\nu}  &@>E_{i,j}>>& P_{\nu+\alpha_i} \\
@VV p V&&@VVp'V\\
C^\l_{\nu} &@>>E_{i,j}>& C^\l_{\nu+\alpha_i}, 
\end{CD}
\quad \quad \quad \quad
\begin{CD}
P_{\nu}  &@>F_{i,j}>>& P_{\nu-\alpha_i} \\
@VV p V&&@VVp''V\\
C^\l_{\nu} &@>>F_{i,j}>& C^\l_{\nu-\alpha_i},
\end{CD}
$$
where \(p\), \(p'\) and \(p''\)  are projections with kernels \(I^\l_{\nu}\), \(I^\l_{\nu+\alpha_i}\) and \(I^\l_{\nu-\alpha_i}\) respectively.

Recall that we defined an integer
  \begin{equation}
d^\l_\nu  = \max \{(\l, \l) - (\nu, \nu), 0\}.
\end{equation}  
It is easy to show that if \(\l > \nu\) with respect to the dominance order, we have \(d^\l_\nu > 0\). 
Brundan-Ostrik \cite{brun2} prove that the ring \(C^\l_\nu\) is isomorphic to the cohomology ring of Spaltenstein variety -- the variety of partial flags of type \(\nu\) which are annihilated by matrices of Jordan type \(\l^T\).

For a graded, finite-dimensional current algebra module \(M\), {\em graded composition multiplicity} of a simple \(\mathfrak{sl}_n\)-module \(V\) is the polynomial 
\[ \sum_{s \geq 0} \left[ M \{ 2s \}  :V \right] \, t^s, \]
where \(\left[ M \{ 2s \}  :V \right] \) is the number of copies of \(V\) in \(M \{ 2s \}  \). 
Given a partition \( \tau\), the graded composition multiplicity of  \(V(\bar{\tau})\)  in \(\bigoplus_\nu C^\l_\nu\) is given by Brundan \cite{brun1} via the formula
\begin{equation} \label{chr} 
\sum_{r \geq 0} 
 \left[ {\textstyle \bigoplus_\nu} C^\l_\nu \{ d^{\l}_\nu - 2r \} : V(\bar{\tau}) \right] t^r =
K_{\tau^T,\l^T}(t),
\end{equation}
where \(\tau^T\) and \(\l^T\)  are the transposes of the partitions \(\tau\), \(\l\), and \(K_{\tau^T,\l^T}(t)\) is the {\em Kostka-Foulkes polynomial}. We refer to \cite{mcd}, section III.6 for the definition of the Kostka-Foulkes polynomial. We will use their following property.
\begin{Proposition} \label{prr}
Let \(\l, \mu \in \mathcal{P}^+(n, N) \) be partitions, and let \(\l' = \l +(m, m, \dots, m),  \mu' = \mu +(m, m, \dots, m) \)
for some non-negative integer \(m\). Then the following equations hold:
\begin{equation} \label{hum}
K_{\l, \mu}(t)=K_{\l', \mu'}(t),
\end{equation}
\[K_{\mu^T,\l^T }(t)=K_{(\mu')^T, (\l')^T}(t).\]
\end{Proposition}
\begin{proof}
\(K_{\l, \mu}(t)\) is nonzero if and only if \(\l \geq \mu\) with respect to the dominance order. Let 
\(\alpha_i =(0, \dots, 0, 1, -1, 0, \dots, 0)\) be the \(i\)-th simple root as before, and let \[R^+=\{\alpha_i+\alpha_{i+1}+\cdots+ \alpha_j\, | \, 1 \leq i \leq j \leq n-1\} \] be the set of positive roots. For any \( \xi = (\xi_1, \dots, \xi_n) \in \Z\) such that 
\(\sum_i \xi_i = 0\), we define the polynomial 
\[P(\xi; t)= \sum_{\{m_\alpha\}_{\alpha \in R^+}} t^{\sum m_\alpha}, \]
where the sum is over all families \(\{m_\alpha\}_{\alpha \in R^+} \) of non-negative integers such that \(\xi = \sum m_\alpha \alpha\). The polynomial \(P(\xi;t)\) is nonzero if and only if \(\xi = \sum_{i} \eta_i \alpha_i\) for some non-negative integers \(\eta_i\), \(1 \leq i \leq n-1\).

By example III.6.4 in \cite{mcd}, 
\begin{equation} \label{deff}
K_{\l, \mu}(t) = \sum_{\sigma \in S_n} \text{sgn}(\sigma) \, P(\sigma(\l +\delta) - (\mu+\delta) ; t) , 
\end{equation}
where \(\delta =(n-1, n-2, \dots , 1, 0)\). The first equality in \eqref{hum} immediately follows from \eqref{deff}.

To prove the second equality, notice that
 \((\l')^T = (\underbrace{ n, \dots, n}_{m \, \text{times}}, \l_1^T, \dots, \l_N^T)\), \((\mu')^T = (\underbrace{ n, \dots, n}_{m \, \text{times}}, \mu_1^T, \dots, \mu_N^T)\). We take \(S_N\) to be the subgroup of \(S_{m+N}\) which fixes the first \(m\) entries, then
 \begin{equation} \label{df}
K_{(\mu')^T, (\l')^T }(t) = \sum_{\sigma \in S_{m+N}} \text{sgn}(\sigma) \, P(\sigma((\mu')^T+\delta) - ((\l')^T+\delta) ; t).
\end{equation}
 Notice that summand is zero if \(\sigma \in S_{m+N} \setminus S_N\). Therefore, we can write \eqref{df} as \begin{equation} \label{df1}
K_{(\mu')^T, (\l')^T }(t) = \sum_{\sigma \in S_{N}} \text{sgn}(\sigma) \, P(\sigma((\mu')^T+\delta) - ((\l')^T+\delta) ; t) = 
\end{equation}
\[ =\sum_{\sigma \in S_{N}} \text{sgn}(\sigma) \, P(\sigma(\mu^T+\delta) - (\l^T+\delta) ; t) = K_{\mu^T, \l^T }(t) . \]

\end{proof}

The graded dimension of \( C^\l_\nu\) is also given by Brundan \cite{brun1}:
\begin{equation}  \label{bqq}
\sum_{r \geq 0}
 \dim_{\mathbb{C}} \,  C^\l_\nu \{ 2r \} \, t^r = t^{d_\nu^\l /2} \sum_{\tau \in \mathcal{P^+}(n, N)} K_{\tau, \nu} \, K_{ \tau^T, \l^T}  (t^{-1}), 
 \end{equation}
where \(K_{\tau, \nu} = K_{\tau, \nu}(1)\) denotes the Kostka number. The Kostka-Foulkes polynomial \(K_{\tau, \nu} (t)\)  for a composition \(\nu\) is defined to be \(K_{\tau, \hat{\nu}} (t) \), where \(\hat{\nu}\) is the partition obtained by permuting the entries of \(\nu\).

 By \cite{lcurr} and \cite{vas}, the center \(Z(R^{\lb}) \) of the cyclotomic KLR algebra \(R^{\lb} \)  is isomorphic to the dual local Weyl module \(W_{}^*(\lb)\) . 
 The next theorem is proved independently by Shan-Vasserot-Varagnolo \cite{vas} and  Webster \cite{bens}, section 3.
\begin{Theorem} \label{main} \cite{bens}
The current algebra module \(\bigoplus_\nu C^\l_\nu\) is isomorphic to the dual local Weyl module \(W^*(\lb)\).
\end{Theorem}
\begin{proof} We give a proof of existence of an injective map \(\phi\colon \bigoplus_\nu C^\l_\nu \to W^*(\lb)\). The module \(\bigoplus_\nu C^\l_\nu\) is cocyclic, cogenerated by the identity element \(1_{\l}\) of the highest weight space \( C^\l_\l \simeq \mathbb{C}\). This means that for any \(v \in \bigoplus_\nu C^\l_\nu\), there exist \(u \in \mathbf{U}(\mathfrak{sl}_n[t])\) such that \(uv = 1_\l\). Moreover,  \(\bigoplus_\nu C^\l_\nu\) is a finite-dimensional highest weight module with the highest weight \(\lb\). The remaining part of the proof follows from the dualization of the Theorem \ref{unip} -- the universality property of local Weyl modules. More explicitly, any cocyclic finite-dimensional highest weight module of highest weight \(\lb\)  injects into a dual local Weyl module \(W^*(\lb)\). The injective, degree zero \(\mathbf{U}(\mathfrak{sl}_n[t])\)-invariant map \(\phi\), which is uniquely defined \(\phi (1_{\lb}) = \delta_{v_{\lb}} \).

In order to prove surjectivity, it must shown that the co-kernel \(W^*(\lb) \big/ \phi (\bigoplus_\nu C^\l_\nu)\) is zero.

Socle of a \(\mathbf{U}(\mathfrak{sl}_n[t])\)-module \(M\) is the largest semisimple submodule of \(M\). Cosocle of \(M\) is the largest semisimple quotient of \(M\). The dual of the socle of \(M\) is isomorphic to the cosocle of the dual \(M^*\). Kodera-Naoi \cite{kod} show that the socle of the local Weyl module \(W(\lb)\) is the degree \({d_\nu^\l}\) homogeneous piece, and it is isomorphic to the simple \(\mathfrak{sl}_n\)-module \(V(\lb_{\min})\) where \(\lb_\min\) is the unique minimal dominant weight amongst those \(\leq \lb\). 

Dually,  the cosocle of \(W(\lb)\) is \(V(\lb_{\min})\), consisting of degree zero elements. By the multiplicity formula \eqref{chr}, the cosocle of \(\bigoplus_\nu C^\l_\nu\)  is also isomorphic to \(V(\lb_{\min})\) in degree zero, generated by the unit of the algebra \(C^\l_{\l_\min}\) over \(\mf{sl}_n\). Hence, \(W^*(\lb) \big/ \phi (\bigoplus_\nu C^\l_\nu)\) has no degree zero elements.

The dual local Weyl module \(W^*(\lb)\) has a finite length, therefore every quotient of  \(W^*(\lb)\) contains a simple submodule, which lies in the cosocle, by the definition of the cosocle. Since the cosocle of \(W^*(\lb)\) is simple, it is contained in the co-kernel \(W^*(\lb) \big/ \phi (\bigoplus_\nu C^\l_\nu)\). However, \(W^*(\lb) \big/ \phi (\bigoplus_\nu C^\l_\nu)\) has no degree zero elements, and hence is zero.
\end{proof}

 Shan-Vasserot-Varagnolo \cite{vas} and  Webster \cite{bens} also prove that there is a graded algebra isomorphism between the \( C^\l_\nu\) and the center \(Z \left( R^{\lb}( \lb -\nub )  \right) \) of the cyclotomic KLR algebra \(R^{\lb}( \lb -\nub )\) of rank \(\lb - \nub\). Theorem \ref{main} and graded dimension formula in \cite{brun1} gives the grade dimension of the center  \(Z \left( R^{\lb}( \lb -\nub )  \right) \):
\begin{equation}  \label{lqq}
\sum_{r \geq 0}
 \dim_\mathbb{C} \, Z \left( R^{\lb}( \lb -\nub )  \right) \{2r \} \, t^r = t^{d_\nu^\l/2}  \sum_{\tau \in \mathcal{P^+}(n, N)} K_{\tau, \nu} \, K_{ \tau^T, \l^T}  (t^{-1}).
 \end{equation}
 
 By Proposition \ref{hum}, the right hand side of the equation  \eqref{lqq} depend on \(\lb\) and \(\nub\) rather than actual \(\l\) and \(\nu\). \\

\begin{Theorem}
The graded character of the local Weyl module is given by the formula
 \begin{equation} \label{wgchr2}
\text{ch}_t \,  W(\lb) = \sum_{\tau \in \mathcal{P^+}(n, N)} K_{ \tau^T, \l^T}  (t^{}) \, \text{ch}_{} \, V(\bar{\tau}) ,
\end{equation}
and the graded dimension of the weight space \(\nub\) is given by
\begin{equation}  \label{wlqq}
 \dim_\mathbb{C} \, W_{\nub}(\lb) = \sum_{r \geq 0}
 \dim_\mathbb{C} \, W_{\nub}(\lb) \{ 2r \} \, t^r = \sum_{\tau \in \mathcal{P^+}(n, N)} K_{\tau, \nu} \, K_{ \tau^T, \l^T}  (t^{}).
 \end{equation}
\end{Theorem} 
\begin{proof}
Combining the muliplicity formula \eqref{chr} with the Theorem \ref{main}, we have
\begin{equation} \label{wchr} 
\sum_{r \geq 0} 
 \left[ {\textstyle \bigoplus_\nu} W_{\nub}^*(\lb) \{ d^{\l }_\nu - 2r \} : V(\bar{\tau}) \right] t^r =K_{\tau^T, \l^T}(t),
\end{equation}
Since \(W_{\nub}^*(\lb)  \{ d^{\l }_\nu - 2r \} = (W_{\nub} (\lb) \{ 2r \} )^*\) and \(V(\bar{\tau}) \simeq V^*(\bar{\tau})\), the multiplicity formula for the local Weyl module is 
\begin{equation} \label{wwchr} 
\sum_{r \geq 0} 
 \left[ {\textstyle \bigoplus_\nu} W_{\nub}(\lb) \{ 2r \} : V(\bar{\tau}) \right] t^r =K_{\tau^T,\l^T}(t),
\end{equation}
which implies \eqref{wgchr2}.
The graded dimension formula follows similarly from the equation \eqref{bqq}. \qedhere
\end{proof}

\begin{Remark}
Let \(\l \in \mathcal{P^+}(n, N) \) and \(W(\lb)\) be the local Weyl module with the highest weight \(\lb\). If we set \(t=1\) in the formula \eqref{wgchr2}, we get
\begin{equation}\label{ung}
\text{ch}_t \,  W(\lb) \, |_{t=1}= \sum_{\tau \in \mathcal{P^+}(n, N)} K_{ \tau^T, \l^T}  (1^{}) \, \text{ch}_{} \, V(\bar{\tau})= \text{ch}\, \left( {\textstyle\bigwedge}^{\l^T_1} \mathbb{C}^n \otimes {\textstyle\bigwedge}^{\l_2^T}\mathbb{C}^n \otimes \cdots \otimes {\textstyle\bigwedge}^{\l_p^T} \mathbb{C}^n\right),
\end{equation}
where the \(\l^T = (\l_1^T, \l_2^T, \dots, \l^T_p)\) is the transpose of \(\l\). The equation \eqref{ung} also follows from Chari-Loktev \cite{lamp}, where they prove that \(W(\lb)\) is isomorphic to a certain graded tensor product of fundamental representations, called the fusion product, where recover the usual tensor product

\begin{equation}
 {\textstyle\bigwedge}^{\l^T_1} \mathbb{C}^n \otimes {\textstyle\bigwedge}^{\l_2^T} \mathbb{C}^n \otimes \cdots \otimes {\textstyle\bigwedge}^{\l_p^T} \mathbb{C}^n=\bigoplus_{\tau \in \mathcal{P^+}(n, N)} \left(V(\bar{\tau}) \right)^{\oplus K_{ \tau^T, \l^T} }
\end{equation}
if we set \(t=1\). 
\end{Remark}
A computation shows that if \(\l_0= (N, 0, \dots, 0)\), the graded dimension of \(C^{\l_0}_\nu \) is given by the quantum multinomial coefficient
\[ \binom{N}{\nu}_{t} = \binom{N}{\nu_1,\dots, \nu_n}_{t}= \frac{[N]_t!}{[\nu_1]_t! \cdots  [\nu_n]_t!},\]
 where \([a]_t!\) denotes the \(t\)-factorial 
 \[\frac{(1-t) (1-t^2) \cdots (1-t^a)}{(1-t)^a} = (1+t)(1+t+t^2) \cdots (1+t +t^2+ \cdots +t^{a-1}). \] 
 
 \begin{Example}
  Let \(N=9\), \(n=4\), \(\l= (5,2,1,1)\) and \(\nu =(3, 1,2, 3)\).
 We compute the graded dimension of \(Z \left( R^{\lb}( \lb -\nub ) \right)  \). We have \(\l^T = (4,2,1,1,1)\),  \(d_\nu^\l= 31-23=8\).
 \begin{equation} 
\sum_{r \geq 0}
 \dim \, Z \left( R^{\lb}( \lb -\nub ) \right) \{ 2r \}  \, t^r = t^4  \sum_{\tau \in \mathcal{P^+}(4, 9)} K_{\tau, (3,1,2,3)} \, \cdot K_{ \tau^T, (4,2,1,1,1)}  (t^{-1}) =
\end{equation}
\[ = t^4 [ K_{(3,3,2,1), (3,1,2,3)} \, \cdot K_{ (4,3,2,0,0) , (4,2,1,1,1)}  (t^{-1}) + K_{(5,2,1,1), (3,1,2,3)} \, \cdot K_{ (4,2,1,1,1), (4,2,1,1,1)}  (t^{-1}) +\]
\[+ K_{(4,3,1,1), (3,1,2,3)} \, \cdot K_{ (4,2,2,1,0), (4,2,1,1,1)}  (t^{-1}) + K_{(4,2,2,1), (3,1,2,3)} \, \cdot K_{ (4,3,1,1,0), (4,2,1,1,1)}  (t^{-1}) ]=\] 
\[ = t^4 ( t^{-2}+t^{-3}+t^{-4} + 2+ 2(t^{-2}+t^{-1})+(t^{-3}+t^{-2}+t^{-1}) ) =2t^4+ 3t^3+4t^2+2t+1.\]
It is easy to check that we get the same graded dimension if we choose \(\l'= (4,1,0,0)\) and \(\nu' =(2, 0,1, 2)\). In both cases, \(\lb =\lb'=(3,1,0)\) and \(\nub = \nub' = (2, -1, -1)\).
 \end{Example}

  \begin{Example}
  Let \(N=5\), \(n=3\), \(\l= (5,0,0)\) and \(\nu =(3, 1,1)\).
The graded dimension of the center \(Z \left( R^{\lb}( \lb -\nub ) \right)  \) is computed as follows. We have \(\l^T = (1,1,1,1,1)\),  \(d_\nu^\l= 14\).
 \begin{equation} 
\sum_{r \geq 0}
 \dim_{\mathbb{C}} \, Z \left( R^{\lb}( \lb -\nub ) \right)  \{ 2r \} \, t^r = t^{7}  \sum_{\tau \in \mathcal{P^+}(3, 5)} K_{\tau, (3,1,1)} \, \cdot K_{ \tau^T, (1,1,1,1,1)}  (t^{-1}) =
\end{equation}
\[ = t^{7} ( K_{(5,0,0), (3, 1,1)} \, \cdot K_{ (1,1,1,1,1) , (1,1,1,1,1)}  (t^{-1}) + K_{(4,1,0), (3, 1,1)} \, \cdot K_{ (2,1,1,1), (1,1,1,1,1)}  (t^{-1}) +\]
\[+ K_{(3,2,0), (3, 1,1)} \, \cdot K_{ (2,2,1,0,0), (1,1,1,1,1)}  (t^{-1}) + K_{(3,1,1), (3, 1,1)} \, \cdot K_{ (3,1,1,0,0), (1,1,1,1,1)}  (t^{-1}) )=\] 
\[ = t^{7} ( 1+ 2(t^{-1}+t^{-2}+t^{-3}+t^{-4})+(t^{-2}+t^{-3}+t^{-4}+t^{-5}+t^{-6})+(t^{-3}+t^{-4}+2t^{-5}+t^{-6}+t^{-7}) )=\]
\[=t^{7}+2t^{6}+3t^{5}+4t^{4}+4t^{3} + 3t^{2}+2t^{1}+1=\binom{5}{3,1,1}_{t}.\]
\end{Example}


\noindent Institute of Mathematics, University of Zurich, \\ Winterthurerstr. 190, 8057 Zurich, Switzerland, \\ E-mail: zaur.guliyev@math.uzh.ch

\end{document}